\documentclass{amsart}          

\usepackage{amsmath,amsthm}     
\usepackage{amssymb,bbm}            
\usepackage{euscript}           
\usepackage{array,enumerate,calc,graphicx,xcolor}
\usepackage[matrix,arrow,curve,frame]{xy}    
\usepackage[colorlinks]{hyperref}
\usepackage{stmaryrd}
\usepackage{tikz-cd}

\makeatletter
\newcommand*\bigcdot{\mathpalette\bigcdot@{.7}}
\newcommand*\bigcdot@[2]{\mathbin{\vcenter{\hbox{\scalebox{#2}{$\m@th#1\bullet$}}}}}
\makeatother

\usepackage{tikz}
\usetikzlibrary{matrix,arrows,decorations.pathmorphing}

\xymatrixcolsep{1.9pc}                          
\xymatrixrowsep{1.9pc}
\newdir{ >}{{}*!/-5pt/\dir{>}}                  

\newtheorem{thm}[subsection]{Theorem}
\newtheorem{defn}[subsection]{Definition}
\newtheorem{prop}[subsection]{Proposition}

\newtheorem{cor}[subsection]{Corollary}
\newtheorem{lemma}[subsection]{Lemma}

\theoremstyle{definition}  
\newtheorem{example}[subsection]{Example}

\newtheorem{remark}[subsection]{Remark}

  {\end{list}}

\newcommand{\dfn}{\textbf} 

\newcommand{\mdfn}[1]{\dfn{\mathversion{bold}#1}} 

\newcommand{\Smash}             {\wedge}

\newcommand{\Wedge}             {\vee}

\newcommand{\tens}              {\otimes}               

\newcommand{\iso}               {\cong}

\newcommand{\cat}{\EuScript}    
\newcommand{\cA}{{\cat A}}      
\newcommand{\cC}{{\cat C}}
\newcommand{\cD}{{\cat D}}

\newcommand{\cI}{{\cat I}}
\newcommand{\cJ}{{\cat J}}
\newcommand{\cK}{{\cat K}}
\newcommand{\cL}{{\cat L}}

\newcommand{\cO}{{\cat O}}
\newcommand{\cP}{{\cat P}}
\newcommand{\cS}{{\cat S}}

\newcommand{\Mack}{{\cat Mack}}
\newcommand{\Mod}{\text{Mod}}

\newcommand{\Ho}{{\cat Ho}}

\renewcommand{\cA}{{\mathcal A}}
\renewcommand{\cC}{{\mathcal C}}
\renewcommand{\cI}{{\mathcal I}}

\newcommand{\field}[1]  {\mathbb #1} 

\newcommand{\F}         {\field F}
\newcommand{\R}         {\field R}
\newcommand{\M}         {\field M}

\newcommand{\Z}         {\field Z}

\newcommand{\PF}        {\field P}

\DeclareMathOperator*{\Cof}{Cof}
\DeclareMathOperator*{\Ann}{Ann}

\DeclareMathOperator*{\colim}{colim}

\DeclareMathOperator{\Hom}{Hom}
\DeclareMathOperator{\uHom}{\underline{Hom}}
\DeclareMathOperator{\uExt}{\underline{Ext}}
\DeclareMathOperator{\uTor}{\underline{Tor}}

\DeclareMathOperator{\tr}{tr}

\DeclareMathOperator{\Spec}{Spec}
\DeclareMathOperator{\Supp}{Supp}

\newcommand{\ra}{\rightarrow}                   
\newcommand{\lra}{\longrightarrow}              
\newcommand{\llra}[1]{\stackrel{#1}{\lra}}      

\newcommand{\inc}{\hookrightarrow}              
\newcommand{\dbra}{\rightrightarrows}           

\newcommand{\blank}{-}                          
\newcommand{\id}{id}                            
\newcommand{\und}{\underline}

\newcommand{\period}    {{\makebox[0pt][l]{\hspace{2pt} .}}}

\newcommand{\semicolon} {{\makebox[0pt][l]{\hspace{2pt} ;}}}

\newcommand{\bdot}{\bullet}

\newcommand{\he}{\simeq}

\numberwithin{equation}{section}

\newcommand{\Ch}{Ch}

\newcommand{\cR}{{\cat R}}

\newcommand{\Ab}{\cat Ab}

\newcommand{\ev}{ev}

\newenvironment{bsmallmatrix}
  {\left [\begin{smallmatrix}}
  {\end{smallmatrix}\right ] }

\newcommand{\cpdot}{\bigcdot}
\newcommand{\Sth}{S_{\Theta}}
\newcommand{\Sd}{S_{\cpdot}}
\newcommand{\bbox}{\Box}
\newcommand{\tht}{_{\Theta}}
\newcommand{\dt}{_{\cpdot}}
\newcommand{\chk}{\vee}
\newcommand{\op}{{op}}
\newcommand{\MMack}[2]{\xymatrix{
{#1} \ar@(ul,dl)[] \ar@<0.5ex>[r] & {#2}\ar@<0.5ex>[l]}}

\newcommand{\cotens}{\mathcal{F}}
\newcommand{\dcotens}{\mathcal{F}\!\!\!\!\mathcal{F}}
\newcommand{\dbox}{\bbox \!\!\!\!\bbox}  
\newcommand{\bD}{\mathbb{D}}
\newcommand{\MMod}{\!-\!\Mod}

\newcommand{\pars}[1]{\left({#1}\right)} 

\DeclareMathOperator{\tp}{top} 
\DeclareMathOperator{\wt}{wt} 
\DeclareMathOperator{\cowt}{cowt} 
\DeclareMathOperator{\perf}{perf}
\DeclareMathOperator{\Rep}{Rep}
\DeclareMathOperator{\Pic}{Pic}

\newcommand{\undl}{\underline{\smash{\Z/\ell}}}
\newcommand{\undt}{\underline{\smash{\Z/2}}}
\newcommand{\undr}[1]{\underline{\smash{\Z/#1}}}
\newcommand{\bSigma}{\overline{\Sigma}}

\begin{document}

\title{Equivariant $\undl$-modules for the cyclic group $C_2$}

\author{Daniel Dugger}
\author{Christy Hazel}
\author{Clover May}

\address{Department of Mathematics\\University of Oregon\\Eugene, OR
  97403}
\email{ddugger@uoregon.edu}

\address{Department of Mathematics\\UCLA\\ Los Angeles, CA 90095}
\email{chazel@math.ucla.edu}

\address{Department of Mathematical Sciences\\NTNU\\ 7491 Trondheim, Norway}
\email{clover.may@ntnu.no}

\begin{abstract}
For the cyclic group $C_2$ we give a complete description of the
derived category of perfect complexes of modules over the constant
Mackey ring $\undl$, for $\ell$ a prime.  This is fairly simple for
$\ell$ odd, but for $\ell=2$ depends on a new splitting theorem.  As
corollaries of the
splitting theorem we compute the associated Picard group and
the Balmer spectrum for compact objects in the derived category, and
we  obtain a complete classification of finite modules over the
$C_2$-equivariant Eilenberg--MacLane spectrum $H\undt$.
We also use the splitting theorem to give new and illuminating proofs of some
facts about $RO(C_2)$-graded Bredon cohomology, namely Kronholm's
freeness theorem \cite{K, HM} and the structure theorem of C. May
\cite{M}.
\end{abstract}

\maketitle

\tableofcontents

\section{Introduction}

Mackey functors were first defined by Dress and Green in the early
1970s \cite{D1,D2,Gr}.  They
have proven to be important objects in equivariant homotopy theory and
representation theory.  For $G$ a finite
group, the category of $G$-Mackey functors is abelian and symmetric
monoidal via the box product.  We call a monoid $R$ in
this category a {\it Mackey ring\/} (note that commutative Mackey rings are
also called Green functors in the literature).  
From $R$ one obtains an abelian category of $R$-modules, a
corresponding category of chain complexes of $R$-modules, and an
associated derived category $\cD(R)$.

The derived category of a Mackey ring is our main object of study in
this paper, though our motivation comes from topology.  In classical
topology, if $T$ is a ring then $\cD(T)$ is the natural homotopical recipient for
singular homology with $T$-coefficients. In the $G$-equivariant
setting, if $R$ is a Mackey ring then $\cD(R)$ is the analogous
recipient for ordinary Bredon homology with $R$-coefficients.  For this
reason it is natural to investigate $\cD(R)$ and try to understand as
much as we can about it.

This paper specializes to the case where $G=C_2$, the cyclic group of order two, and $R=\undl$, the constant-coefficient Mackey functor for the ring
$\Z/\ell$, with $\ell$ a prime.  In this case we completely describe the
full subcategory $\cD(\undl)_{\perf}$ consisting of perfect
complexes. Here a ``perfect complex'' is a bounded chain complex of
finitely-generated projective modules.  There are only two
indecomposable, finitely-generated projectives, called $H$ and $F$ in
the $\ell=2$ case and $H$ and $S\tht$ in the $\ell>2$ case (see Section~\ref{se:background}).  Topologically, $H$ and $F$ correspond to fixed and free
$C_2$-equivariant cells, respectively, and when $\ell>2$ one has $F\iso
H\oplus S\tht$.  

 When $\ell$ is odd, the category of
$\undl$-modules is semisimple and hence $\cD(\undl)_{\perf}$ is
easy to understand.  Every perfect complex decomposes---as a chain
complex, not just an object in the
derived category---as a direct sum of
suspensions of $H$ and $S\tht$ (regarded as complexes concentrated in
degree 0) as well as ``disks'' on $H$ and $S\tht$ (chain complexes
consisting of exactly one identity map)---see Proposition~\ref{pr:odd-splitting}.
But the case of real interest is $\ell=2$.  Here we are again able to describe all perfect complexes, but the result and techniques used are more complicated.  For this case we define certain families of
complexes $A_k$, $B_r$ (for $k,r\geq 0$),
and $H(n)$ (for $n\in \Z$). These are simple ``linear strands''
consisting of a single $H$ or $F$ in each degree;
see Section~\ref{se:complexes} for the precise definitions.  Our main classification result is the following:

\begin{thm}
\label{th:intro-splitting}
Every perfect complex of $\undt$-modules is isomorphic to a finite
direct sum of suspensions of complexes of type $A_k$, $B_r$, and $H(n)$, together
with suspensions of elementary contractible disk-complexes $\bD(P)$, where $P$ is a
projective and $\bD(P)$ denotes the complex $0\lra P \llra{\id} P\lra 0$.
\end{thm}

We actually prove more than is stated in the above theorem. We
give an algorithm for
taking any perfect complex and splitting it into the above form.
We also prove a similar splitting theorem for bounded below complexes of
finitely-generated modules.  The bounded below case requires two more types of
complexes (allowing $k=\infty$ and $r=\infty$); see Theorem~\ref{th:bd-below-splitting}.

We call each of the special complexes in the splitting theorems
``strands'' because they take a particularly nice form: they are
nonzero for a finite string of degrees, and in each of these degrees
they have a single summand that is either $H$ or $F$. A similar splitting
phenomenon occurs for the classical derived category of abelian
groups: every perfect complex  splits as a direct sum of
shifts of the strands $0 \lra \Z \lra 0$ and $0 \lra \Z \llra{n} \Z
\lra 0$.  In some ways our main argument is similar to what happens in
this classical example, in that we can do concrete change-of-bases to produce the
splitting---but the bookkeeping is more difficult and comes with several
subtleties.  
Note that there is something special about $G=C_2$ here; for
$G=C_p$ with $p$ an odd prime, there are
perfect complexes that cannot decompose into strands in
$\cD(\undr{p})$. We discuss this in Remark \ref{re:strand}.

Returning to $G=C_2$, it follows
from Theorem~\ref{th:intro-splitting} that we can understand
$\cD(\undt)_{\perf}$ by studying all maps between the $A_k$,
$B_r$, and $H(n)$ complexes.  The derived category
$\cD(\undt)$ also has a monoidal product $\bbox$, analogous to the tensor
product on the derived category of an ordinary commutative ring.  We completely calculate all maps and products in $\cD(\undt)_{\perf}$. We also study several other aspects of this derived category such as
the Picard group, the Balmer spectrum, and a certain kind of duality.
We find the Picard group of $\cD(\undt)$ is $\Z^2$ (see
Theorem~\ref{th:Picard}) and the Balmer spectrum for
$\cD(\undt)_{\perf}$ consists of three points, two closed and one generic (see Theorem~\ref{th:Balmer-spectra}).

\begin{remark}
\label{re:alggeom}
The reader will have noticed the different notation of indices for the
$H$-family as opposed to the $A$- and $B$-families.  In the derived category, the $H(n)$
are invertible with $H(n)\bbox H(m)\he H(n+m)$.  The derived category
$\cD(\undt)$ in many ways behaves like the derived category of a
projective curve having exactly two closed points $x$ and $y$.  The
$H(n)$ are analogs of the canonical sheaves $\cO(n)$, and the $A_k$ and $B_k$ are analogs of the $k$th order
thickenings of the structure sheaves for the points $x$ and $y$.  For
further analogies with algebraic geometry see  Remark~\ref{re:M2-ag} and
Theorem~\ref{th:Serre-duality}.
\end{remark}

\medskip

While the main work of this paper is entirely algebraic, our
motivation comes from equivariant homotopy theory.
For every Mackey functor $M$, there is a corresponding
genuine equivariant spectrum $HM$ called the Eilenberg--MacLane (or
Bredon) spectrum.  If $A$ is an abelian group
then we can take $M=\und{A}$, the constant-coefficient Mackey functor,
and if $A$ is a ring then $H\und{A}$ is a ring spectrum.

In the equivariant setting, homotopy and (co)homology are often graded on the representation ring $RO(G)$. For $G=C_2$, this gives us bigraded homotopy and (co)homology.  Unlike the classical setting, the $RO(G)$-graded homotopy of $H\und{A}$ is typically not concentrated in a single degree.  This can make computations significantly more challenging.  In general, the homotopy category of $H\und{A}$-modules is not well understood.  However, we can transform this into a problem in algebra.

The homotopy category of $H\und{A}$-modules is equivalent to the
derived category $\cD(\und{A})$ (for $A=\Z$ this was explained in
\cite{Z}).  We use this equivalence and our classification of perfect
complexes in $\cD(\undl)$ to give a classification of finite
$H\undl$-modules.  For $\ell$ odd, every finite $H\undl$-module
decomposes as a wedge of suspensions of $H\undl$.  Here the
suspensions are by (possibly virtual) representation spheres.  Our
description of perfect complexes is not actually required here:  this
decomposition follows from the fact that the $RO(C_2)$-graded homotopy
of $H\undl$ is a graded field, and graded modules over a graded field
decompose nicely.  As before, $\ell = 2$ is the much more interesting
case.  The $RO(C_2)$-graded homotopy of $H\undt$ is a non-Noetherian
ring and has a complicated module theory.

In \cite{M} it was shown that 
if $X$ is a finite $C_2$-CW spectrum then
$H\undt\Smash X$
splits into a wedge of suspensions of $H\undt$ and $H\undt \Smash (S^k_a)_+$, where $S^k_a$ is the $k$-dimensional sphere with the antipodal action and suspensions are by (virtual) representation spheres.
Using the splitting of perfect complexes in $\cD(\undt)$, we prove a generalization.  This generalization requires a third type of $H\undt$-module given by ``cofibers of powers of $\tau$'', where
$\tau$ is a certain familiar element in the
homotopy of $H\undt$ (see Section~\ref{se:distinguish} for a precise definition).

\begin{thm}
\label{th:intro-topsplitting}
Every finite $H\undt$-module is equivalent to a wedge of
$RO(C_2)$-graded suspensions of $H\undt$, $H\undt \Smash (S^k_a)_+$, and $\Cof(\tau^r)$ for various $k \geq 0$ and $r \geq 1$.
\end{thm}

This follows from the splitting theorem by identifying the complexes $H(n)$ with weight $n$ representation spheres (smashed with $H\undt$), the complexes $A_k$ with the antipodal spheres (again, smashed with $H\undt$), and finally the complexes $B_{r-1}$ with $\Cof(\tau^{r})$.

While this splitting theorem might appear a bit asymmetric with the appearance of $\Cof(\tau^r)$, we can reinterpret the splitting by identifying $H\undt \Smash (S^k_a)_+$ with a desuspension of $\Cof(\rho^{k+1})$.  Here $\rho$ is another familiar element in the homotopy of $H\undt$.
Then we observe every finite $H\undt$-module splits as a wedge of suspensions of the unit $H\undt$, $\Cof(\rho^{k})$, and $\Cof(\tau^r)$ for $k, r \geq 1$.

\medskip

The splitting result of Theorem~\ref{th:intro-splitting} has some
other nice
consequences.  One is that it leads to a new proof of the
topological splitting theorem from \cite{M}.  We also use it to give a
new perspective on an important result of Kronholm.
Kronholm \cite{K} observed that for a $C_2$-space built
from a finite number of representation cells (called a finite
$\Rep(C_2)$-complex), the $RO(C_2)$-graded Bredon cohomology is always
free as a module over the Bredon cohomology of a point.  The proof of
this ``freeness theorem'' is subtle, and a small but significant gap
was discovered and repaired by Hogle and C. May \cite{HM}, who also
extended the result to the finite-type case.  We use
Theorem~\ref{th:intro-splitting} to give another proof of this
freeness theorem, which again has some subtleties but also offers some
enlightening perspectives.  In particular, one notable aspect of
Kronholm's freeness theorem is the phenomenon now known as ``Kronholm
shifting'', whereby cells from the original decomposition appear to
shift weights in terms of how they contribute to the cohomology of the space.  The algebraic splitting algorithm underlying
Theorem~\ref{th:intro-splitting} gives a concrete perspective on
these Kronholm shifts.

\begin{remark}
\label{re:wilson}
In a private communication, Dylan Wilson shared a perspective on how the
derived category versions of some of our results can be understood via
the Tate square from equivariant homotopy theory.  We briefly include a rough outline since this might
serve as a useful
guidepost to the reader, though some of the details are
outside the scope of this paper.

The
Tate square for $H\undt$ takes the form
\[ \xymatrix{H\undt \ar[r]\ar[d] & H\undt[\rho^{-1}] \ar[d] \\
H\undt[\tau^{-1}] \ar[r] & H\undt[\tau^{-1},\rho^{-1}].
}
\]
Isomorphism classes of modules over $H\undt$ can be seen to correspond to
isomorphism classes of triples $(M,N,\phi)$ where $M$ is a module over
$H\undt[\tau^{-1}]$, $N$ is a module over $H\undt[\rho^{-1}]$, and
$\phi$ is an isomorphism $M[\rho^{-1}]\ra N[\tau^{-1}]$ (according to Wilson there is an appropriate $\infty$-categorical equivalence here).

The bigraded homotopy rings of the equivariant ring spectra $H\undt[\tau^{-1}]$ and
$H\undt[\rho^{-1}]$ are very simple: they are
$\Z/2[\tau^{\pm 1},\rho]$ and $\Z/2[\tau,\rho^{\pm 1}]$, respectively.
In particular, these are bigraded PIDs and this 
implies that the module theory of these ring spectra is formal---that
is, the module theory is the same as
the (bigraded) module theory of their bigraded homotopy rings (using that bigraded homotopy detects weak equivalences in the
$C_2$-equivariant setting).  By algebra,
the 
graded module theory of these rings is the same as that of the
singly-graded rings $\Z/2[x_1]$ and $\Z/2[x_2]$ where
$x_1=\frac{\rho}{\tau}$ and $x_2=\frac{\tau}{\rho}$.  One is therefore led
to consider triples $(M_1,M_2,\phi)$ where $M_i$ is a graded
$\Z/2[x_i]$-module and $\phi$ is an isomorphism $M_1[x_1^{-1}]\ra
M_2[x_2^{-1}]$.  Such tuples can be seen to break up into sums of three types:
$(\Z/2[x_1]/(x_1^r),0,0)$, $(0,\Z/2[x_2]/(x_2^s),0)$, and the free case
$(\Sigma^{k_1}\Z/2[x_1],\Sigma^{k_2}\Z/2[x_2],1\mapsto x_2^{k_2-k_1})$.
These are the analogs of our $A$-, $B$-, and $H(n)$-familes,
respectively (the third tuple listed above corresponds to
$\Sigma^{k_1}H(k_1-k_2)$).    

One can find shades of the above argument sprinkled amongst our
methods throughout the paper, though our perspective is
more rigid and algebraic.  But in particular we want to accentuate the
viewpoint that the module theory of $H\undt$-modules (or equivalently,
the homological algebra of $\undt$-modules) is almost like
that of a PID.  Note also that the above perspective shows the 
connection between finite-type $H\undt$-modules and coherent sheaves over the projective
line $\PF^1_{\F_2}$, which was foreshadowed in Remark~\ref{re:alggeom}
and also appears later in Section~\ref{se:complexes}.  
\end{remark}

\medskip

\subsection{Organization of the paper}
Section 2 gives background on the category of Mackey functors and its
closed symmetric monoidal structure, and sets up much of our notation.  In Section 3 we focus on
understanding $\undr{n}$-modules, and find that there are only a few
finitely-generated indecomposables.  The main results are in Section
4, which contains a thorough investigation of the derived category of
$\undt$-modules together with our main splitting theorem.  Consequences of this
work are then developed in Sections 5 and 6, whereas Section 7 gives
the proof of the splitting theorem.  Finally, Section 8 uses techniques from the proof of the
splitting theorem to prove an algebraic version of Kronholm's freeness
theorem.

\subsection{Notation and terminology}
If $\cC$ is a category then we write $\cC(A,B)$ for $\Hom_\cC(A,B)$.
The cyclic group with two elements is denoted $C_2$ when it is the
background group acting on spaces, and denoted $\Z/2$ when it plays
the role of a coefficient system.

\subsection{Acknowledgements} Many thanks to the anonymous referee for their careful reading, helpful feedback, and for suggesting the short proof of the freeness theorem via localization. The third author would also like to thank Anna Marie Bohmann, Drew Heard, Mike Hill, Eric Hogle, and Mingcong Zeng for a number of helpful conversations.


\section{Background on Mackey functors}
\label{se:background}

In this section we review the basic definitions and structures on
Mackey functors.  We then introduce the Mackey rings $\undr{n}$.
The main objects of study throughout
the paper will be $\undr{n}$-modules.  More information about Mackey functors and their
homological algebra
can be found in
papers such as \cite{W}, \cite{G}, \cite{Bc}, \cite{BSW}, and \cite{Z}.

\medskip

A {\bf{Mackey functor}} for the group $C_2$ is a diagram of abelian groups
\[ \xymatrix{
*{\phantom{--} M\tht\ } \ar@<0.5ex>[r]^{\ \ p_*}  \ar@(ul,dl)[]_t
& M\dt \ar@<0.5ex>[l]^{\ \ p^*}
}
\]
where the maps satisfy the formulas $tp^*=p^*$, $p_*t=p_*$, $t^2=\id$,
and $p^*p_*=1+t$.  There are other ways of defining Mackey functors
as certain additive functors defined on all finite $C_2$-sets.  This
`coordinate-free' approach is more convenient for some purposes.  If
$\cpdot$ denotes the trivial $C_2$-set with one element and $\Theta$
denotes $C_2$ regarded as a $C_2$-set via left multiplication, then every
$C_2$-set is a disjoint union of copies of $\cpdot$ and $\Theta$.  So
an additive functor $M$ on all $C_2$-sets is completely determined by the
data in the above diagram.

 Write $\Mack_{C_2}$ for the category of
Mackey functors for the group $C_2$, and note that this is an abelian
category.

\begin{remark}
It is often convenient to denote elements of $M\tht$ by $x\tht$,
and elements of $M\dt$ by $x\dt$.  In short, we use subscripts to
remind us what part of the Mackey functor elements live in.
\end{remark}

\subsection{Free Mackey functors}
Consider the evaluation functors $\ev\dt\colon \Mack_{C_2} \ra \Ab$
and $\ev\tht\colon \Mack_{C_2}\ra \Ab$ sending $M\mapsto M\dt$ and
$M\mapsto M\tht$, respectively.  These have left adjoints, which
will be denoted $F\dt$ and $F\tht$, respectively.  If $A$ is an
abelian group it is easy to
compute that $F\dt (A)$ has $(F\dt (A))\tht=A$ and $(F\dt (A))\dt=A\oplus A$,
where $t=\id$, the map $p_*$ is the inclusion of the right factor, and $p^*$ is
the identity on the left factor and multiplication-by-two on the right
factor.  We can use the symbols $p_*$ and $p^*$ as placekeepers and write
\[ F\dt (A):\qquad
\xymatrixcolsep{3pc}
\xymatrix{
 *{\phantom{--} p^*A\ }\ar@(ul,dl)[]_{\id} \ar@<0.5ex>[r]^-{\begin{bsmallmatrix} 0 & 1 \end{bsmallmatrix}}
&  A\oplus p_*p^*(A). \ar@<0.5ex>[l]^-{\begin{bsmallmatrix} 1 \\ 2 \end{bsmallmatrix}}
}
\]

Similarly, $F\tht (A)$ can be computed to have $(F\tht (A))\tht=A\oplus
A$ and $(F\tht (A))\dt=A$, where $t$ switches the two summands, $p_*$ is
the fold map, and $p^*$ is the diagonal.  We can write
\[ F\tht (A):\qquad
\xymatrixcolsep{3pc}\xymatrix{
 *{\phantom{---}A\oplus tA\ } \ar@(ul,dl)[]_{\begin{bsmallmatrix} 0 & 1\\ 1 &
   0\end{bsmallmatrix}} \ar@<0.5ex>[r]^-{\begin{bsmallmatrix} 1 & 1 \end{bsmallmatrix}}
&  p_*A. \ar@<0.5ex>[l]^-{\begin{bsmallmatrix} 1 \\ 1\end{bsmallmatrix}}
}
\]

Of particular importance are the Mackey functors $F\tht(\Z)$ and
$F\dt(\Z)$, which play the role of free objects in our category.
It can be useful to write these in terms of chosen generators $g$:
\[ F\dt(\Z): \qquad
\xymatrixcolsep{3pc}\xymatrix{
  *{\phantom{---}\Z\langle p^*g\dt\rangle\, } \ar@(ul,dl)[]_{\id}
\ar@<0.5ex>[r]^-{\begin{bsmallmatrix} 0 & 1 \end{bsmallmatrix}}
&  \Z\langle g\dt\rangle \oplus \Z\langle p_*p^*g\dt\rangle,
\ar@<0.5ex>[l]^-{\begin{bsmallmatrix} 1 \\ 2 \end{bsmallmatrix}}
}
\]
\[ F\tht(\Z):\qquad
\xymatrixcolsep{3pc}\xymatrix{
 *{\phantom{------.}\Z\langle g\tht\rangle \oplus \Z\langle tg\tht\rangle\,
} \ar@(ul,dl)[]_{\begin{bsmallmatrix} 0 & 1\\ 1 &
   0\end{bsmallmatrix}} \ar@<0.5ex>[r]^-{\begin{bsmallmatrix} 1 & 1 \end{bsmallmatrix}}
&  \Z\langle p_*g\tht\rangle.
\ar@<0.5ex>[l]^-{\begin{bsmallmatrix} 1 \\ 1\end{bsmallmatrix}}
}
\]

\noindent
Note
that giving a map $F\tht(\Z)\ra M$ is equivalent to specifying
the image of $g\tht$ in $M\tht$ and likewise giving a map $F\dt(\Z)\ra M$
is equivalent to specifying the image of $g\dt$ in $M\dt$.

The Mackey functor $F\dt(\Z)$ coincides with what is usually called the
{\it Burnside Mackey
functor\/} and is frequently denoted as $\cA$.

\medskip

\subsection{Constant Mackey functors}
Let $A$ be an abelian group.   There is a Mackey functor $\und{A}$
where $\und{A}\tht=\und{A}\dt=A$,  $p^*=t=\id$, and $p_*$ is
multiplication by $2$.  This is called the \dfn{constant coefficient
  Mackey functor} with value $A$ and has diagram:
  \[ \und{A}: \qquad
  \xymatrix{
  A \ar@<0.5ex>[r]^{2}  \ar@(ul,dl)[]_{\id}
  & A. \ar@<0.5ex>[l]^{\id}
  }
  \]

\subsection{The box product}
Given two Mackey functors $M$ and $N$, their \dfn{box product} $M\bbox N$ is
the Mackey functor given by
\[ (M\bbox N)\tht=M\tht\tens N\tht, \qquad (M\bbox N)\dt =
[(M\dt\tens N\dt)\oplus (M\tht\tens N\tht)]/\sim
\]
where the
equivalence relation will be defined in a moment.  The structure map
$p_*$ will send $M\tht\tens N\tht$ to the image of the right summand in
$(M\bbox N)\dt$, so we denote the element $m\tht\tens n\tht$ in
$(M\bbox N)\dt$ as $p_*(m\tht\tens n\tht)$.
The other structure maps in $M\bbox N$ are given by
\begin{itemize}
\item $t(m\tht\tens n\tht)=t(m\tht) \tens t(n\tht)$,
\item $p^*(m\dt\tens n\dt)=p^*(m\dt)\tens p^*(n\dt)$,
\item $p^*(p_*(m\tht \tens n\tht))=m\tht\tens n\tht + t(m\tht)\tens t(n\tht)$.
\end{itemize}
Finally, the equivalence relation that defines $(M\bbox N)\dt$ is
generated by
\begin{itemize}
\item $p_* (p^*(m\dt) \tens n\tht) = m\dt\tens p_*(n\tht)$,
\item
$p_* (m\tht \tens p^*(n\dt)) = p_*(m\tht)\tens n\dt$, and
\item $p_*( t(m\tht)\tens t(n\tht))=p_*(m\tht\tens n\tht)$.
\end{itemize}

There is a natural unit isomorphism $\cA\bbox M\ra M$ that sends
$g\dt\tens m\dt \mapsto m\dt$ and $p^*g\dt\tens m\tht \mapsto m\tht$.
There is a similar right unital isomorphism  $M\bbox \cA \ra M$.

The associativity isomorphism $(M\bbox N)\bbox Q\iso M\bbox (N\bbox
Q)$ is the evident one.  The twist isomorphism $M\bbox N\iso
N\bbox M$ sends $m\tht\tens n\tht\mapsto n\tht\tens m\tht$ and $m\dt
\tens n\dt \mapsto n\dt\tens m\dt$.  With these structures, $(\Mack_{C_2},
\bbox, \cA)$ is a symmetric monoidal category.

\begin{example}
\label{ex:box-with-F}
As an example we analyze $F\tht(\Z)\bbox F\tht(\Z)$.  On the $\Theta$
side this is the free abelian group with generators $g\tht \tens g\tht$, $g\tht\tens tg\tht$,
$tg\tht\tens g\tht$, and $tg\tht \tens tg\tht$, with the twist $t$
acting diagonally.  On the $\cpdot$ side we have the free abelian
group with generators $p_*(g\tht\tens g\tht)$ and $p_*(g\tht\tens
tg\tht)$.  A quick analysis of the $p^*$ and $p_*$ maps shows this is
isomorphic to $F\tht(\Z)\oplus F\tht(\Z)$, with corresponding generators $g\tht\tens
g\tht$ and $g\tht\tens tg\tht$.

As a generalization of the above, one can check that if $M$ is any
Mackey functor then $M\bbox F\tht(\Z)\iso F\tht(M\tht)$.
\end{example}

\bigskip

\subsection{Rings and modules}

A ring structure on a given Mackey functor $M$ consists of a unit map
$\cA\ra M$ and a multiplication $M\bbox M\ra M$ satisfying the usual
axioms (the commutative version of this structure is also called a Green functor).  This is equivalent to specifying a ring structure on $M\dt$
and a ring structure on $M\tht$ such that
\begin{enumerate}[(I)]
\item $p^*$ and $t$ are maps of rings, and
\item $p_*\colon M\tht \ra M\dt$ is a map of $M\dt$-$M\dt$-bimodules,
where $M\tht$ is an $M\dt$-$M\dt$-bimodule via $p^*$.
\end{enumerate}
Condition (II) is equivalent to the ``projection formulas''
\[ p_*( p^*(m\dt)\cdot m\tht) = m\dt \cdot p_*(m\tht), \qquad
p_*( m\tht\cdot p^*(m\dt)) = p_*(m\tht) \cdot m\dt. \qquad
\]

In particular, note that if $R$ is a ring then the constant Mackey
functor $\und{R}$ is a  Mackey ring in a natural way, by using the
multiplication in $R$ for the multiplication in both the $\Theta$ and
$\cpdot$ components.

If $R$ is a Mackey ring and $M$ is a Mackey functor, then equipping
$M$ with a (left) $R$-module structure is equivalent to giving an
$R\dt$-module structure on $M\dt$ and an $R\tht$-module structure on $M\tht$ such
that
\begin{enumerate}[(i)]
\item $p^*(x\dt \cdot m\dt)=p^*(x\dt)\cdot p^*(m\dt)$
\item $t(x\tht \cdot m\tht)= t(x\tht) \cdot t(m\tht)$
\item $p_*( p^*(x\dt)\cdot m\tht) = x\dt \cdot p_*(m\tht)$
\item $p_*( x\tht\cdot p^*(m\dt))= p_*(x\tht)\cdot m\dt$.
\end{enumerate}

When $R$ is the ring $\Z$ or $\Z/n$, we have the following description of $\und{R}$-modules:

\begin{prop}
\label{pr:constant}
A Mackey functor $M$ admits at most one structure of $\und{\Z}$-module, and
it admits such a structure if and only if $p_*p^*=2$.  Similarly, a
Mackey functor $M$ admits at most one structure of $\undr{n}$-module, and
it admits such a structure if and only if $nM\tht=0$, $nM\dt=0$, and
$p_*p^*=2$.
\end{prop}

\begin{proof}
Because there is exactly one $\Z$-module structure on an abelian
group, conditions (i)--(iii) above are trivial (they all involve
products with a multiple of the identity).  Only condition (iv) has
content and it is the condition that $p_*p^*=2$.  The analysis for
$\undr{n}$-modules is exactly the same.
\end{proof}

\begin{remark}
In the classical literature, $\und{\Z}$-modules were originally called
``cohomological Mackey functors''.  See \cite{W}, for example.
\end{remark}

\begin{remark}
If $R$ is a Mackey ring, $M$ is a right $R$-module, and $N$ is a left
$R$-module, then we define $M\bbox_R N$ to be the coequalizer of the
two evident arrows
\[ M\bbox R \bbox N \dbra M\bbox N.
\]
This is the usual definition that works in any monoidal category.
\end{remark}

We will have special need for
the ``free'' $R$-modules $F\tht^R=R\bbox F\tht(\Z)$ and $F\dt^R=R\bbox F\dt(\Z)$.
Note that these have the adjunction properties
\[ \Hom_R(F\tht^R,M)\iso M\tht, \qquad
\Hom_R(F\dt^R,M)\iso M\dt.
\]
Note as well that $F\dt^R$ is just another name for $R$, since
$F\dt(\Z)=\cA$ is the unit for the box product.  It is best to think
of $F\tht^R$ as the free $R$-module generated by an element on the
$\Theta$ side, and $F\dt^R$ as the free $R$-module generated by an
element on the $\cpdot$ side.  We leave it as an exercise for the
reader to check that $F\tht^R\iso F\tht(R\tht)$.  That is, $F\tht^R$
is isomorphic to the $R$-module
\[ F\tht^R:
\xymatrixcolsep{3pc}\xymatrix{
 *{\phantom{--------}R\tht\langle g\tht\rangle \oplus R\tht\langle tg\tht\rangle\,
} \ar@(ul,dl)[]_{\begin{bsmallmatrix} 0 & 1\\ 1 &
   0\end{bsmallmatrix}} \ar@<0.5ex>[r]^-{\begin{bsmallmatrix} 1 & 1 \end{bsmallmatrix}}
&  R\tht\langle p_*g\tht\rangle
\ar@<0.5ex>[l]^-{\begin{bsmallmatrix} 1 \\ 1\end{bsmallmatrix}}
}
\]
where $R\tht$ acts diagonally on the $\Theta$ side and $R\dt$ acts via
$p^*\colon R\dt\ra R\tht$ on the $\cpdot$ side.

It is easy to check that both $F\tht^R$ and $F\dt^R$ are projective
$R$-modules.  We say that an $R$-module is \dfn{free} if it is a direct sum of copies of
 $F\tht^R$ and $F\dt^R$.  We have the expected notion of free basis:

\begin{defn}
\label{de:basis}
Let $R$ be a Mackey ring and let $M$ be an $R$-module.  Define a \dfn{basis} of $M$
to be a collection
of elements $b_{1}^{\Theta},\dots,b_{m}^{\Theta}\in M\tht$ and
$b^{\cpdot}_{m+1},\dots,b^{\cpdot}_{m+n} \in M\dt$ such that the
induced map
\[ \Biggl ( \bigoplus_{i=1}^m F\tht^R \Biggr)  \oplus \Biggl (
\bigoplus_{j=1}^n F\dt^R\Biggr ) \lra M
\]
is an isomorphism.
\end{defn}

Not every $R$-module has a basis, of course, only the free
modules.

\subsection{The internal hom (cotensor)}

Let $M$ and $N$ be Mackey functors, and write $\Hom(M,N)$ for the set
of maps of Mackey functors from $M$ to $N$.  This is an abelian group
in the natural way.  We will next describe an internal hom
$\uHom(M,N)$, itself a Mackey functor, with the property that
$\uHom(M,N)\dt=\Hom(M,N)$.

Consider the diagram of Mackey functors
\[ \xymatrix{
  *{\phantom{--.}F\tht(\Z)\ } \ar@(ul,dl)[]_{t}
\ar@<0.5ex>[r]^-{p^*}
&  F\dt(\Z)
\ar@<0.5ex>[l]^-{p_*}
}
\]
where here $t$ is the map sending $g\tht\mapsto tg\tht$, $p^*$ sends
$g\tht$ to $p^*(g\dt)$, and $p_*$ sends $g\dt$ to $p_*(g\tht)$.
One readily checks that $p_*\circ p^*=1+t$.
As a caution to the reader, note the maps in the above diagram go in the
opposite direction from the similarly named maps of abelian groups
within a Mackey functor. The names $p^*$ and $p_*$ are
just used to remember how we defined these maps of Mackey functors.

Given our Mackey functors $M$ and $N$, we first apply $(\blank)\bbox
M$ to the above diagram to get
\[ \xymatrix{
  *{\phantom{----.}F\tht(\Z)\bbox M\, } \ar@(ul,dl)[]_{t}
\ar@<0.5ex>[r]^-{p^*}
&  F\dt(\Z)\bbox M
\ar@<0.5ex>[l]^-{p_*}
}
\]
(suppressing $\bbox \id_M$ on the maps) and then we apply $\Hom(\blank,N)$ to get
\[ \xymatrixcolsep{3.5pc}\xymatrix{
  *{\phantom{--------.}\Hom(F\tht(\Z)\bbox M,N)\ } \ar@(ul,dl)[]_{\Hom(t,N)}
\ar@<0.5ex>[r]^-{\Hom(p_*,N)}
&  \Hom(F\dt(\Z)\bbox M,N).
\ar@<0.5ex>[l]^-{\Hom(p^*,N)}
}
\]
This is a Mackey functor, and this is our definition of $\uHom(M,N)$.
Note that $F\dt(\Z)\bbox M\iso M$, and so $\uHom(M,N)\dt\iso \Hom(M,N)$.

\begin{prop}
There are natural adjunctions
\[ \Hom(M,\uHom(N,Q))\iso \Hom(M\bbox
N,Q)
\]
and
\[ \uHom(M,\uHom(N,Q))\iso \uHom(M\bbox N,Q).
\]
\end{prop}

\begin{proof}
This is a standard argument and left to the reader.
\end{proof}

\medskip

Now suppose that $R$ is a Mackey ring.
If $M$ and $N$ are left
$R$-modules we define $\Hom_R(M,N)\subseteq \Hom(M,N)$ to be the
subset consisting of the $R$-linear maps.

\smallskip

If $R$ is commutative we
now define an
$R$-module $\uHom_R(M,N)$.  The categorical definition is to define
this to be the equalizer of
\[ \uHom(M,N)\dbra \uHom(R\bbox M,N)
\]
where one map is induced by $R\bbox M\ra M$ and the other is the
composition
$\uHom(M,N)\ra \uHom(R\bbox M,R\bbox N) \ra \uHom(R\bbox M,N)$.
Alternatively, we can form the diagram
\[ \xymatrixcolsep{3.5pc}\xymatrix{
*{\phantom{----}F\tht^R\bbox_R M\, }
\ar@(ul,dl)[]_{t\bbox \id_M}
\ar@<0.5ex>[r]^-{p^*\bbox \id_M}
&  F\dt^R\bbox_R M
\ar@<0.5ex>[l]^-{p_*\bbox \id_M}
}
\]
and then apply $\Hom_R(\blank,N)$.  The resulting Mackey functor is
$\uHom_R(M,N)$.

The following result is again standard and left to the reader.

\begin{prop}
Suppose $R$ is a commutative Mackey ring.
There are natural adjunctions
\[ \Hom_R(M,\uHom_R(N,Q))\iso \Hom_R(M\bbox_R
N,Q)
\]
and
\[ \uHom_R(M,\uHom_R(N,Q))\iso \uHom_R(M\bbox_R N,Q).
\]
\end{prop}

\begin{prop}
Let $R$ be $\und{\Z}$ or $\undr{n}$.  If $M$ and $N$ are $R$-modules
then the canonical map $M\bbox N \ra
M\bbox_{R} N$ is an isomorphism.
\end{prop}

\begin{proof}
Observe $\Hom_{R}(N,N')=\Hom(N,N')$ for any $R$-module $N'$. This fact and the above adjunctions give us that
\begin{align*}
\Hom(M\bbox N,N')\cong \Hom(M, \uHom(N,N')) &= \Hom_{R}(M, \uHom_{R}(N,N'))\\
& \cong \Hom_{R}(M\bbox_{R}N,N') \\
&= \Hom(M\bbox_{R}N,N').
\end{align*}
We see that $\Hom(M\bbox N,\blank)$ and $\Hom(M\bbox_{R} N,\blank)$ are
the same functor, and the result follows.
\end{proof}

\subsection{Some adjoint functors}
\label{se:adjoint}
Now we specialize to the case where $R$ is a ring and $\und{R}$ is the
associated constant Mackey ring.  An $\und{R}$-module consists of a
Mackey functor $M$ together with $R$-module structures on $M\tht$ and
$M\dt$ having the properties that $p^*$, $p_*$, and $t^*$ are maps of
$R$-modules and $p_*p^*=2$.

The evaluation functor $\ev\tht\colon \und{R}\MMod\ra R[C_2]\MMod$ has
a left adjoint $\cK_L$ and a right adjoint $\cK_R$ given as follows:
\[ \cK_L(M)\tht=M, \quad \cK_L(M)\dt=M/C_2,\quad  \text{$p_*$ is projection},\quad
p^*(\bar{x})=x+t x;
\]
\[ \cK_R(M)\tht=M, \quad \cK_R(M)\dt=M^{C_2},\quad p_*(x)=x+tx, \quad
\text{$p^*$ is inclusion}.
\]
Note that the adjunctions yield a natural transformation $\cK_L\ra
\cK_R$ that is the identity on the $\Theta$-component.  

The next result is an easy exercise:

\begin{prop}
\label{pr:adjoint}
If $2$ is invertible in $R$ then both pairs $(\cK_L,\ev\tht)$ and
$(\ev\tht,\cK_R)$ are adjoint equivalences, and the natural
transformation $\cK_L\ra \cK_R$ is an isomorphism.  
\end{prop}


\section{$\protect\undr{\ell}$-modules}
\label{se:modules}

In this section we investigate the category of $\undl$-modules, where
$\ell$ is a prime.  We give a
simple classification of all finitely-generated modules.  Then we compute all
box products and internal homs, and investigate their
associated derived functors.

\medskip

Recall from Proposition~\ref{pr:constant}
that a $\undr{\ell}$-module is a pair of $\Z/\ell$-vector spaces $V\tht$ and
$V\dt$ together with maps
\[ \xymatrix{
*{\phantom{--} V\tht\ } \ar@<0.5ex>[r]^{\ \ p_*} \ar@(ul,dl)[]_t
& V\dt \ar@<0.5ex>[l]^{\ \ p^*}}
\]
such that $p^*p_*=1+t$ and $p_*p^*=2$.  Unsurprisingly, the structure of
the category $\undr{\ell}\MMod$ in the $\ell=2$ case turns out to be very different from the
$\ell\neq 2$ case.  We investigate both cases in detail below.
A $\undr{\ell}$-module $M$ will be called \dfn{finitely-generated} if both
$M\tht$ and $M\dt$ are finite-dimensional vector spaces over $\Z/\ell$.

\subsection{Duality}

If $V$ is a $\Z/\ell$-vector space, let $V^\chk$ denote the dual vector
space.  If $M$ is a $\undl$-module, write $M^\op$ for the result of
applying $(\blank)^\chk$ objectwise to $M$.  That is,
\[ (M^\op)\tht=(M\tht)^\chk \qquad\text{and}\qquad
(M^\op)\dt=(M\dt)^\chk,
\]
with the dual structure maps $t_{M^\op}=(t_M)^\chk$, $p^*_{M^\op}=(p_*^M)^\chk$,
$p_*^{M^\op}=(p^*_M)^\chk$.
The contravariant functor $(\blank)^\op$ is an
anti-equivalence when restricted to the finitely-generated
$\undr{\ell}$-modules.  The following useful result is a consequence:

\begin{prop}
\label{pr:pr=>inj}
If $P$ is a finitely-generated projective $\undr{\ell}$-module, then
$P^\op$ is an injective $\undr{\ell}$-module.
\end{prop}

\begin{proof}
If $R$ is a Mackey ring then an $R$-module $M$ is injective if and
only if it satisfies Baer's criterion.  In this setting, that means
$M$ has the extension property with respect to submodule inclusions $I\inc
F\dt^R$ and $J\inc F\tht^R$ (cf. \cite[Theorem 2.4.1]{R}, but the
proof is readily adapted from the classical proof of Baer's
criterion as in \cite[2.3.1]{W}).  When $R=\undr{\ell}$, the modules $F\dt^R$
and $F\tht^R$ are finitely-generated, and so $I$ and $J$ will be as
well.  The result then follows immediately by applying duality.
\end{proof}

\subsection{Classification}
We begin by introducing some basic $\undl$-modules.  Define the two
basic ``free module'' functors $F\tht^{\undl}$, $F\dt^{\undl}\colon
\Z/\ell\MMod\ra \undl\MMod$ by
\[
F\tht^{\undl}(V)=\undl\,\bbox\, F\tht(V), \qquad
F\dt^{\undl}(V)=\undl\,\bbox\, F\dt(V).
\] We will shorten these to $F\tht^\ell$
and $F\dt^\ell$, and also abbreviate $F\tht^{\ell}(\Z/\ell)=F$ and
$F\dt^{\ell}(\Z/\ell)=H$.
Once again, these are

\[ H:
\xymatrix{
 *{\phantom{--}\Z/\ell\ } \ar@(ul,dl)[]_{\id} \ar@<0.5ex>[r]^-{2}
& \Z/\ell \ar@<0.5ex>[l]^-{\id}
}\qquad\text{and}\qquad
 F:
\xymatrix{
 *{\phantom{----.}\Z/\ell \oplus \Z/\ell\ }
\ar@(ul,dl)[]_{\begin{bsmallmatrix} 0 & 1\\ 1 &
   0\end{bsmallmatrix}} \ar@<0.5ex>[r]^-{\begin{bsmallmatrix} 1 & 1 \end{bsmallmatrix}}
&  \Z/\ell. \ar@<0.5ex>[l]^-{\begin{bsmallmatrix} 1 \\ 1\end{bsmallmatrix}}
}
\]
If $M$ is a $\undl$-module then giving a
 $\undl$-module map $H \ra M$ is equivalent to
specifying an element of $M\dt$, and giving a $\undl$-module map $F\ra M$
is equivalent to specifying an element in $M\tht$.

Consider the duals of the basic free objects. We have
\[
 H^{op}:
\xymatrix{
*{\phantom{--}\Z/\ell~}\ar@(ul,dl)[]_{\id} \ar@<0.5ex>[r]^-{\id}
& \Z/\ell, \ar@<0.5ex>[l]^-{2}
}
\]
while $F^{op}\cong F$. The modules $F$ and $H$ are projective, and consequently their duals
$F$ and $H^{\op}$ are injective by Proposition~\ref{pr:pr=>inj}.
In particular, $F$ is both
projective and injective.

Let $V$ be a $\Z/\ell$-vector space. We will also need the Mackey
functor
\[ \Sth(V): \qquad {\MMack{V}{0}}
\]
where $t$ is multiplication by $-1$.  Using the functors $\cK_L$ and
$\cK_R$ from Section~\ref{se:adjoint}, observe that  $\Sth(V)$ is $\cK_L$
(or $\cK_R$) applied to the sign representation.  For convenience we will
abbreviate $S\tht(\Z/\ell)$ as $S\tht$.  

\subsection{Classification for \mdfn{$\ell$} odd}
The classification of $\undl$-modules for $\ell$ odd is
straightforward: by Proposition~\ref{pr:adjoint} the category of
$\undl$-modules is equivalent to the category of $\Z/\ell[C_2]$-modules.
The latter is semi-simple 
with two irreducibles: $\Z/\ell$ with the trivial action and $\Z/\ell$ with
the sign action.  Applying the functor $\cK_R$ to these yields $H$ and
$S\tht$.  Thus, we obtain the following:

\begin{prop}\label{pr:odd-splitting}
Assume $\ell$ is an odd prime.  Then every $\undl$-module is a
direct sum of copies of $H$ and $S\tht$.  Moreover, every
$\undl$-module is projective and (consequently) every short exact
sequence splits.
\end{prop}

Note that when $\ell$ is odd, we have $H\iso H^\op$ and $F\iso H\oplus
\Sth$.  Note also that if $M$ is
finitely-generated then $M\iso H^a \oplus S\tht^b$ where
$a=\dim_{\Z/\ell}(M\dt)$ and
$b=\dim_{\Z/\ell}(M\tht)-\dim_{\Z/\ell}(M\dt)$, since this can be
detected using the above equivalence of categories.

Using that $\ev\tht$ is strong monoidal lets us easily compute box
products in $\undl$-modules.  
For completeness we record the following 
nontrivial box
products and internal homs (all other computations follow from $H$
being the unit).  From here on, we let $\uHom$ and $\bbox$ denote $\uHom_{\undl}$ and $\bbox_{\undl}$.

\begin{prop}
\label{pr:odd-stuff}
Assume $\ell$ is odd.  Then we have
\[ S\tht\bbox S\tht\iso H, \qquad
\uHom(S\tht,S\tht)\iso H, \qquad
\uHom(S\tht,H)\iso S\tht.
\]
\end{prop}

\begin{proof}
Routine.
\end{proof}

\subsection{Classification for \mdfn{$\ell=2$}}
For the bulk of this paper we focus on the case $\ell=2$.  Note that
when working mod $2$ we can view $t$ as the identity in $S\tht$.  We
also pick up a new Mackey functor in this case: if $V$ is a
$\Z/2$-vector space then we have
\[ \Sd(V): \qquad  {\MMack{0}{V}} \phantom{y_j}
\]
As usual, we abbreviate
$\Sd(\Z/2)$ as $\Sd$.

The following result gives a complete classification for
finitely-generated $\undt$-modules:

\begin{prop}
\label{pr:Hmod-classify}
Every finitely-generated $\undt$-module $M$ is isomorphic to a direct sum
of the Mackey functors $H$, $H^{\op}$, $F$, $\Sd$, and $\Sth$.  The
number of summands of each type is uniquely determined and is given
by the following formulas.  Let $f$ be the rank of $1+t\colon M\tht\ra
M\tht$.  Then
\begin{itemize}
\item The number of $F$ summands equals $f$;
\item The number of $H^\op$ summands equals $\dim M\tht-f-\dim(\ker
p_*)$;
\item The number of $H$ summands equals $\dim M\dt-\dim(\ker p^*)-f$;
\item The number of $\Sd$ summands equals $\dim(\ker p^*)-\dim M\tht+f+\dim (\ker p_*)$;
\item The number of $\Sth$ summands equals
$\dim(\ker
p_*)-\dim M\dt+\dim(\ker p^*)$.
\end{itemize}
\end{prop}

Note that the above decomposition of a module is not canonical.  Also,
we are not saying that the category of $\undt$-modules is semisimple: for
example, there is an evident exact sequence $0\ra \Sth \ra H \ra
\Sd\ra 0$ but $H\ncong \Sth\oplus \Sd$.

\begin{proof}[Proof of Proposition~\ref{pr:Hmod-classify}]
Let $M$ be a finitely-generated $\undt$-module.  If the map $t\colon M\tht\ra
M\tht$ is not the identity, pick $x\in M\tht$ such that $tx\neq
x$.  Since $p^*(p_*x)=x+tx\neq 0$, it follows that $p_*x\neq 0$.
Then the map $F\ra M$ sending the generator $g\tht$ to $x$ is an
injection.  Since $F$ is injective, we have $M\iso F\oplus M'$ for
some $M'$.  Proceeding inductively to keep splitting off copies of
$F$, we reduce to the case where $t\colon M\tht\ra M\tht$ is the
identity.

If $p_*$ is nonzero, pick $x\in M\tht$ such that $p_*(x)\neq 0$.
Then $p^*(p_*x)=x+tx=x+x=0$, and so we get an injection of $\undt$-modules
$H^\op\inc M$.  Again, $H^\op$ is an injective $\undt$-module and so we
split $M\iso H^\op\oplus M'$ for some $M'$.  Continuing in this way,
we reduce to the case where $t=\id$ and $p_*=0$.

Choose a vector space splitting $M\dt=(\ker p^*)\oplus U$.  Likewise,
choose a vector space slitting $M\tht=p^*(U)\oplus V$.  One readily
checks that $M\iso \Sth(V)\oplus \Sd(\ker p^*)\oplus H^n$ where $n=\dim
U$.

The calculation of the number of summands of each type simply comes
from going back through the above proof and counting.
\end{proof}

The only indecomposable objects are $F$, $H$, $H^{\op}$, $\Sth$, and $\Sd$.
We next turn to understanding maps between these.  Recall that maps
$F\ra M$ are in bijective correspondence with elements of $M\tht$, and
maps $H\ra M$ are in bijective correspondence with elements of $M\dt$.
Routine
calculations show that, except for maps from $F$ to itself, there is
at most one nonzero map between any two of the indecomposables.

We will have particular need for understanding maps between the
projective objects $F$ and $H$.
Maps $F\ra H$ are determined by the image of
the generator $g\tht$, and since $H\tht=\Z/2$ there are exactly two
of these: the zero map and the map that sends $g\tht\mapsto
p^*(g\dt)$.  We denote the nonzero map as $p^*$.  Similarly, there is
exactly one nonzero map $H\ra F$, sending $g\dt\mapsto p_*(g\tht)$; we
call this map $p_*$.  The only nonzero map $H\ra H$ is the identity,
and there are three nonzero maps $F\ra F$: $1$, $t$, and $1+t$ (where
$t$ is the map sending $g\tht\mapsto tg\tht$).  As a short summary,
the following diagram depicts all of these nonzero maps:

\[\xymatrix{
F\ar@(ul,dl)[]_{1,t,1+t}
\ar@<0.5ex>[r]^{p^*} & H \ar@<0.5ex>[l]^{p_*} \ar@(ur,dr)[]^{1}
}
\]
Observe that these satisfy the expected relations, e.g.
\[ p^*t=p^*, \ \ tp_*=p_*,\ \  p_*p^*=1+t, \ \ p^*p_*=0,\ \   t^2=1.
\]
Again, a word of caution to the reader: the maps $p^*$ and $p_*$ in the diagram go in the opposite direction of those in a Mackey functor.
Since there is only one nonzero map $F\ra H$ and only one nonzero map
$H\ra F$, we sometimes just
denote both by $p$ since the precise map is clear from context.

For future reference we record how all these maps look on the $\Theta$
and $\cpdot$ sides, with respect to our standard bases for $F$ and $H$:

\vspace{0.2in}

\begin{center}
\bgroup
\def\arraystretch{1.5}
\begin{tabular}{|c||c|c||c|c|}
\hline
& $t$ & $1+t$ & $p^*$ & $p_*$\\
\hline
\hline
$\Theta$ & $\begin{bsmallmatrix} 0 & 1\\ 1 & 0\end{bsmallmatrix}$
&
$\begin{bsmallmatrix} 1 & 1\\ 1 & 1\end{bsmallmatrix}$
& $\begin{bsmallmatrix} 1 & 1\end{bsmallmatrix}$
& $\begin{bsmallmatrix} 1 \\ 1\end{bsmallmatrix}$
\\
\hline
$\cpdot$ & $1$ & $0$ & $0$ & $1$\\
\hline
\end{tabular}
\egroup
\end{center}

\vspace{0.1in}

\subsection{Homological algebra of $\undt$-modules}
\label{se:hom-alg}
We now investigate the homological algebra of $\undt$-modules, making
use of the classification given in the previous section. We have the
following projective resolutions of the nonfree indecomposable modules:
\[ \xymatrixcolsep{1.6pc}\xymatrix{
0\ar[r] & H \ar[r] & F \ar[r] & \Sth\ar[r] & 0, &
0 \ar[r] & H \ar[r] & F \ar[r] & H \ar[r] & \Sd \ar[r] & 0
}
\]
and
\[
\xymatrix{
0\ar[r] & H \ar[r] & F \ar[r]^{1+t} & F \ar[r] \ar[r] & H^{op}\ar[r] & 0.
}
\]
We also have short exact sequences
\[ \xymatrix{
0\ar[r] & \Sth \ar[r] & H \ar[r] & \Sd \ar[r] & 0,
&
0\ar[r] & \Sd \ar[r] & H^{op} \ar[r] & \Sth \ar[r] & 0.
}
\]
The above sequences show that $\Sth$, $\Sd$, and $H^\op$ all have projective
resolutions with at most three nonzero terms.  The following is an immediate
consequence:

\begin{prop}
For any two $\undt$-modules $A$ and $B$ where $A$ is finitely-generated,
we have
$\uExt^i(A,B)=0$ for all $i\geq 3$.
\end{prop}

\vspace{0.1in}

The following proposition lists the results of several routine
calculations in the homological algebra of $\undt$-modules.  Recall that we use $\uHom$ and $\bbox$ to mean $\uHom_{\undt}$ and $\bbox_{\undt}$.

\begin{prop}\label{pr:box_table}
Let $A$ and $B$ be finitely-generated $\undt$-modules.  Then $\uHom(A,B)$
and $A\bbox B$ are given by the following tables:

\vspace{0.2in}

\begin{minipage}{2.5in}
$\uHom(A,B)$:

\vspace{0.1in}
\renewcommand{\arraystretch}{1.2}
\begin{tabular}{|c||c|c|c|c|c|}
\hline
$A\backslash B$ & $H$ & $F$ & $H^{op}$ & $\Sd$ & $\Sth$ \\
\hline\hline
$H$ & $H$ & $F$ & $H^{op}$ & $\Sd$ & $\Sth$ \\
\hline
$F$ & $F$ & $F^2$ & $F$ & $0$ & $F$ \\
\hline
$H^{op}$ & $H$ & $F$ & $H$ & $0$ & $H$ \\
\hline
$\Sd$ & $0$ & $0$ & $\Sd$ & $\Sd$  & $0$ \\
\hline
$\Sth$ & $H$ & $F$ & $\Sth$ & $0$ & $H$   \\
\hline
\end{tabular}
\end{minipage}
%
%
\begin{minipage}{2in}
$A\bbox B$:

\vspace{0.1in}
\renewcommand{\arraystretch}{1.2}
\begin{tabular}{|c||c|c|c|c|c|}
\hline
$A\bbox B$ & $H$ & $F$ & $H^{op}$ & $\Sd$ & $\Sth$ \\
\hline\hline
$H$ & $H$ & $F$ & $H^{op}$ & $\Sd$ & $\Sth$ \\
\hline
$F$ &  & $F^2$ & $F$ & $0$ & $F$ \\
\hline
$H^{op}$ &  &  & $H^{op}$ & $0$ & $H^{op}$ \\
\hline
$\Sd$ &  &  &  & $\Sd$  & $0$ \\
\hline
$\Sth$ &  &  &  &  & $H^{op}$   \\
\hline
\end{tabular}
\end{minipage}

\vspace{0.2in}

\noindent
For the cases where $A$ is not projective and $B$ is not injective,
the Mackey functors $\uExt^i(A,B)$ are as follows:

\vspace{0.2in}

$\uExt^i(A,B)$:

\vspace{0.1in}

\renewcommand{\arraystretch}{1.2}
\begin{tabular}{|c||c|c|c|c|c|c|c|c|}
\hline
$A$,$B$ & $H^{op}$,$H$ & $H^{op}$, $\Sth$ & $H^{op}$, $\Sd$ &
 $\Sth$, $H$ & $\Sth$, $\Sth$ & $\Sth$, $\Sd$
& $\Sd$, $\Sth$ & $\Sd$, $\Sd$ \\ \hline
2& $\Sd$ & $0$ & $\Sd$ & $0$ & $0$ & $0$ &  $\Sd$ & $\Sd$ \\ \hline
1 & $\Sd$ & $\Sd$ & $0$ & $\Sd$ & $0$ & $\Sd$  & $\Sd$ & $0$
\\ \hline
0 & $H$ & $H$ & $0$ & $H$ & $H$ & $0$  & $0$ & $\Sd$ \\ \hline
\end{tabular}

\vspace{0.2in}
\noindent
Finally, for the cases where neither $A$ nor $B$ is free the Mackey
functors $\uTor_i(A,B)$ are as follows:

\vspace{0.2in}

$\uTor_i(A,B)$:

\vspace{0.1in}

\renewcommand{\arraystretch}{1.2}
\begin{tabular}{|c||c|c|c|c|c|c|}
\hline
 & $H^{op}$,$H^{op}$ & $H^{op}$, $\Sth$ & $H^{op}$, $\Sd$ &
 $\Sth$, $\Sth$ & $\Sth$, $\Sd$ & $\Sd$, $\Sd$
\\ \hline
2& $\Sd$ & $0$ & $\Sd$ & $0$ & $0$ & $\Sd$  \\ \hline
1 & $\Sd$ & $\Sd$ & $0$ & $0$ & $\Sd$ & $0$ \\ \hline
0 & $H^{op}$ & $H^{op}$ & $0$ & $H^{op}$ &  $0$ & $\Sd$ \\ \hline
\end{tabular}
\vspace{0.1in}
\end{prop}

\begin{proof}
Lots of calculations, but completely routine.
\end{proof}

\begin{remark}
Notice in the above tables that $M\bbox F$, $\uHom(M,F)$, and
$\uHom(F,M)$ are always sums of copies of $F$, no matter what $M$ is.
This can be explained by the observation that $F$ can be given the structure
of a Mackey field over $\undt$: it is the Mackey field
\[ \xymatrix{*{\phantom{-.} \F_4\ }  \ar@(ul,dl)[]_t \ar@<0.5ex>[r]^\tr  & \F_2 \ar@<0.5ex>[l]^{i}
}
\]
where $t$ is the nontrivial automorphism.
The above constructions have natural
structures of $F$-modules, and over a Mackey field all modules are
free.

Likewise, one might notice that for $i>0$ one always has $\uExt^i(A,B)\tht=0=
\uTor_i(A,B)\tht$.  This follows from the observation that for any
Mackey functor $M$ one has $M\tht=0$ if and only if $M\bbox F=0$
(because $M\bbox F=F\tht(M\tht)$ as in Example~\ref{ex:box-with-F}).  Then one calculates that
\[ \uTor_i(A,B)\bbox F=\uTor_i(A,B\bbox F), \qquad \uExt^i(A,B)\bbox
F=\uExt^i(A,B\bbox F)
\]
and uses the fact that $B\bbox F$ is a finite direct sum of copies of
$F$, making it both projective (hence flat) and injective.
\end{remark}


\section{Complexes of $\protect\undr{2}$-modules}
\label{se:complexes}

If $\cA$ is any abelian category we can consider the category of chain
complexes $\Ch(\cA)$ and the associated homotopy category
$\Ho(\Ch(\cA))$ obtained by inverting quasi-isomorphisms.  This is
also known as the derived category $\cD(\cA)$.  As usual, if $R$ is a
Mackey ring and
$\cA=R\MMod$ this homotopy category will also be denoted $\cD(R)$.
Our goal is to understand what we can about the structure of
$\cD(\undt)$.

\medskip

We will occasionally have to distinguish between constructions on
$\Ch(\undt)$ and their derived versions on $\cD(\undt)$.  The derived version
of the box product will be denoted $\dbox$, and the derived version of
the cotensor will be denoted $\dcotens(\blank,\blank)$.

A complex in $\Ch(\undt)$ is called \dfn{perfect} if it is
quasi-isomorphic to a bounded complex that in each degree is a
direct sum of finitely-many copies of $H$ and $F$.
These are the compact objects in $\cD(\undt)$, by the same argument as
given in \cite[Proposition 6.4]{BN}.    Write
$\cD(\undt)_{\perf}$ for
the full subcategory of $\cD(\undt)$ whose objects are the perfect complexes.
Our first goal is to
completely classify these objects.

The full subcategory of $\undt\MMod$ whose objects are $F$ and $H$ is, by inspection,
isomorphic to the full subcategory of $\Z/2[C_2]$-modules consisting of
$\Z/2[C_2]$ and $\Z/2$ (with $C_2$ acting trivially).  All
finitely-generated $\Z/2[C_2]$-modules are direct sums of these two types,
so the full
subcategory of finitely-generated free $\undt$-modules is isomorphic to
the category of finitely-generated $\Z/2[C_2]$-modules.  Comparing the
box products of $F$ and $H$ from Proposition~\ref{pr:box_table} with
the tensor products of $\Z/2[C_2]$ and $\Z/2$, we see the symmetric
monoidal structures agree.  Finally, recall that $\cD(\cR)_{\perf}$
for any ring (or  Mackey-ring) $\cR$ can be
modeled by the category of bounded complexes of
finitely-generated projectives and chain homotopy classes
of maps.  The following
is an immediate consequence:

\begin{prop}
\label{pr:K-category}
The map $M\mapsto M\tht$ gives an equivalence of symmetric monoidal categories $\cD(\undt)_{\perf}\he
K_{b,fg}(\Z/2[C_2])$, where $K$ denotes the category of
chain complexes and chain homotopy classes of maps and ``$b$, $fg$''
indicate bounded complexes of finitely-generated modules.
\end{prop}

The above proposition is not required for our classification, but it
is informative to identify our main problem with a more classical
problem from ordinary ring theory.  We are grateful to Paul Balmer for
pointing out this connection.\medskip

We now isolate the
following special classes of perfect complexes.  By a \dfn{strand}, we
mean a perfect complex that in each degree is equal to either $H$, $F$, or
$0$, with all nonzero terms consecutive, and where all maps between nonzero terms are nonzero.  It is not hard to determine all possible strands. In addition
to the contractible complexes that have a single isomorphism and zero elsewhere
(e.g.\ $t:F\to F$), we have the following complete list of strands.

For $k\geq 0$, let $A_k$ be the complex
\[
A_k\colon \qquad\qquad \xymatrix{0 \ar[r] &
F\ar[r]^{1+t} & F\ar[r]^{1+t} & F \ar[r]^{1+t} & \cdots
\ar[r]^{1+t} & F\ar[r] & 0,
}
\]
where the nonzero groups are in degrees $0$
through $k$ (note that our complexes use homological grading, where
the differentials decrease degree).  For $n>0$, let $H(-n)$ be the complex
\[ H(-n)\colon \qquad\qquad
\xymatrix{
0\ar[r] & H \ar[r]^p& F \ar[r]^{1+t} &F \ar[r]^{1+t} & \cdots
\ar[r]^{1+t} &F \ar[r] & 0,
}
\]
where $H$ is in degree $n$ and the rightmost $F$ is in degree $0$.
Likewise (still for $n>0$), let $H(n)$ be the complex
\[ H(n)\colon \qquad\qquad
\xymatrix{
0\ar[r] & F \ar[r]^{1+t} &F \ar[r]^{1+t} &\cdots \ar[r]^{1+t} &F \ar[r]^{p} &H \ar[r]
&0,
}
\]
where the leftmost $F$ is in degree $0$ and the $H$ is in degree $-n$.  Let $H(0)$ be the complex $0 \lra H \lra 0$ with $H$ in degree $0$.
 Finally, let $B_r$ be the complex
\[
B_r\colon \qquad
\xymatrix{
0\ar[r] & H \ar[r]^{p} & F \ar[r]^{1+t} & F \ar[r]^{1+t} &\cdots \ar[r]^{1+t} &F \ar[r]^{p}
& H \ar[r] & 0,
}
\]
where the leftmost $H$ is in degree $0$ and the rightmost $H$ is in
degree $-(r+2)$.

We call $A_k$, $B_r$, and $H(n)$ (allowing $k,r \geq 0$ and $n \in \Z$) our {\bf fundamental
  complexes\/}.  There are evident homotopy cofiber
sequences
\[ A_0\lra A_k \lra \Sigma A_{k-1},\qquad B_0\lra B_{r} \lra
\Sigma^{-1}B_{r-1},\]
\[ A_{n-1}\lra H(-n)\lra \Sigma^n H, \qquad \Sigma^{-(n+2)}H \lra B_n
\lra \Sigma^{-(n+1)}H(-(n+1)),
\]
as well as many others which the reader can readily identify.

\begin{remark}
Note that $\Sth\he H(-1)$, $\Sd\he \Sigma^2 B_0$, and $H^{op}\he
H(-2)$, using the resolutions from Section~\ref{se:hom-alg}.
\end{remark}

It will turn out that for all $n\in \Z$, $H(n)$ is invertible in $\cD(\undt)$, and in fact the
objects $\Sigma^p H(n)$ constitute all the invertible objects in
$\cD(\undt)$.  The inverse of $H(n)$ is $H(-n)$, and more generally
one has $H(n)\bbox H(m)\he H(n+m)$.  See Proposition~\ref{pr:dbox} below
for this and related facts, as well as Theorem~\ref{th:Picard}.

\medskip

For $C$ any object in $\cD(\undt)$, define
\[ H^{p,q}(C)=\cD(\undt)(C,\Sigma^{p} H(q)).
\]
The derived box product of maps makes  $\bigoplus_{p,q} H^{p,q}(H)$ into a bigraded ring, which
we denote $\M_2$, and $\bigoplus_{p,q} H^{p,q}(C)$ is a bigraded
$\M_2$-module.  We record the bigraded module structure on a grid, where
each dot indicates a $\Z/2$, and the vertical and diagonal lines indicate
action by certain elements $\tau\in H^{0,1}(H)$ and $\rho\in H^{1,1}(H)$.

The ring $\M_2$ is shown in Figure~\ref{fig:M2}.  The ring structure
is completely determined by the picture (showing $\tau$- and
$\rho$-multiplications) together with the assertion that $\theta^2=0$, from which it follows that the product of any two classes in the
``lower cone'' descending from $\theta$ is also zero.  This
computation, as well as Proposition~\ref{pr:cohom-rings} below,
is jumping ahead a bit, but it is useful to see these
 calculations right away to get a sense of the fundamental complexes.
 More information about these calculations and their consequences will be given in Section~\ref{se:distinguish} below.

\begin{figure}[ht]
\begin{tikzpicture}[scale=0.7]
\draw[step=1cm,gray,very thin] (-4.5,-2.5) grid (4.5,6.5);
\draw[] (-4.5,2) -- (4.5,2) node[below, black] {\small $p$};
\draw[] (0,-2.5) -- (0,6.5) node[left, black] {\small $q$};

\foreach \y in {2,...,5}
  \foreach \x in {\y,...,2}
    \fill (\x-1.5,\y+0.5) circle (3pt);

\foreach \x in {0,...,3}
  \draw[thick,->] (\x+0.5,\x+2.5) -- (\x+0.5,6.25);
\foreach \y in {2,...,5}
  \draw[thick,->] (0.5,\y+0.5) -- (6.25-\y,6.25);

\foreach \x in {-2,...,0}
  \foreach \y in {-2,...,\x}
    \fill (\x+0.5,\y+0.5) circle (3pt);

\foreach \x in {-2,...,0}
  \draw[thick,->] (\x+0.5,\x+0.5) -- (\x+0.5,-2.25);
\foreach \y in {-2,...,0}
  \draw[thick,->] (0.5,\y+0.5) -- (-2.25-\y,-2.25);

\draw[] (0.5,3.25) node[left,black] {\small $\tau$};
\draw[] (0.5,2.25) node[left,black] {\small $1$};
\draw[] (1.95,3.25) node[left,black] {\small $\rho$};
\draw[] (1.1,0.35) node[left,black] {\small $\theta$};
\draw[] (1.1,-0.65) node[left,black] {\small $\frac{\theta}{\tau}$};
\draw[] (-0.6,-0.55) node[left,black] {\small $\frac{\theta}{\rho}$};
\end{tikzpicture}
\caption{The ring $\M_2=H^{*,*}(H)$.}
\label{fig:M2}
\end{figure}
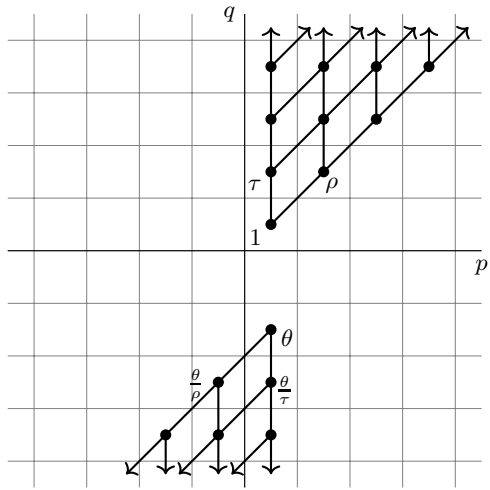

\begin{prop}
\label{pr:cohom-rings}
The cohomology groups $H^{p,q}(C)$ for $C\in \{A_k,H(n),B_r\}$
are given
by the following pictures:
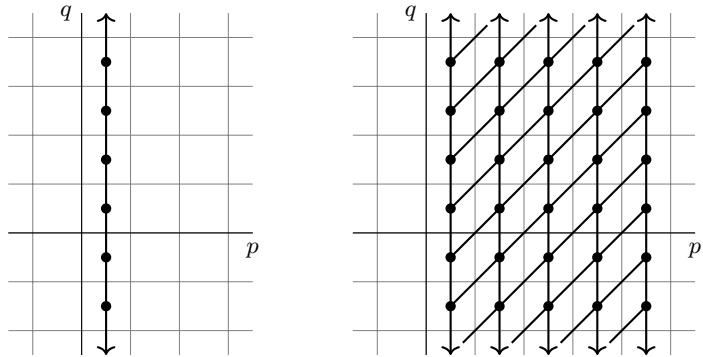
\begin{figure}[ht]
\begin{tikzpicture}[scale=0.65]
\draw[step=1cm,gray,very thin] (-1.5,-2.5) grid (3.5,4.5);
\draw[] (-1.5,0) -- (3.5,0) node[below, black] {\small $p$};
\draw[] (0,-2.5) -- (0,4.5) node[left, black] {\small $q$};

\draw[thick,<->] (0.5,-2.5) -- (0.5,4.5);
\foreach \y in {-2,...,3}
  \fill (0.5, \y+0.5) circle (3pt);

\end{tikzpicture}\hspace{1cm}
\begin{tikzpicture}[scale=0.65]
\draw[step=1cm,gray,very thin] (-1.5,-2.5) grid (5.5,4.5);
\draw[] (-1.5,0) -- (5.5,0) node[below, black] {\small $p$};
\draw[] (0,-2.5) -- (0,4.5) node[left, black] {\small $q$};

\foreach \y in {-2,...,3}
  \foreach \x in {0,...,4}
    \fill (\x+0.5, \y+0.5) circle (3pt);

\foreach \x in {0,...,4}
  \draw[thick,<->] (\x+0.5,-2.5) -- (\x+0.5,4.5);

\draw[thick] (0.5,-1.5) -- (4.5,2.5);
\draw[thick] (0.5,-0.5) -- (4.5,3.5);
\draw[thick] (0.5,0.5) -- (4.25,4.25);
\draw[thick] (0.5,1.5) -- (3.25,4.25);
\draw[thick] (0.5,2.5) -- (2.25,4.25);
\draw[thick] (0.5,3.5) -- (1.25,4.25);

\draw[thick] (0.75,-2.25) -- (4.5,1.5);
\draw[thick] (1.75,-2.25) -- (4.5,0.5);
\draw[thick] (2.75,-2.25) -- (4.5,-0.5);
\draw[thick] (3.75,-2.25) -- (4.5,-1.5);

\end{tikzpicture}
\caption{The cohomology of $A_0$ and $A_4$.}
\label{fig:A}
\end{figure}

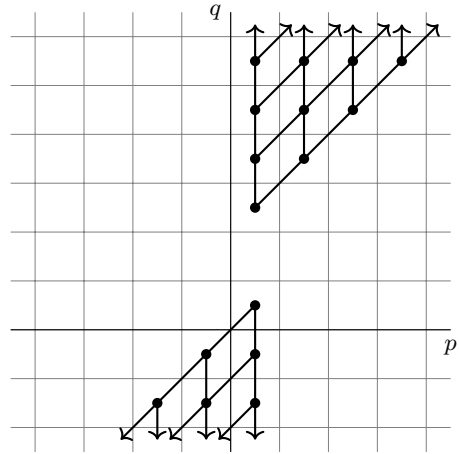
\begin{figure}[ht]
\begin{tikzpicture}[scale=0.65]
\draw[step=1cm,gray,very thin] (-4.5,-2.5) grid (4.5,6.5);
\draw[] (-4.5,0) -- (4.5,0) node[below, black] {\small $p$};
\draw[] (0,-2.5) -- (0,6.5) node[left, black] {\small $q$};

\foreach \y in {2,...,5}
  \foreach \x in {\y,...,2}
    \fill (\x-1.5,\y+0.5) circle (3pt);

\foreach \x in {0,...,3}
  \draw[thick,->] (\x+0.5,\x+2.5) -- (\x+0.5,6.25);
\foreach \y in {2,...,5}
  \draw[thick,->] (0.5,\y+0.5) -- (6.25-\y,6.25);

\foreach \x in {-2,...,0}
  \foreach \y in {-2,...,\x}
    \fill (\x+0.5,\y+0.5) circle (3pt);

\foreach \x in {-2,...,0}
  \draw[thick,->] (\x+0.5,\x+0.5) -- (\x+0.5,-2.25);
\foreach \y in {-2,...,0}
  \draw[thick,->] (0.5,\y+0.5) -- (-2.25-\y,-2.25);

\end{tikzpicture}
\caption{The cohomology of $H(2)$.}
\label{fig:H}
\end{figure}

\begin{figure}[ht]
\begin{tikzpicture}[scale=0.65]
\draw[step=1cm,gray,very thin] (-4.5,-1.5) grid (5.5,5.5);
\draw[] (-4.5,0) -- (5.5,0) node[below, black] {\small $p$};
\draw[] (0,-1.5) -- (0,5.5) node[left, black] {\small $q$};

\foreach \y in {0,...,5}
  \fill (\y-3.5,\y-0.5) circle (3pt);
\foreach \y in {0,...,5}
  \fill (\y-2.5,\y-0.5) circle (3pt);
\foreach \y in {0,...,5}
  \fill (\y-1.5,\y-0.5) circle (3pt);
\foreach \y in {0,...,5}
  \fill (\y-0.5,\y-0.5) circle (3pt);

\foreach \x in {-1,...,-4}
  \draw[thick,<->] (\x-0.5,-1.5) -- (\x+6.5,5.5);
\foreach \x in {-4,...,-2}
  \draw[thick] (\x+0.5,\x+3.5) -- (\x+0.5,-1.15);
\foreach \x in {-1,...,1}
  \draw[thick] (\x+0.5,\x+3.5) -- (\x+0.5,\x+0.5);
\foreach \x in {2,...,4}
  \draw[thick] (\x+0.5,\x+0.5) -- (\x+0.5,5.15);

\end{tikzpicture}
\caption{The cohomology of $B_3$.}
\label{fig:D}
\end{figure}
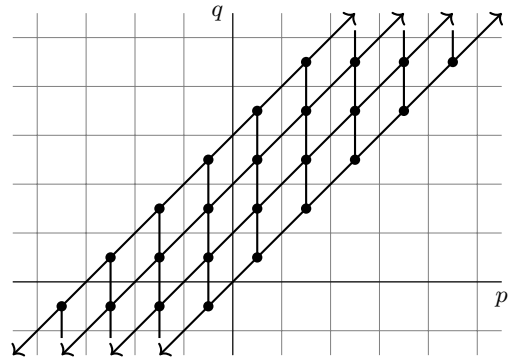

\begin{enumerate}[(i)]
\item $A_k$: Vertical strip stretching from $p=0$ through $p=k$. See Figure \ref{fig:A}.
\item $H(n)$: Copy of $\M_2$ generated by a class in bidegree $(0,n)$. Here $n$ can be positive or negative. See Figure \ref{fig:H}.
\item $B_r$: Diagonal strip occupying the diagonals $p-q=0$ through
$p-q=-r$.  See Figure \ref{fig:D}.
\end{enumerate}

In more algebraic terms, we have
\begin{enumerate}[(i)]
\item $H^{*,*}(A_k)=\M_2[\tau^{-1}]/(\rho^{k+1})$,
\item $H^{*,*}(H(n))=\M_2\langle e\rangle$ where $e$ has bidegree
$(0,n)$,
\item $H^{*,*}(B_r)=\M_2[\rho^{-1}]/(\tau^{r+1})$.
\end{enumerate}
\end{prop}

The proof of Proposition~\ref{pr:cohom-rings} appears in
Section~\ref{se:distinguish} below.

\begin{remark}
\label{re:M2-ag}
The ring $\M_2$ also appears in algebraic geometry, but with a
different grading: it is the ring $\bigoplus_{p,q} H^p(\PF^1,\cO(q))$
(over the ground field $\F_2$).  But in this context the ``upper
cone'' is entirely concentrated in degree $p=0$ whereas the ``lower
cone'' is entirely concentrated in degree $p=1$.  
\end{remark}

The following is our main algebraic classification theorem.  In short, it says every bounded complex of projectives splits into strands.  Before stating it we need one more
piece of notation.  If $M$ is any $\undt$-module, let $\bD(M)$ denote the
chain complex with $M$ in degrees $0$ and $1$, zeros in all other
degrees, and with the differential $\bD(M)_1\ra \bD(M)_0$ given by the identity map.
Notice that $\bD(M)$ is contractible (and here $\bD$ stands for `disk').

\begin{thm}
\label{th:complex-splitting}
Let $C$ be a bounded complex that in each degree is given by a finite direct
sum of copies of $F$ and $H$.  Then $C$ is isomorphic to a direct sum
of shifts of the complexes $A_k$, $B_r$, $H(n)$, for various values of $k,r \geq 0$ and $n \in \Z$,
and the contractible complexes $\bD(F)$ and $\bD(H)$.
\end{thm}

\begin{proof}
The proof is technical and we defer it to its own
section.  See Section~\ref{se:complex-classify-proof}.
\end{proof}

\begin{cor}
\label{co:perfect-complexes-splitting}
Any perfect complex is quasi-isomorphic to a direct sum of shifts of
the fundamental complexes $A_k$, $B_r$, and $H(n)$, for various values of $k,r \geq 0$ and $n \in \Z$.
\end{cor}

\begin{proof}
Immediate.
\end{proof}

The following corollary is also very useful:

\begin{cor}
\label{co:F-complexes}
Let $C$ be a bounded complex which in each degree is a finite direct
sum of copies of $F$.  Then $C$ is isomorphic to a direct sum of
shifts of the complexes $A_k$ (for various $k \geq 0$) and the contractible complexes $\bD(F)$.
\end{cor}

\begin{proof}
Theorem~\ref{th:complex-splitting} implies that $C$ is isomorphic to a direct sum of
the desired types of complexes and also the $B_r$, $H(n)$, and $\bD(H)$
complexes.  None of these last three types can appear since $H$
is not a direct summand of any $F^m$.  This last fact follows because
in $H$ the structure map $p_*$ is zero, whereas in $F^m$ it is surjective.
\end{proof}

\subsection{Results for bounded below complexes}
We can extend the above results to chain complexes that are bounded below
by adding two more fundamental complexes to our list. Let $B_\infty$ denote
the complex that has an $F$ in each positive degree and ends with an
 $H$ in degree zero, and let $A_\infty$ denote the complex that has an
 $F$ in all nonnegative degrees. In both, the maps $F\to F$ are given by
 $1+t$, and in $B_\infty$ the map $F\to H$ is $p^*$.

The complex $A_\infty$ clearly deserves its name, as it is the colimit
of an evident sequence:
\[ A_\infty \iso \colim \bigl [ A_0 \ra A_1 \ra A_2 \ra \cdots \bigr].
\]
The situation is less clear for $B_\infty$.  On the one hand,
$B_\infty$ is the colimit of a sequence
\[ H \ra \Sigma H(1) \ra \Sigma^2 H(2) \ra \Sigma^2 H(3) \ra \cdots.
\]
But at the same time, $B_\infty$ is also the inverse limit of an evident
sequence involving $B$-complexes:
\[ B_\infty \iso \lim \bigl[ \cdots \ra \Sigma^4 B_2\ra
\Sigma^3 B_1\ra \Sigma^2 B_0    \bigr ].
\]
The appropriateness of the name $B_\infty$ really becomes clear in
Proposition~\ref{pr:dbox}(vii) below, as well as in the following
result (whose proof is again deferred until Section~\ref{se:distinguish}):

\begin{prop}
\label{pr:cohom-rings-2}
The bigraded cohomology of $A_\infty$ is
$H^{*,*}(A_\infty)=\M_2[\tau^{-1}]$.  For $B_\infty$,
$H^{*,*}(B_\infty)$ is the $\Sigma^{1,0}$-suspension of the bigraded module quotient
$\M_2[\tau^{-1},\rho^{-1}]/\M_2[\rho^{-1}]$.  The picture of
$H^{*,*}(B_\infty)$ is similar to Figure~\ref{fig:D} except the
diagonal strip extends to
occupy the entire half-plane below the diagonal $p-q=2$.
\end{prop}

\begin{thm}
\label{th:bd-below-splitting}
Let $C$ be a bounded below complex that in each degree is a finite
direct sum of copies of $F$ and $H$.  Then $C$ is isomorphic to a
direct sum of shifts of the complexes $A_k$, $B_r$, $H(n)$,
$A_\infty$, $B_\infty$, and the contractible complexes
$\bD(F)$ and $\bD(H)$ (for various values of $k,r \geq 0$ and $n \in \Z$).
\end{thm}

\begin{proof}
Fix $C$ as in the statement of the proposition.  Without loss of
generality we can assume that $C$ vanishes below degree zero.   Let $C[0,m]$ be the truncation of $C$ consisting of
degrees $0$ through $m$.  Recall the notion of basis from
Definition~\ref{de:basis} and let $S_m$ be the set of all
(homogeneous) bases for
$C[0,m]$ that make the truncation split as a direct sum of shifts of
the complexes $A_k$, $B_r$, $H(n)$, $\bD(F)$, and $\bD(H)$. There are
evident maps $S_m\to S_{m-1}$ given by forgetting the basis in degree
$m$.  The proposition is then equivalent to the statement that the
inverse limit $S=\varprojlim S_m$ is nonempty. By Theorem
\ref{th:complex-splitting}, the set $S_m$ is nonempty for each $m$,
and since $C$ is finite-dimensional in each degree, $S_m$ is finite for each
$m$.  A filtered inverse limit of any collection of finite, nonempty sets is
nonempty \cite[E III.58, Theorem 1]{Bo}, and thus $S$ is nonempty. This completes the proof.
\end{proof}

\begin{remark}
\label{re:strand}
When $p$ is an odd prime and one considers $C_p$-Mackey functors,
there are again basic projective $\undr{p}$-modules $F$ and $H$.  One could
ask about a similar splitting theorem for complexes of $\undr{p}$-modules in this context.  The situation here is much more
complicated, however, and nothing as simple as Theorem~\ref{th:complex-splitting} could work.
Theorem~\ref{th:complex-splitting} implies every perfect
complex is quasi-isomorphic to a direct sum of strands, but in
the $C_p$ case  there are perfect complexes that do not split this way.

For example, let $p = 3$ and consider $G=C_3$ with generator $\gamma$.  Now let $H
= \undr{3}$
and $F = F\tht(\Z/3)$ be the analogous Mackey functors in $\Mack_{C_3}$.
The complex
\[
0 \ra F \xrightarrow{\tiny \begin{bmatrix} 1+\gamma+\gamma^2 \\ 1-\gamma \end{bmatrix}} F \oplus F \xrightarrow{\tiny \begin{bmatrix} 1-\gamma & 0 \end{bmatrix}} F \ra 0
\]
cannot be quasi-isomorphic to a direct sum of strands.
In particular, the homology at the middle spot is a Mackey functor that is not found in the homology of any possible strand.
\end{remark}

\subsection{Duality for complexes}

The functor $(\blank)^\op\colon \undt\MMod\ra \undt\MMod$ extends to a
functor $\Ch(\undt)\ra \Ch(\undt)$.  This functor preserves quasi-isomorphisms
and so further extends to a functor $\cD(\undt)\ra \cD(\undt)$.  We calculate this functor
for the fundamental complexes:

\begin{prop}  The duals of the fundamental complexes are
\begin{enumerate}[(a)]
\item $A_k^\op\he \Sigma^{-k}A_k$,
\item $H(n)^\op\he H(-(n+2))$, and
\item $B_r^\op\he \Sigma^{r+4}B_r$.
\end{enumerate}
\end{prop}

\begin{proof}
Part (a) is a direct computation.  For (b), when $n>0$ the complex
$H(n)^\op$ is
\[ 0 \ra H^\op \ra F \ra F \ra \cdots \ra F \ra 0
\]
where the $H^\op$ is in degree $n$ and the rightmost $F$ is in degree
$0$.  Patching the resolution $0 \ra H \ra F \ra F \ra H^\op$ on to
the end of this complex, we see that $H(n)^\op\he H(-(n+2))$.
Similarly, for the complex $H(-n)^\op$ we use the quasi-isomorphism
\[ \xymatrix{
0 \ar[r] & F \ar[r] & F \ar[r] & \cdots \ar[r] & F \ar[r] & F \ar[r] &
F \ar[r] & H^\op \ar[r] & 0  \\
0 \ar[r]\ar[u]^\id & F\ar[u]^\id\ar[r] & F\ar[u]^\id \ar[r] & \cdots \ar[r] & F
\ar[r]\ar[u]^\id & H \ar[r]\ar[u]^p & 0\ar[u]
}
\]
to see that $H(-n)^\op\he H(n-2)$.  This proves (b).

Finally, the proof of (c) is a combination of the two types of
arguments used for (b).  The complex $B_r^\op$ has an $H^\op$ at both
ends, in degrees $r+2$ and $0$.  The one in degree $r+2$
 can be removed by patching in a resolution, and the one in degree $0$
is removed via a quasi-isomorphism as we used in
(b).  The end result is a copy of $B_r$ that lies in degrees $r+4$ down
through $2$, which is $\Sigma^{r+4}B_r$.
\end{proof}

Another kind of duality involves the functor $\dcotens(\blank,H)$.
There is a canonical map
\[ \Gamma_{X,Z}\colon \dcotens(X,H)\dbox Z \lra  \dcotens(X,Z)
\]
obtained as the adjoint of the composite
\[ \xymatrixcolsep{3pc}\xymatrix{
\dcotens(X,H)\dbox Z \dbox X \ar[r]^-{\id\tens t_{Z,X}}
&\dcotens(X,H)\dbox X \dbox Z \ar[r]^-{\ev\tens \id_Z} &
H\dbox Z \ar[r]^\iso & Z.
}
\]

The following is a standard argument:

\begin{prop}
\label{pr:H-duality}
If $X$ is a perfect complex then $\Gamma_{X,Z}$ is an isomorphism in
$\cD(\undt)$, for any $Z$.
\end{prop}

\begin{proof}
One checks the result directly for the two special objects $X=H$ and $X=F$.
Then use that the property persists under direct sums, extensions, and
retracts.
\end{proof}

\begin{remark}
We will not need it, but the standard arguments also show that the
canonical map $X\ra \dcotens(\dcotens(X,H),H)$ is an isomorphism in
$\cD(\undt)$ whenever $X$ is perfect.  
\end{remark}

\subsection{Calculations for fundamental complexes}
Moving towards our goal of understanding everything we can about
$\cD(\undt)$, the next step is to calculate the effects of basic
operations on the fundamental
complexes.  We start by computing the homology modules of the
complexes:

\begin{prop}
\label{pr:strand-homology}
For $k\geq 2$, $n \geq 1$, and $r \geq 0$ the homology of the different strands is given by the following
formulas.  
\begin{enumerate}[(a)]
\item $H_i(A_0)= \begin{cases} F & \text{if \ $i=0$,}\\ 0 & \text{else.} \end{cases}$
\item $H_i(A_1)= \begin{cases} H & \text{if \ $i = 1$,}\\ H^\op & \text{if \ $i=0$,}\\ 0 & \text{else.} \end{cases}$
\item $H_i(A_k)= \begin{cases} H & \text{if \ $i = k$,}\\ S\dt & \text{if \ $0<i<k$,}\\ H^\op & \text{if \ $i=0$,}\\ 0 & \text{else.} \end{cases}$
\item $H_i(H(n))= \begin{cases} H & \text{if \ $i=0$,}\\ S\dt & \text{if \ $-n \leq i < 0$,}\\ 0 & \text{else.}\end{cases}$
\item $H_i(H(0))= \begin{cases} H & \text{if \ $i=0$,}\\ 0 & \text{else.} \end{cases}$
\item $H_i(H(-1))= \begin{cases} S\tht & \text{if \ $i=0$,}\\ 0 & \text{else.} \end{cases}$
\item $H_i(H(-2))= \begin{cases} H^\op & \text{if \ $i=0$,}\\ 0 & \text{else.} \end{cases}$
\item For $n> 2$, $H_i(H(-n))= \begin{cases} S\dt & \text{if \ $1 \leq i \leq n-2$,} \\ H^\op & \text{if \ $i=0$,} \\ 0 & \text{else.} \end{cases}$
\item $H_i(B_r)= \begin{cases} S\dt & \text{if \ $-r-2 \leq i \leq -2$,} \\ 0 & \text{else.}\end{cases}$
\end{enumerate}
\end{prop}

\begin{proof}
These are direct (and straightforward) computations with the complexes.
\end{proof}

\begin{prop}
\label{pr:dbox}
Let $k,l,r \geq 0$ and $m,n \in \Z$.  The (derived) box product of fundamental complexes is given by the following
formulas:
\begin{enumerate}[(i)]
\item $A_k\dbox A_l\he A_k \oplus \Sigma^l A_k$ for $k\leq l$.
\item $A_k\dbox H(n)\he A_k$.
\item $A_k\dbox B_r\he 0$.
\item $H(n)\dbox H(m)\he H(n+m)$.
\item $H(n)\dbox B_r \he \Sigma^{-n} B_r$
\item $B_r\dbox B_l \he \Sigma^{-(l+2)}B_r \oplus B_r$ for $r\leq l$.
\item Parts (i)--(iii) and (v)--(vi) also hold when $k$, $l$, or $r$ is
$\infty$, if we interpret any term with $\Sigma^\infty$ or
$\Sigma^{-\infty}$ as being zero.  For example, $A_k\dbox A_\infty \he
A_k$ and $H(n)\dbox B_\infty \he \Sigma^{-n}B_\infty$ for all $k\leq
\infty$ and $n<\infty$.
\end{enumerate}
\end{prop}
\begin{proof}
There are a variety of ways to do these computations using the
standard machinery of homological algebra.  The following is one route
through this.

\smallskip
\noindent
(1): We claim $F\bbox B_r\he 0$, $F\bbox H(n) \he F$, and $F\bbox A_k\he
F\oplus \Sigma^k F$.

\smallskip
To see this, note since $F$ is projective we have $H_i(F\bbox C)\iso F\bbox H_i(C)$ for
any $C$.  Then Proposition~\ref{pr:strand-homology}, together with Proposition~\ref{pr:box_table},  shows that $H_*(F\bbox B_r)=0$, and $H_*(F\bbox H(n))$ is $F$ concentrated entirely in
degree $0$.  Now Corollary~\ref{co:F-complexes} implies that $F\bbox
H(n)\he A_0=F$.  Similarly, $H_*(F\bbox A_k)$ is $F$ in degrees $0$
and $k$, and zero elsewhere.  By Corollary~\ref{co:F-complexes},
$F\bbox A_k$ must be quasi-isomorphic
to a direct sum of $A$-strands, so the only possibility is $F\bbox
A_k\he F\oplus \Sigma^k F$.

\smallskip
\noindent
(2): $B_r\bbox A_k\he 0$, $B_r\bbox H(n)\he \Sigma^{-n}B_r$, and
$A_k\bbox H(n)\he A_k$.

\smallskip
For this step, apply $B_r\bbox(\blank)$ to the cofiber sequence $A_0\ra A_k \ra \Sigma A_{k-1}$ and use induction to conclude that $B_r\bbox A_k\he 0$ for all $k$.
When $n>0$, for $B_r\bbox H(n)$
use the cofiber sequence $A_{n-1} \ra H \ra \Sigma^n H(n)$ and
apply $B_r \bbox(\blank)$.   Similarly, for $B_r\bbox H(-n)$ use
a cofiber sequence $\Sigma^{n-1}H \ra
A_{n-1} \ra H(-n)$.

When $n\geq 0$, applying $A_k\bbox (\blank)$ to the cofiber sequence $\Sigma^{-(n+2)}H \ra
B_n \ra \Sigma^{-(n+1)}H(-(n+1))$ and using what we have
already proven immediately yields $A_k\bbox H(-(n+1))\he A_k$.
A similar argument applied to the sequence $\Sigma^{-1} H(n+1)\ra B_n
\ra H$ yields $A_k\bbox H(n+1)\he A_k$.

\smallskip
\noindent
(3): If $k\leq l$ then $A_k\bbox A_l \he A_k\oplus \Sigma^l A_k$ and if $r \leq l$ then
$B_r\bbox B_l \he B_r\oplus \Sigma^{-(l+2)} B_r$.

\smallskip
Observe first that there is an evident cofiber sequence
\[ \Sigma^{-(l+1)} H \lra H(l+1) \lra \Sigma^{-l}A_l.
\]
Boxing with $A_k$ and using $A_k\bbox H(l+1)\he A_k$, and for
convenience applying $\Sigma^{l}$, gives a homotopy cofiber
sequence in the derived category
\[ \Sigma^{-1}{A_k} \lra \Sigma^{l} A_k \lra A_k\bbox A_l.
\]
But it is immediately observed for degree reasons that there are no
nonzero maps $\Sigma^{-1}A_k \ra \Sigma^l A_k$ (since $k\leq l$), and
so we conclude $A_k\bbox A_l\he \Sigma^l A_k\oplus A_k$.  The same
proof works for the $B$-strands, starting with the cofiber sequence
\[ \Sigma^{-2} H(l) \lra B_l \lra \Sigma^{-1}H(-1)
\]
and then using that there are no nonzero maps $\Sigma^{-1}B_r \ra
\Sigma^{-(l+2)}B_r$ for $r\leq l$.

\medskip
\noindent
(4): $H(n)\bbox H(m) \he H(n+m)$.

Corollary~\ref{co:perfect-complexes-splitting} says that $H(n)\bbox H(m)$ decomposes (up to
weak equivalence) as a direct sum of shifts of $A$-, $B$-, and
$H$-strands.  Applying $A_0\bbox (\blank)$ and using the previous
parts immediately shows that no $A$-strands can appear, and that exactly
one $H$-strand must appear.  Similarly, applying $B_0\bbox (\blank)$
shows that no $B$-strands can appear and also that the $H$-strand that
appears must
be $H(n+m)$.  These four steps finish (i)--(vi).

We only briefly sketch the proofs for (vii), leaving the details to
the reader.  For
part (i), explicit hand calculation readily shows $F\bbox A_\infty \he
F$.  Pass to the general case of $A_k\bbox A_\infty$ by a double complex argument.  For
$B_r\bbox A_\infty$ in (iii) use that $B_r\bbox A_l\he 0$ and that
$A_\infty$ is the homotopy colimit of the $A_l$.
Then for $H(n)\bbox A_\infty$ in (ii) use the cofiber sequences
$\Sigma^{-2} H\ra B_0\ra
\Sigma^{-1} H(-1)$ and $H\ra B_0 \ra \Sigma^{-1} H(1)$ and induction.

For the remaining cases of (iii), first prove $F\bbox B_\infty \he 0$ by explicit
calculation.  Then get $A_k\bbox B_\infty \he 0$ by induction
using the cofiber sequence $A_{k-1}\ra A_k \ra \Sigma^k F$.  For
$B_r\bbox B_\infty$ in (vi) use $H \ra B_\infty
\ra \Sigma A_\infty$ together with (iii).  Finally, for (v) use the
cofiber sequence $\Sigma^{-n}H \ra H(n) \ra \Sigma^{-(n-1)}A_{n-1}$
and box with $B_\infty$.
\end{proof}

\begin{cor}
For all $n\in \Z$, $H(n)$ is invertible with inverse $H(-n)$.
\end{cor}

\smallskip

Just as we defined bigraded cohomology groups $H^{p,q}(X)$ for $X\in \cD(\undt)$, we can
also defined bigraded homology groups.  Here the definition is
\[ H_{p,q}(X)=\cD(\undt)(\Sigma^p H(q),X)=\cD(\undt)(\Sigma^pH,X\dbox
H(-q))=H_p(X\dbox H(-q))_\bdot,
\]
where we have used the invertibility of $H(q)$ for the second
equality.  The right-most term suggests an extension of this to
Mackey-functor-valued homology, given simply by $H_{p,q}(X)=H_p(X\dbox
H(-q))$.
For example, a portion of the bigraded homology of $H$ is
shown in Figure \ref{fig:Hhom} (the bigraded homology extends
infinitely in the vertical directions).

\begin{figure}[ht]
\begin{tikzpicture}[scale=0.6]
\draw[step=1cm,gray,very thin] (-4.5,-2.5) grid (4.5,7.5);
\draw[] (-4.5,2) -- (4.5,2) node[below, black] {\small $p$};
\draw[] (0,-2.5) -- (0,7.5) node[left, black] {\small $q$};
\fill (0,2) circle(2pt);

\foreach \y in {-1.5,...,2.5}
  \draw (0.5,\y) node {\small{$H$}};
\draw (0.5,3.5) node {\small{$\Sth$}};
\foreach \y in {4.5,...,6.5}
  \draw (0.5,\y) node {\small{$H^{op}$}};
\foreach \y in {-3.5,...,-0.5}
  \foreach \x in {\y,...,-0.5}
    \draw (\x,\y+2) node {\small{$\Sd$}};
\foreach \y in {5.5,...,6.5}
  \foreach \x in {5.5,...,\y}
    \draw (\x-4,\y) node {\small{$\Sd$}};

\end{tikzpicture}
\caption{The bigraded homology of $H$.}
\label{fig:Hhom}
\end{figure}

We next turn to the cotensor in $\cD(\undt)$:  

\begin{prop}
\label{pr:cotens}
The cotensor between fundamental complexes is given by the following
formulas for $k,r \geq 0$ and $m,n \in \Z$:
\begin{enumerate}[(a)]
\item $\dcotens(H(m),H(n))\he H(n-m)$
\item $\dcotens(A_k,B_r)\he 0 \he \dcotens(B_r,A_k)$
\item $\dcotens(H(n),B_r)\he \Sigma^n B_r$
\item $\dcotens(H(n),A_k)\he A_k$
\item $\dcotens(A_k,H(n))\he \Sigma^{-k}A_k$
\item $\dcotens(B_r,H(n))\he \Sigma^{r-n+2}B_r$
\item $\dcotens(A_k,A_l)\he
\begin{cases}
\Sigma^{-k}A_k\oplus \Sigma^{l-k}A_k & \text{if $k\leq l$},\\
A_l\oplus \Sigma^{-k}A_l & \text{if $k\geq l$.}
\end{cases}$
\item $\dcotens(B_r,B_l)\he
\begin{cases}
\Sigma^{r-l}B_r\oplus \Sigma^{r+2} B_r & \text{if $r\leq l$},\\
\Sigma^{r+2}B_l \oplus B_l & \text{if $r\geq l$.}
\end{cases}$
\end{enumerate}
\end{prop}

\begin{proof}
One readily computes by inspection that $\cotens(A_k,H)\iso \Sigma^{-k}A_k$, while 
$\cotens(H(n),H)\iso H(-n)$, and $\cotens(B_r,H)\iso
\Sigma^{r+2}B_r$.  The desired results are then immediate from
Proposition~\ref{pr:H-duality} and Proposition~\ref{pr:dbox}.
\end{proof}

We include the following result as a curiosity.
We have
called this ``Grothendieck-Serre duality'' because of the evident
analog to the similar statement in algebraic geometry.  Note that the
appearance of $H(-2)$  is really due to the equivalence $H^{\op}\he
H(-2)$.

\begin{thm}[Grothendieck-Serre Duality]
\label{th:Serre-duality}
For perfect complexes $X$ and $Y$ one has
\[ \dcotens(X,\,Y\dbox H(-2)) \he \dcotens(Y,X)^{op}.
\]
\end{thm}

This theorem can be proven by brute force using
Theorem~\ref{th:complex-splitting} simply by checking it for the fundamental
complexes, but this is clearly not the ideal method.
As we have no need of the result elsewhere in the paper, we do not
take the time here to develop a more satisfying proof.

\vspace{0.2in}

We next look at morphisms between fundamental complexes.  Recall that we
write $\cD(\undt)(X,Y)$ for maps in the derived category from $X$ to $Y$.
Note
that since the $H(n)$ are invertible one always has
\[ \cD(\undt)(X,Y)\cong \cD(\undt)(X\dbox H(n),Y\dbox H(n)) \]
for every $n \in \Z$.

We start with the case $X=H$, which has already been done by virtue of
the isomorphism $\cD(\undt)(\Sigma^i H,Y)\iso H_i(Y)\dt$ (the
$\cpdot$ side of the Mackey functor $H_i(Y)$):

\begin{prop}
\label{pr:homotopy-compute}
Let $k, n, r \geq 0$.  The homology of the fundamental strands is nonzero
only in the following cases:
\begin{enumerate}[(a)]
\item $\cD(\undt)(\Sigma^i H, A_k)=
\Z/2$ if $0\leq i\leq k$,
\item $\cD(\undt)(\Sigma^i H, H(n))=\Z/2$ if $-n\leq i\leq 0$,
\item $\cD(\undt)(\Sigma^i H, H(-1))=0$ for all $i$,
\item $\cD(\undt)(\Sigma^i H, H(-n))=\Z/2$ if $0\leq i\leq n-2$ (for $n \geq 2$),
\item $\cD(\undt)(\Sigma^i H, B_r)=\Z/2$ if $-(r+2)\leq i\leq -2$.
\end{enumerate}
\end{prop}

\begin{proof}
Immediate from Proposition~\ref{pr:strand-homology}.
\end{proof}

Using the previous results we can easily compute $\cD(\undt)(\Sigma^i X,
Y)$ for any two fundamental complexes $X$ and $Y$.  As one example, we
compute $\cD(\undt)(\Sigma^i A_2,A_5)$ for all $i$.  We use the
isomorphisms
\begin{align*}
 \cD(\undt)(\Sigma^i A_2, A_5)&\iso \cD(\undt)(\Sigma^i H,\dcotens(A_2,A_5))\\
&\iso \cD(\undt)(\Sigma^i H, \Sigma^{-2}A_2\oplus \Sigma^3 A_2)\\
&\iso \cD(\undt)(\Sigma^{i+2}H,A_2)\oplus \cD(\undt)(\Sigma^{i-3}H,A_2)\\
&\iso
\begin{cases}
\Z/2 & \text{if $-2\leq i\leq 0$ or $3\leq i\leq 5$},\\
0 & \text{otherwise.}
\end{cases}
\end{align*}
The first isomorphism is just an adjunction, the second is by
Proposition~\ref{pr:cotens}, and the fourth isomorphism is by
Proposition~\ref{pr:homotopy-compute}.

The above technique allows one to compute $\cD(\undt)(\Sigma^i X,Y)$ for
any fundamental complexes $X$ and $Y$.  We state one specific result along
these lines, which will be needed later:

\begin{prop}
\label{pr:AD}
One has $\cD(\undt)(\Sigma^i A_k,B_r)=0$ and $\cD(\undt)(\Sigma^i B_r,A_k)=0$ for
$k, r \geq 0$ and all $i \in \Z$.
\end{prop}

\begin{proof}
This follows immediately from $\dcotens(A_k,B_r)\he 0\he
\dcotens(B_r,A_k)$, which is Proposition~\ref{pr:cotens}(b).
\end{proof}

\subsection{Distinguishing complexes}
\label{se:distinguish}

The classification in \ref{th:complex-splitting} guarantees any
perfect complex is quasi-isomorphic to a direct sum of fundamental
complexes. It does not, however, guarantee there is a unique such
direct sum.  Our main goal in this subsection is to prove the following
proposition, from which we can conclude every perfect complex is
quasi-isomorphic to a direct sum of fundamental complexes in a unique
way.  This completes the classification of objects in
$\cD(\undt)_{\perf}$.

\begin{prop}
\label{pr:classify-2}
Let $X=\bigoplus_{\alpha \in \cJ} \Sigma^{n_\alpha} A_\alpha \oplus
\bigoplus_{\beta \in \cK} \Sigma^{n_\beta} H(\beta) \oplus
\bigoplus_{\gamma \in \cL} \Sigma^{n_\gamma} B_\gamma$ and similarly
$Y=\bigoplus_{\alpha \in \cJ'} \Sigma^{n_\alpha} A_\alpha \oplus
\bigoplus_{\beta \in \cK'} \Sigma^{n_\beta} H(\beta) \oplus
\bigoplus_{\gamma \in \cL'} \Sigma^{n_\gamma} B_\gamma$ where all the
indexing sets are finite and allow for repetition.
If $X\he Y$ then the sets $\cJ$ and $\cJ'$ are equal up to
permutation, and likewise for $\cK$, $\cK'$ and $\cL$, $\cL'$.
Moreover, the suspension factors for corresponding summands must be equal.
\end{prop}

Note that we cannot prove this by simply looking at homology modules $H_*(\blank)$.
For example, by Proposition~\ref{pr:strand-homology}
$H(1)$ and $H\oplus \Sigma B_0$ have isomorphic homology modules ($H$
in degree $0$ and $S\dt$ in degree $-1$), but we will see in this
section that they are not quasi-isomorphic.  In order to distinguish
homotopy types in $\cD(\undt)$ we therefore need to use a finer
invariant.  Recall the bigraded cohomology groups of a
complex $X$ given by
\[ H^{p,q}(X)=\cD(\undt)(X,\Sigma^p H(q))
\]
and $H^{*,*}(X)=\bigoplus_{p,q} H^{p,q}(X)$.
Recall also that $\M_2=H^{*,*}(H)$. Note that the derived box
product gives pairings
\[ H^{p,q}(X)\tens H^{r,s}(Y) \lra H^{p+r,q+s}(X\dbox Y)
\]
where we are using the unique equivalence $H(q)\bbox H(s)\he H(q+s)$
in the derived category (such an equivalence exists by
Proposition~\ref{pr:dbox}(iv), and is unique because
$\cD(\undt)(H(n),H(n))\iso \Z/2$ for all $n$).  In particular, $H^{*,*}(X)$
is an $\M_2$-bimodule, and one can readily check that the left and
right module structures coincide.

Note that $\M_2^{p,q}=\cD(\undt)(H,\Sigma^p H(q))\iso H_{-p}(H(q))_\bdot$
and these groups are calculated by
Proposition~\ref{pr:homotopy-compute}.  This gives the picture of
dots in Figure~\ref{fig:M2}.  Let $\tau\in \M_2^{0,1}$, $\rho\in \M_2^{1,1}$, and $\theta\in
\M_2^{0,-2}$ be the unique nonzero classes.  It is routine to check
that these can be represented by the maps $H\ra H(1)$,
$H\ra \Sigma H(1)$, and $H\ra H(-2)$ shown in Figure~\ref{fig:rhothetatau}.
By convention, unlabeled maps between $H$'s and $F$'s
always denote the unique nonzero map, except in the case of maps $F\ra
F$ where an unlabeled map always denotes $1+t$.
To justify that the maps in Figure~\ref{fig:rhothetatau} do
 represent the indicated classes,
one only has to prove that the given  maps are not chain homotopic to
zero; this is routine.

\begin{figure}
\[\xymatrixrowsep{1pc}\xymatrix{
H\ar[d] & {\text{\fbox{$\mathbf{\tau}$}}} &&
{\text{\fbox{$\mathbf{\rho}$}}} & H\ar[d]
&&  {\text{\fbox{$\mathbf{\theta}$}}}&& H\ar[d] \\
F\ar[r] & H && F\ar[r] & H && H\ar[r] & F \ar[r] & F\\
0 & -1 && 1 & 0 &&2 & 1 & 0
}
\]
\caption{The maps $\tau$, $\rho$, and $\theta$ (complexes are drawn
  horizontally here).}
\label{fig:rhothetatau}
\end{figure}

Before giving the proof of Proposition~\ref{pr:cohom-rings} we need a
simple lemma:

\begin{lemma}
\label{le:cof-theta}
The cofiber of the map $\theta\colon H \ra H(-2)$ is weakly equivalent to
$H(-1)\oplus \Sigma H(-1)$, and under this weak equivalence the
cofiber sequence is
\[ \xymatrix{
H \ar[r]^-\theta & H(-2) \ar[r]^-{[\tau\ \  \rho]} & H(-1)\oplus \Sigma H(-1).
}
\]
\end{lemma}

\begin{proof}
The cofiber of $\theta$ is the complex on the left of the diagram
below (where the complexes are now drawn vertically):
\[\xymatrixrowsep{1pc}\xymatrixcolsep{1.2pc}\xymatrix{
H_a \ar[d]_p  &  &&&&& H_a\ar[d]_p \\
F_b \ar[d] & H_c\ar[dl]^p  &&&&& F_{b+p^*(c)} & H_c\ar[d]^p \\
F_e & &&&&&& F_e
}
\]
Here we use the subscripts to denote basis elements.  So
$d(a)=p_*(b)$, $d(b)=(1+t)e$, and $d(c)=p_*(e)$.  The change of basis
$\{b,c\}\mapsto \{b+p^*(c),c\}$ gives the complex on the right, which
is $H(-1)\oplus \Sigma H(-1)$.

Composing the inclusion of $H(-2)$ into the cofiber with the two
projections for this direct sum, we get maps $H(-2)\ra H(-1)$ and
$H(-2)\ra \Sigma H(-1)$.  To see that these are $\tau$ and $\rho$ we
only need to check that they are nonzero in $\cD(\undt)$, since there
are unique nonzero maps in each case.   The two compositions are
\[
\xymatrixrowsep{1.2pc}\xymatrixcolsep{1.4pc}\xymatrix{
H \ar[d]_p  & &&&&&  H \ar[r]^{1}\ar[d]_p & H \ar[d]_p
\\
F\ar[d]_{1+t} \ar[r]^{p} & H \ar[d]^p && \text{and} &&&   F \ar[r]^{1}\ar[d]_{1+t} & F  \\
F \ar[r]^1 & F   & &&&& F
}
\]
and in each case it is readily checked that null homotopies do not exist.
\end{proof}

\begin{proof}[Proof of Proposition~\ref{pr:cohom-rings} and Proposition~\ref{pr:cohom-rings-2}]
This is tedious but routine, and we only give a sketch.
First note that the groups $H^{*,0}(F)$
are trivially computed, and then one gets a $(0,1)$-periodicity in
$H^{*,*}(F)$ using
that $F\bbox H(1)\he F$ (see Proposition~\ref{pr:dbox}).  Similarly,
$H^{*,0}(B_0)$ is readily computed, and then one gets a
$(1,1)$-periodicity using that $B_0\bbox H(1)\he \Sigma^{-1}B_0$.

Next use the cofiber sequence $H\llra{\rho} \Sigma^1H(1) \lra \Sigma^1
F$ to deduce all the $\rho$-multiplications in $H^{*,*}(H)$.
Similarly, the cofiber sequence $H\llra{\tau} H(1) \lra \Sigma B_0$
lets one deduce all the $\tau$-multiplications in $H^{*,*}(H)$.

The cofiber sequences $A_{k-1}\inc A_k\ra\Sigma^k F$ and $F\inc A_k
\ra \Sigma A_{k-1}$ allow one to inductively compute the
$\Z/2[\tau,\rho]$-module structure on $H^{*,*}(A_k)$.  Similarly, the cofiber
sequences $B_{r-1}\ra B_r\ra \Sigma^{-r}B_0$ and $B_0\ra B_r \ra
\Sigma^{-1}B_{r-1}$ lead to the inductive computation of the
$\Z/2[\tau,\rho]$-module structure on $H^{*,*}(B_r)$.
Note also that the computations can be simplified by using the equivalences
$A_k\bbox H(1)\he A_k$ and $B_r\bbox H(1)\he \Sigma^{-1}B_r$ from Proposition~\ref{pr:dbox}, which
yield periodicities in the module structures.

Finally, by using Lemma~\ref{le:cof-theta} one readily computes all of
the $\theta$-multiplications in $\M_2=H^{*,*}(H)$ (they are all zero, except
for $1\cdot \theta=\theta$).  The rest of the ring structure can be
deduced from simple algebraic arguments.  (For example:
$\frac{\theta}{\tau} \cdot \frac{\theta}{\rho}$ is a class that when
multiplied by $\tau\rho$ gives $\theta^2$---which is zero---and by our analysis of
all the $\tau$- and $\rho$-multiplications we know this forces
the class to be zero).  Similarly, such algebraic arguments also
readily yield the full $\M_2$-module structures on $H^{*,*}(A_k)$ and
$H^{*,*}(B_r)$.

Only a bit more work is required to compute $H^{*,*}(A_\infty)$ and
$H^{*,*}(B_\infty)$.  For the first, the filtration by $A_k$'s shows
that $A_n\inc A_\infty$ induces isomorphisms on $H^{p,*}$ for $p\leq
n$ and the desired result readily follows from this.  For the second,
one first uses the cofiber sequence $H\ra B_\infty \ra \Sigma A_\infty$ to
calculate all of the cohomology groups, but there is one set of
unresolved extension problems.  
 Then one considers the map $\Sigma B_\infty\to B_\infty$ that is $p:H\to F$ in degree one and the identity in higher degrees. The cofiber is readily seen to be quasi-isomorphic to
$\Sigma^{2}B_0$, and this cofiber sequence then resolves those extensions.   
\end{proof}

Say that a bigraded $\M_2$-module is \dfn{perfect} if it is isomorphic
to $H^{*,*}(X)$ for some perfect complex $X$.  By
Theorem~\ref{th:complex-splitting} this is equivalent to saying that the module
is a finite direct sum of bigraded shifts $\Sigma^{p,q}\M_2$ as well
as shifts of $H^{*,*}(A_k)$ and $H^{*,*}(B_r)$ for various values of
$k$ and $r$.  It is not {\it a priori\/} true that the constituent
pieces of such a direct sum are uniquely determined, but they are:

\begin{prop}
\label{pr:perfect-modules}
Let $M$ be a perfect $\M_2$-module.  Then there exist unique integers
(up to permutation)
$p_i$, $q_i$, $s_j$, and $t_r$ together with $n_j\geq 0$ and $k_r\geq 0$ such that
\[ M\iso \bigoplus_i \Sigma^{p_i,q_i}\M_2 \oplus \bigoplus_j
\Sigma^{s_j}H^{*,*}(A_{n_j}) \oplus \bigoplus_r
\Sigma^{t_r}H^{*,*}(B_{k_r}).
\]
\end{prop}

\begin{proof}
We know existence, so the only thing to be proven is uniqueness.
 Observe that $\Ann_M(\tau,\rho)$ is a
finite-dimensional bigraded vector space over $\Z/2$ whose homogeneous
basis is in bidegrees $(p_i-2,q_i-2)$, so this gives uniqueness of the
$p$'s and $q$'s.

The operation $M\mapsto M[\rho^{-1}]$ kills the $H^{*,*}(A)$-summands and does
nothing to the $H^{*,*}(B)$-summands.  The construction
\[ \Ann\nolimits_{M[\rho^{-1}]} (\tau^\infty)=\{x\in
M[\rho^{-1}]\,\bigl |\, \tau^nx=0 \ \text{for some $n>0$}\}
\]
exactly isolates the $H^{*,*}(B)$-summands of $M$.  This construction
is a module over $\M_2[\rho^{-1}]\iso \Z/2[\rho,\rho^{-1}][\tau]$
which is a graded PID, so the uniqueness of the $t_r$ and $k_r$
follow from the usual classification of modules over a PID.

Finally, the operation $M\mapsto M[\tau^{-1}]$ kills the $H^{*,*}(B)$-summands,
and the construction $\Ann_{M[\tau^{-1}]}(\rho^\infty)$ exactly
isolates the $H^{*,*}(A)$-summands.  The analogous argument to the preceding
paragraph shows that the $s_j$ and $n_j$ are uniquely determined,
since $\M_2[\tau^{-1}]\iso \Z/2[\tau,\tau^{-1}][\rho]$ is again a
graded PID.
\end{proof}

Proposition~\ref{pr:classify-2} is really just a corollary of the
above:

\begin{proof}[Proof of Proposition~\ref{pr:classify-2}]
Let $X$ and $Y$ be as in the statement of the proposition and assume
$X\he Y$.  Let
$M=H^{*,*}(X)$ and $N=H^{*,*}(Y)$, so that $M\iso N$ as bigraded
$\M_2$-modules.  By Proposition~\ref{pr:perfect-modules} it follows
that $M$ and $N$ have the same constituent summands, and these exactly
correspond to the constituent summands of $X$ and $Y$.
\end{proof}

We also call attention to the following useful consequence:

\begin{cor}
\label{co:homology-determines}
Let $X$ and $Y$ be perfect complexes of $\undt$-modules.  Then $X\he
Y$ if and only if $H^{*,*}(X)\iso H^{*,*}(Y)$ as bigraded $\M_2$-modules.
\end{cor}

\begin{proof}
Immediate from Proposition~\ref{pr:perfect-modules} and
Theorem~\ref{th:complex-splitting}.
\end{proof}


\section{Algebraic consequences of the classification theorem}

In this section we compute the Picard group of $\cD(\undt)$ as well as
the Balmer spectrum for
 $\cD(\undt)_{\perf}$, deducing both as consequences of our work
 in Section~\ref{se:complexes}.

\medskip

\subsection{The Picard group of \mdfn{$\cD(\undt)$}}

\begin{thm}
\label{th:Picard}
The Picard group of $\cD(\undt)$ is $\Z \oplus \Z$, with generators $\Sigma
H$ and $H(1)$.
\end{thm}

\begin{proof}
First note that an invertible object in $\cD(\undt)$ is necessarily
compact, hence perfect.  To see this, observe that if $X$ is
invertible with inverse $Y$ and $\{Z_\alpha\}$ is any collection of
objects
then there is a commutative diagram
\[ \xymatrix{
\bigoplus_\alpha \cD(\undt)(X,Z_\alpha) \ar[r]\ar[d]_\iso &
\cD(\undt)(X,\bigoplus_\alpha Z_\alpha)\ar[d]^\iso \\
\bigoplus_\alpha \cD(\undt)(X\dbox Y ,Z_\alpha\dbox Y) \ar[r] &
\cD(\undt)(X\dbox Y,\bigl (\bigoplus_\alpha Z_\alpha\bigr )\dbox Y) \\
}
\]
Now use the fact that $(\blank)\dbox Y$ commutes with direct sums,
together with the fact that $X\dbox Y\he H$ is compact, to see that
the bottom horizontal map is an isomorphism.  Thus, the top horizontal
map is an isomorphism as well.

Since $X$ is compact, we know by Corollary~\ref{co:perfect-complexes-splitting}
that $X$ is quasi-isomorphic to a direct sum of shifts of complexes of type $A_k$, $B_r$, and
$H(n)$.  But it is easy to see that the direct sum must
only involve {\it one\/} term.  For suppose
 $X\simeq J\oplus K$, and let $Y$ be the inverse of $X$. 
Then
\[ H \he X\dbox Y \he (J\dbox Y)\oplus (K\dbox Y).
\]
By taking homology and using that $H$ is indecomposable as a $\undt$-module,
this can only happen if either $J\dbox Y\he 0$ or $K\dbox Y\he 0$.
Without loss of generality, we assume the former.  Then box with $X$
to get $0\he (J\dbox Y)\dbox X\he J$.  This proves that a nontrivial
direct sum can never be invertible.

The classification theorem then implies that the only possible
invertible objects are suspensions of $A_k$, $B_r$, and $H(n)$.
If $A_k$ has an inverse $W$, then take $A_k\bbox B_r\he 0$ and box
with $W$ to get $B_r\he 0$; this is a contradiction (by
Proposition~\ref{pr:homotopy-compute}(d), for example).  So $A_k$ is not
invertible, and the same proof shows that $B_r$ is not invertible.  Of
course we know that the $H(n)$ are invertible by Proposition~\ref{pr:dbox}(iv).

Consider the group homomorphism $\Z^2 \ra \Pic(\cD(\undt))$ sending $(m,n)$ to $\Sigma^m H \dbox H(n) \he \Sigma^m H(n)$.  The first generator of $\Z^2$ maps to $\Sigma H$ and the second to $H(1)$.  We
have just proven that this map is surjective.  For injectivity just
note that
the homology calculations of
Proposition~\ref{pr:homotopy-compute}(a,b) show that
$\Sigma^m H(n)\he H$
 can only
happen if $m=n=0$.
\end{proof}

\subsection{The Balmer spectrum of \mdfn{$\cD(\undt)_{\perf}$}}
\label{se:Balmer}

Recall that a tensor-triangulated category $(\cC,\tens)$ has an
associated topological space $\Spec \cC$ called the \dfn{Balmer
  spectrum} of $\cC$ \cite{B}.  The elements of $\Spec \cC$ are the primes
 of $\cC$, i.e.\ proper thick subcategories $\cI$ of $\cC$ having the
properties that
\begin{enumerate}[(i)]
\item If $X\in \cI$ and $Y\in \cC$ then $X\tens Y\in \cI$;
\item If $X,Y\in \cC$ and $X\tens Y\in \cI$ then either $X\in \cI$ or
$Y\in \cI$.
\end{enumerate}
Recall that a subcategory is thick if it is full and closed under
formations of suspensions and desuspensions, retracts, and extensions.  Thick
subcategories satisfying (i) are called \mdfn{tensor ideals}.

The topology on $\Spec \cC$ is
an analog of the Zariski topology from algebraic geometry.  For $X\in
\cC$ define $\Supp X=\{\cP\in \Spec \cC\,|\, X\notin \cP\}$.  Then
$\{\Supp X\,|\, X\in \cC\}$ is a basis for the closed sets in $\Spec
\cC$.

A convenient source of tensor ideals is via annihilators.
If $\cS$ is a set of objects in $\cC$, let
$\Ann \cS=\{X\in \cC\,|\,X\tens Y\he 0\ \text{for all $Y\in \cS$}\}$.
Then $\Ann \cS$ is a tensor ideal.

One more piece of notation: is $\cS$ is a set of objects in $\cC$,
write $\bSigma\cS$ for the closure of $\cS$ under suspension and
desuspension.  

By Proposition~\ref{pr:K-category} the following result can also be
interpreted as computing the Balmer spectrum for the triangulated
category $(K_{b,fg}(\Z/2[C_2]),\tens)$.  In that context, the
computation was independently done by Balmer and Gallauer \cite{BG}.

\begin{thm}
\label{th:Balmer-spectra}
There are only three prime ideals in $(\cD(\undt)_{\perf},\dbox)$:
\begin{itemize}
\item the full
subcategory $\langle A\rangle$ whose objects are the finite direct sums made from the set
$\bSigma\{A_k\,|\,k\geq 0\}$,
\item the full subcategory $\langle B\rangle$
whose objects are the finite direct sums made from the set
$\bSigma\{B_r\,|\,r\geq 0\}$,
\item the full
subcategory $\langle A,B\rangle$
whose objects are the finite direct sums made from the set
 $\bSigma\{A_k,B_r\,|\, k,r\geq 0\}$.
\end{itemize}
The first two are closed points of the Balmer spectrum, whereas the closure of
the third point is the whole space.  This is depicted via the diagram
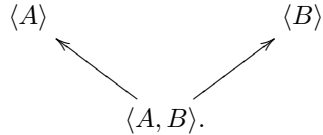
\[ \xymatrix{
\langle A\rangle && \langle B\rangle\\
& \langle A,B\rangle.\ar[ur]\ar[ul]
}
\]
\end{thm}

\begin{proof}
First note that if a tensor ideal contains $H(n)$ then it also contains
$H$, using that $H(n)$ is invertible.  Therefore it contains every
object, and so is not a proper ideal.

We make the following observations:
\begin{itemize}
\item Since $A_k\bbox B_r\he 0$ by Proposition~\ref{pr:dbox}(iii), 
any prime ideal must contain either
$A_k$ or $B_r$.
\item By Proposition~\ref{pr:dbox}(i), if a tensor ideal contains
$A_k$ then it contains all $A_i$ for $i\leq k$.
\item Using the cofiber sequences $A_0\ra A_{n+1}\ra \Sigma A_n$ and induction, any
tensor ideal containing $A_0$ must contain all other $A_n$.
\item By Proposition~\ref{pr:dbox}(vi), if a tensor ideal contains
$B_r$ then it contains all $B_i$ for $i\leq r$.
\item Using the cofiber sequences $B_0\ra B_{n+1}\ra \Sigma^{-1}B_n$
and induction, any tensor ideal containing
$B_0$ contains all other $B_n$.
\end{itemize}
It follows at once from Theorem~\ref{th:complex-splitting}
that the only possible prime ideals are $\langle
A\rangle$, $\langle B\rangle$, and $\langle A,B\rangle$.

We must next check that each of these really is a prime ideal.
Using Theorem~\ref{th:complex-splitting} and Proposition~\ref{pr:dbox}
it follows immediately that
 $\langle A\rangle=\Ann \{B_r\}$, and so is a tensor ideal.
Suppose $X$ and $Y$ are perfect complexes and $X\dbox Y\in \langle
A\rangle$.  By Theorem~\ref{th:complex-splitting} we can write

\[ X\he \bigoplus_{\alpha\in \cJ} A_{\alpha} \,\oplus \, \bigoplus_{\beta
\in \cK} H(\beta) \,\oplus \,\bigoplus_{\gamma\in \cL} B_{\gamma}
\]
where the formula should also have various suspensions on all the
factors, which we have omitted to write.
Similarly, we can write
\[ Y\he \bigoplus_{\alpha' \in \cJ'} A_{\alpha'} \,\oplus \, \bigoplus_{\beta'
\in \cK'} H(\beta') \,\oplus \,\bigoplus_{\gamma'\in \cL'} B_{{\gamma'}}.
\]
Since $X\dbox Y \in \langle A\rangle$ we conclude immediately that
either $\cK$ or $\cK'$ is empty; without loss of generality we assume
$\cK=\emptyset$.

If $\cK'\neq \emptyset$ then we must have $\cL=\emptyset$, else we
have a $B$-type summand in $X\dbox Y$.  But this yields $X\in
\langle A\rangle$.
So now assume $\cK'=\emptyset$.  If $\cL\neq\emptyset$ and $\cL'\neq
\emptyset$ then we again get a $B$-type summand in $X\dbox Y$; so
either $\cL=\emptyset$ or $\cL'=\emptyset$.  But this precisely says
that either $X\in \langle A\rangle$ or $Y\in \langle A\rangle$.  This
completes the proof that $\langle A\rangle$ is prime.

The same style of argument shows that $\langle B\rangle$ is prime.

It remains to prove that $\langle A,B\rangle$ is a prime ideal.
This is largely similar to what we have already done, except the proof
that the subcategory is closed under extensions.  For this, assume that
\[ X=\bigoplus_{\alpha \in \cJ} A_\alpha \,\oplus\, \bigoplus_{\beta \in
  \cK} B_\beta, \qquad
 Y=\bigoplus_{\alpha' \in \cJ'} A_{\alpha'} \,\oplus\, \bigoplus_{\beta' \in
  \cK'} B_{\beta'}
\]
(again, with suspensions on all summands omitted for brevity).
We must check that for any map $f\colon X\ra Y$ the cofiber $\Cof(f)$
is still in $\langle A,B\rangle$.  Here we use
Proposition~\ref{pr:AD} to see that there are no maps from the
$A_k$ to the $B_r$ and vice versa, so that our map $f$ must split as
$f_1\oplus f_2$ where $f_1\colon \bigoplus_\alpha A_\alpha \ra
\bigoplus_{\alpha'} A_{\alpha'}$ and $f_2\colon \bigoplus_\beta
B_\beta \ra \bigoplus_{\beta'} B_{\beta'}$.  Since we have already
proven $\langle A\rangle$ and $\langle B\rangle$ are thick,
we have $\Cof(f_1)\in \langle A\rangle$ and $\Cof(f_2)\in\langle
B\rangle$.
Since $\Cof(f)=\Cof(f_1)\oplus \Cof(f_2)$, we are done.

Finally, it remains to investigate the topology on $\Spec \cD(\undt)$.
For $k\geq 0$ one clearly has that $\Supp(B_k)=\{\langle
A\rangle\}$ and $\Supp(A_k)=\{\langle B\rangle\}$, using
Proposition~\ref{pr:dbox}.  These generate the
closed sets of $\Spec \cD(\undt)$, so the topology is as described in the
statement of the theorem.
\end{proof}

For good measure we also determine the Balmer spectrum of the category
$\cD(\undt)_{bb,fg}$ consisting of complexes which are bounded below
and have finitely-generated projective modules in each degree.  Let
$\langle A,A_\infty \rangle$ denote the full subcategory whose objects are the direct
sums made from the set $\bSigma\{A_k\,|\,0\leq k\leq \infty\}$, with
arbitrary indexing sets but
where the direct sum must be bounded below and finitely-generated in
each degree.  Similarly, define the full subcategories $\langle
B,B_\infty\rangle$, $\langle A,B,A_\infty\rangle$, and $\langle
A,B,B_\infty\rangle$.

\begin{thm}
\label{th:Balmer-spectra-bb}
The Balmer spectrum for $(\cD(\undt)_{bb,fg},\bbox)$ consists of 
exactly four points, separated into two connected components as
depicted here:
\[ \xymatrix{
\langle A,A_\infty \rangle && \langle B,B_\infty \rangle\\
 \langle A,B,A_\infty\rangle\ar[u] &&
 \langle A,B,B_\infty\rangle.\ar[u]
}
\]
The prime ideals in the top row are closed points of the Balmer spectrum, whereas
\[
\overline{\langle A,B,A_\infty\rangle}=\bigl \{ \langle
A,B,A_\infty\rangle,
\langle A,A_\infty\rangle\bigr \}\ \text{and}\
\overline{\langle A,B,B_\infty\rangle}=\bigl \{ \langle
A,B,B_\infty\rangle,
\langle B,B_\infty\rangle\bigr\}.
\]
The inclusion of categories $\cD(\undt)_{\perf} \inc
\cD(\undt)_{bb,fg}$ induces a map of the corresponding Balmer spectra
in the other direction: this is the
quotient map which sends
$\langle A,B,A_\infty\rangle$ and $\langle A,B,B_\infty\rangle$ to the same point,
namely $\langle A,B\rangle$.
\end{thm}

\begin{proof}
The cofiber sequence $H\ra B_\infty \ra \Sigma A_\infty$ shows that if
a tensor ideal contains both
$A_\infty$ and $B_\infty$ then it contains everything.  So no prime
ideal contains both $A_\infty$ and $B_\infty$, and since
$A_\infty \bbox B_\infty \he 0$ it follows that every prime ideal must contain
exactly one of  $A_\infty$ or $B_\infty$.

If $P$ is a prime ideal containing $A_\infty$, then using $A_k\bbox
A_\infty \he A_k$ it must also contain all the $A_k$.  If it contains
any $B_r$ then just as in the proof of Theorem~\ref{th:Balmer-spectra} it must
contain all of the $B_n$.  Since $P$ cannot contain any of the
invertible objects
$H(n)$, this shows that $P$ is either $\langle A,A_\infty\rangle$ or
$\langle A,B,A_\infty\rangle$.

Similar reasoning for the case where $P$ contains $B_\infty$ shows that
the only possible prime ideals are the four from the statement of the
theorem.  It remains to check that these are indeed prime.

One readily checks using Proposition~\ref{pr:dbox} and
Theorem~\ref{th:bd-below-splitting} that
$\Ann(B_\infty)=\langle A,A_\infty\rangle$, so this is a tensor
ideal.  Primality also readily follows from those two results.
Similarly, $\Ann(A_\infty)=\langle B,B_\infty\rangle$ is a prime
ideal.

Similar considerations show that $\langle A,B,A_\infty\rangle$ is a prime
ideal, but here one must work a bit harder to check that it gives a
triangulated subcategory.  If $X$ is a direct sum of shifts of copies of $A_k$,
$B_r$, and $A_\infty$, write $X(A)$ for the direct sum of the pieces
of $A$-type, and $X(B)$ for the direct sum of pieces of $B$-type.
If $S$ is a single strand, of type $A$ or $B$, we claim that any map
$S\ra X$ in $\cD(\undt)$ must factor through $X(A)$ or $X(B)$,
respectively.  This is an easy computation.  This claim then yields that
$\langle A,B,A_\infty\rangle$ is triangulated by the argument from the
proof of Theorem~\ref{th:Balmer-spectra}.

The topology on the Balmer spectrum is readily identified using
Proposition~\ref{pr:dbox}, since one easily computes the support of
any object.  The computation of the map $\Spec \cD(\undt)_{bb,fg} \ra
\Spec \cD(\undt)_{\perf}$ is immediate.
\end{proof}

\section{Topological consequences of the main theorem}
\label{se:top-consequences}

In this section we explain how our description of $\cD(\undt)$ leads
to various topological results about equivariant $H\undt$-modules and bigraded Bredon cohomology with
coefficients in $\undt$.

\medskip

Let $R$ be a Mackey ring.  There is an associated equivariant
Eilenberg--MacLane ring spectrum $HR$.  As explained in
\cite[Corollary 5.2]{Z} for the special case of $R=\und{\Z}$ (though
it works in general), the
general theory developed by Schwede--Shipley in \cite{SS} gives a
Quillen equivalence between the algebraic model category of
$\Ch(R)$ and the topological category of $HR$-modules:

\smallskip

\addtocounter{subsection}{1}
\begin{equation}
\label{eq:QE}
\xymatrix{
\Ch(R) \ar@/^1pc/[r]^\Gamma &  {\ \ \ \ \ \ HR-\Mod}
\ar@/^1pc/[l]^\Psi \ar@{}[l]|\he.
}
\end{equation}

\smallskip
\noindent
In particular, the homotopy category of $\Ch(\undt)$ is
equivalent to the homotopy category of $H\undt-\Mod$.  The Quillen
equivalence is set up so that $H\undt$ and $H\undt \Smash (C_2)_+ $
correspond to $H$ and $F$ respectively.  By examining cell structures for
$S^{p,q}$ one finds that $H\undt \Smash S^{p,q}$ corresponds to
$\Sigma^p H(q)$ under the Quillen equivalence.  It follows that if $M$
is an $H\undt$-module then its bigraded cohomology
$\Hom_{H\undt}(M,\Sigma^{*,*}H\undt)$, as an $\M_2$-module, is
isomorphic to the bigraded cohomology of $\Psi(M)$.  Here we may use
either Mackey functor valued cohomology or consider only the $\cpdot$
side.  Since compact objects in $\cD(\undt)$ are determined by their  bigraded cohomology (see
Corollary~\ref{co:homology-determines}), this implies that when $M$ is
a finite $H\undt$-module $\Psi(M)$
is completely determined by the $\M_2$-module $\Hom_{H\undt}(M,\Sigma^{*,*}H\undt)$.

If $S^k_a$ denotes a $k$-sphere with the antipodal $C_2$-action, then one
readily finds that $H\undt \Smash (S^k_a)_+ $ corresponds to the fundamental complex $A_k$: this can
be done either by an analysis of explicit cell structure on $S^k_a$,
or by computing the bigraded homology of $S^k_a$, and
recognizing it as that of $A_k$.

It remains to identify the $H\undt$-module corresponding to $B_r$.
For this, recall that there is a unique nonzero homotopy class
$\tau\colon H\undt \ra \Sigma^{0,1}H\undt$.   Write $\Cof(\tau)$ for
the homotopy cofiber of this map, and $\Cof(\tau^r)$ for the homotopy
cofiber of $\tau^r\colon H\undt \ra \Sigma^{0,r}H\undt$.  One readily
computes the bigraded cohomology and
observes that it exactly matches the cohomology of $B_{r-1}$.

In light of the Quillen equivalence from (\ref{eq:QE}) we can
 thus reinterpret Corollary \ref{co:perfect-complexes-splitting} as
 follows:

\begin{thm}
\label{th:spectrum-splitting}
Let $M$ be a finite $H\undt$-module.  Then up to weak equivalence $M$
splits as a wedge of bigraded suspensions of $H\undt$, $H\undt \Smash (S^k_a)_+ $, and $\Cof(\tau^{r})$, for various $k \geq 0$ and $r \geq 1$.
\end{thm}

\begin{remark}
We have not proven that the Quillen equivalence from (\ref{eq:QE}) is
symmetric monoidal, but it is.  There is a folklore proof using
$\infty$-categorical techniques, and a proof is forthcoming in work in progress by Drew Heard and the third author.  Therefore the results of Section~\ref{se:Balmer} can be
reinterpreted as 
computing the Balmer spectra for various homotopy categories of 
 $H\undt$-modules.  
\end{remark}

\subsection{Structure theorem for $C_2$-spaces}
The classification of finite $H\undt$-modules in Theorem
\ref{th:spectrum-splitting} implies the following
structure theorem for $RO(C_2)$-graded cohomology of $C_2$-spaces from \cite{M}.

\begin{thm}[C. May]
\label{th:structure-thm}
Let $X$ be a pointed finite $C_2$-CW complex.  Then $H\undt \Smash X$
splits as a wedge of bigraded suspensions of $H\undt$ and $H\undt \Smash (S^k_a)_+ $ for various $k \geq 0$.
\end{thm}

\begin{proof}
From Theorem \ref{th:spectrum-splitting} we know that $H\undt \Smash X$ splits as a wedge of suspensions of $H\undt$, $H\undt \Smash (S^k_a)_+ $, and $\Cof(\tau^{r})$, for various $k$ and $r$.  So it remains to show there cannot be any summands of the form $\Cof(\tau^r)$.
Recall from Lemma 4.3 of \cite{M} that $\rho$-localization of the cohomology of a finite $C_2$-CW complex is
\[
\rho^{-1}H^{*,*}(X) \cong \rho^{-1}H^{*,*}(X^{C_2}) \cong H^{*}_{sing}(X^{C_2}) \otimes \rho^{-1}\M_2.
\]
Notice $\rho^{-1}\M_2$ does not have any $\tau$-torsion and thus neither does $\rho^{-1}\tilde{H}^{*,*}(X) \cong \rho^{-1} \Hom_{H\undt}(H\undt \Smash X,\Sigma^{*,*}H\undt)$, since it is a free $\rho^{-1}\M_2$-module.
However, $\rho$-localization preserves the cohomology of $\Cof(\tau^r)$, which has $\tau$-torsion by construction.  Thus $H\undt \Smash X$ cannot have any wedge summands of the form $\Cof(\tau^r)$ for any $r$.
\end{proof}

\subsection{Toda bracket decomposition of $1$}
A key piece of the proof of the structure theorem for $RO(C_2)$-graded
cohomology of $C_2$-spaces from \cite{M} is the Toda bracket
decomposition of $1$ in $\M_2$:
$\langle \tau,
\theta, \rho \rangle = 1$ with zero indeterminacy.  This Toda bracket
can be witnessed in $Ch(\undt)$ as follows.

Recall the maps representing $\rho$, $\theta$, and $\tau$ that were
given in Figure~\ref{fig:rhothetatau}.  For our present purposes it
will be better to represent $\theta$ as a map $\Sigma H(1)\ra
\Sigma H(-1)$ and $\tau$ as a map $\Sigma H(-1)\ra \Sigma H$,  namely the following (with complexes drawn
horizontally):

\[\xymatrixrowsep{1pc}
\xymatrix{\Sigma H(1)\colon \ar[d]^\theta &  & F \ar[d]^{\id}\ar[r]& H &&
  \Sigma H(-1)\colon \ar[d]^\tau  & H \ar[r] & F\ar[d] \\
\Sigma H(-1)\colon & H\ar[r] & F &&& \Sigma H\colon  & & H  \\
& 2 & 1 & 0 && & 2 & 1
}
\]

\vspace{0.1in}

The left diagram below depicts the composition $\tau\circ \theta\circ
\rho$, where we now switch to drawing complexes vertically:
\begin{center}
\begin{tikzcd}
&                            & H \arrow{d}{p} &   & &&0 \arrow{d}\\
&F \arrow{r}{1} \arrow{d}{p} & F \arrow{r}{p} & H
&& F \arrow{r}{p} \arrow{d}{p} \arrow{ur}{h}  & H           \\
H \arrow{r}{1} & H                           &                &
&& H. \arrow{ur}{}[swap]{h=1}                         &
\end{tikzcd}
\end{center}

\noindent
To construct an element of the Toda bracket $\langle
\tau,\theta,\rho\rangle$ we start by choosing null-homotopies for
$\tau\theta$ and $\theta\rho$.
But $\theta \circ \rho = 0$ on the nose, so we can use the zero
null-homotopy there.
The second composition $\tau \circ \theta$ is not zero in  degree $1$,
but it is null-homotopic via the null homotopy $h$ shown
in the right diagram above.
Finally, forming $h \circ \rho$ gives
\begin{center}
\begin{tikzcd}
               & F  \arrow{d}{1} &  \\
H \arrow{r}{1} & H \arrow{r}{1}  & H,
\end{tikzcd}
\end{center}
and the composite is  the identity $H \to H$.  Thus $1 \in \langle
\tau, \theta, \rho \rangle$.  The indeterminacy of the Toda bracket is
zero  since $\M_2^{0,0}$ contains no $\tau$ or
$\rho$-multiples for degree reasons.  So we may conclude the bracket
identity $\langle \tau, \theta, \rho \rangle =1$.

\begin{remark}
The splitting for finite $H\undt$-modules in Theorem~\ref{th:spectrum-splitting} is stated here in parallel with Theorem~\ref{th:structure-thm}.  In this form, there appears to be a lack of symmetry of the fundamental objects.  However, we may reinterpret $H\undt \Smash (S^k_a)_+$ as a desuspension of $\Cof(\rho^{k+1})$, where $\rho \colon H\undt \ra \Sigma^{1,1}H\undt$.
Thus Theorem~\ref{th:spectrum-splitting} can be restated as: All
finite $H\undt$-modules split as a wedge of suspensions of $H\undt$, $\Cof(\rho^{k})$, and $\Cof(\tau^r)$ for various $k, r \geq 1$.
\end{remark}

\subsection{Classification of finite $H\undl$-modules}
Here we collect some of the analogous results for odd primes.  In particular, we also obtain a classification of finite $H\undl$-modules for $\ell$ an odd prime.  The splitting theorem for $\cD(\undl)_{\perf}$ and the computation of the Balmer spectrum follow immediately from Proposition~\ref{pr:odd-splitting}.

\begin{prop}\label{pr:odd-spectrum}
Let $\ell$ be an odd prime.  Every perfect complex of $\undl$-modules decomposes as a direct sum of shifts of $0 \lra H \lra 0$, $0 \lra S\tht \lra 0$, and the contractible complexes $\bD(H)$ and $\bD(S\tht)$.
\end{prop}

\begin{cor}\label{co:Balmer-odd}
    Let $\ell$ be an odd prime.  The Balmer spectrum of $\cD(\undl)_{\perf}$ is a single point, the zero ideal.
    \end{cor}

\begin{proof}
This follows immediately from Proposition~\ref{pr:odd-spectrum} once one knows that $S_\Theta$ is invertible,
which is true by Proposition~\ref{pr:odd-stuff}.  
\end{proof}

We also get a splitting at the spectrum level.

\begin{prop}\label{pr:spectrum-splitting-odd}
Let $\ell$ be an odd prime and $M$ be a finite $H\undl$-module.  Then
up to weak equivalence $M$ splits as a wedge of bigraded suspensions of $H\undl$.
\end{prop}

\begin{proof}
This follows from Proposition~\ref{pr:odd-spectrum} and the
observation that $0 \ra S\tht \ra 0$ corresponds under the Quillen
equivalence to $\Sigma^{0,1} H\undl$.  To see the latter claim, note that $\Sigma
S\tht$ is quasi-isomorphic to the complex $F\ra H$ (concentrated in
degrees $0$ and $1$), which is $H\undl\Smash S^{1,1}$ under the
Quillen equivalence.  

Alternatively, one may compute
the bigraded cohomology on the $\cpdot$ side of $H\undl$. 
The result is a graded field $\Z/\ell[x,x^{-1}]$ with $x$ in degree $(0,2)$.
\end{proof}

The computation of the Balmer spectrum in this context is then immediate, as every nonzero $H\undl$-module is invertible.

\begin{cor}\label{co:Balmer-odd-1}
Let $\ell$ be an odd prime.  The Balmer spectrum of compact objects in $H\undl\MMod$ is a single point, the zero ideal.
\end{cor}


\section{Proof of the classification theorem for chain complexes}
\label{se:complex-classify-proof}

In this section we will prove Theorem~\ref{th:complex-splitting}.  That is, we will
show that any perfect complex in $Ch(\undt)$ is isomorphic to a direct
sum of ``strands''.  Recall from Section~\ref{se:complexes}
that a complete lists of strands are the
fundamental complexes $A_k$, $B_r$, and
$H(n)$, for various values of $k,r \geq 0$ and $n \in \Z$, together
with the contractible disk complexes $\bD(H)$ and $\bD(F)$.

Our proof of the splitting decomposition involves
changing the basis at various levels in the complex in an algorithmic
fashion.  We work our way from the bottom of the complex up,
but at each stage the algorithm involves possibly changing the basis at some
or all of the lower levels.   So although the proof is constructive, the algorithm can be time-intensive to implement in practice (at least by hand), and there is no simple way in general to look at a complex and know what strands will come out at the end. This differs from Proposition~\ref{pr:Hmod-classify}, for example, where the algorithm in that proof yielded formulas for the number of summands based on properties of the given $H$-module.

\subsection{Proof outline} Our approach is roughly to proceed by induction on the length of the complex.  Suppose the complex is nonzero only in degrees $0$ through $m$.  For the inductive step, we assume the portion of the complex in degrees $0$ to $m-1$ has been split into strands. We consider summands of $H$ and $F$ in the top degree $m$ mapping via the differential to the strands below.  Then, through a series of steps, we split off all subcomplexes having top degree $m$.

We begin by considering any isomorphisms at the top and use these to split off copies of the contractible complexes $\bD(H)$ and $\bD(F)$.  Next we consider the case where copies of $H$ in degree $m$ map nontrivially to various strands below. We use these maps to split off copies of $B_r$. Then we split off any summands of $H(-n)$ for $n > 0$. Turning to the case where we have copies of $F$ in degree $m$ mapping to strands below, we split off any summands of the form $H(n)$. Finally, we split off summands of $A_k$. All other summands of $H$ and $F$ in degree $m$ support trivial maps and thus split off.  Once we have dealt with all the terms in degree $m$ we will be done, since by induction the rest of the complex in lower degrees was already split.

At each stage of the proof, we choose a particular type of strand mapped to by a particular copy of either $H$ or $F$ in degree $m$.  We change the basis of the complex at level $m$ so no other summands in degree $m$ hit the chosen strand.  Then we change the basis of each term in the chosen strand to form a new strand that splits off from the rest. The general rule for the choice of strand at each stage is this: if we are splitting off a type of strand that ends in an $H$ then we choose the {\it shortest\/}
strand of that type, whereas if we are splitting off a type of strand
that ends in an $F$ then we choose the {\it longest\/} strand of that
type.  As the reader will see, these choices guarantee that the evident change
of basis does what we need in order to split off.\medskip

Recall the notion of a basis for a
free $\undt$-module from Definition~\ref{de:basis}.
The following lemma states that certain adjustments (analogous to the
usual column operations of linear algebra)
give a change of basis.  Notice we can use $p^*$ and $p_*$ to mix basis elements from the $\cpdot$ and $\Theta$ sides.

\begin{lemma} \label{le:basis} Let $M$ be a free $\undt$-module  with
$\gamma=\{b_1^{\Theta},\dots,b_m^{\Theta}, b_{m+1}^{\cpdot},\dots, b_{m+n}^{\cpdot}\}$
a chosen basis $M$.  For a fixed choice of $i$ and $j$ with $i\neq j$, each of the following modifications to $\gamma$ yields a new basis:
\begin{enumerate}[(i)]
\item
Replace $b_i^{\Theta}$ with $\tilde{b}_i^{\Theta}=
b_i^{\Theta}+ b_j^{\Theta}$.
\item
Replace $b_i^{\Theta}$ with $\tilde{b}_i^{\Theta}= b_i^\Theta +
t b_{j}^\Theta$.
\item
Replace $b_i^{\Theta}$ with $\tilde{b}_i^{\Theta}= b_i^\Theta +
(1+t)b_{j}^\Theta$.
\item
Replace $b_i^{\Theta}$ with $\tilde{b}_i^{\Theta}=
b_i^{\Theta}+ p^*b_j^{\cpdot}$.
\item
Replace $b_i^{\cpdot}$ with $\tilde{b}_i^{\cpdot}=
b_i^{\cpdot}+ b_j^{\cpdot}$.
 \item
Replace $b_i^{\cpdot}$ with $\tilde{b}_i^{\cpdot}=b_i^{\cpdot}+p_*b_j^{\Theta}$.
\end{enumerate}
\end{lemma}
\begin{proof}
By inspection.
\end{proof}

In the following proof we draw chain complexes vertically, with basis elements appearing as subscripts (so that $F_{a}$ is a copy of $F$ with basis $a^{\Theta}$ for example). To simplify notation, we omit the superscripts $\Theta$ and $\cpdot$ when it
is clear from context where each basis element lives. Our
convention is to only draw arrows for nonzero maps. We also omit
labels of maps whenever there is only one possible nonzero map, such as $H_a \to H_b$. For
convenience, we denote the map $1+t\colon F\to F$ by $u$. This map appears frequently in chain complexes and so whenever we omit the label of a map $F \to F$, we mean that it is $u$.

An example of a complex with this notation is
given in Figure \ref{fig:chainex} on the right.  On the left in Figure \ref{fig:chainex}, the
same complex is drawn using matrix notation (where
matrices act on the left). In the main
proof, we will refer to individual arrows.  For
example, in Figure \ref{fig:chainex} there is a map $u\colon F_{a_1}\to F_{b_1}$ from the top level to the level below.

\begin{figure}[ht]
\begin{tikzpicture}[scale=1.3, >=angle 90]
\node (1) at (0,-1) {$F\oplus H \oplus H \oplus F_{\phantom{-}}$};
\node (2) at (0,-2) {$F\oplus F \oplus H \oplus F_{\phantom{-}}$};
\node (3) at (0,-3) {$H \oplus F \oplus F_{\phantom{-}}$};
\path[->,font=\scriptsize, >=angle 90]
(1) edge node[left]{$\begin{bsmallmatrix} u & p& 0 & 0\\ 0 & 0 & p & 0
\\ 0 & 0 & 1 & p \\ 0 & 0 & 0 & u \end{bsmallmatrix}$} (2)
(2) edge node[left]{$\begin{bsmallmatrix} p & p & 0 & 0\\ 0 & u & 0 & 0
\\ 0 & t & p & t \end{bsmallmatrix}$} (3) ;
\end{tikzpicture}\hspace{2cm}
\begin{tikzpicture}[scale=1.3, >=angle 90]
\node (11) at (0,-1) {$F_{a_1}$};
\node (12) at (1,-1) {$H_{a_2}$};
\node (13) at (2,-1) {$H_{a_3}$};
\node (14) at (3,-1) {$F_{a_4}$};
\node (21) at (0,-2) {$F_{b_1}$};
\node (22) at (1,-2) {$F_{b_2}$};
\node (23) at (2,-2) {$H_{b_3}$};
\node (24) at (3,-2) {$F_{b_4}$};
\node (31) at (0,-3) {$H_{c_1}$};
\node (32) at (1,-3) {$F_{c_2}$};
\node (33) at (2,-3) {$F_{c_3}$};
\path[->,font=\scriptsize, >=angle 90]
(11) edge node[left]{$u$} (21)
(12) edge (21)
(13) edge (23)
(13) edge (22)
(14) edge (23)
(14) edge node[right]{$u$} (24)
(22) edge node[right]{$t$} (33)
(22) edge node[right]{$u$} (32)
(23) edge (33)
(24) edge node[right]{$t$} (33)
(22) edge (31)
(21) edge (31)
;
\end{tikzpicture}
\caption{An example using matrix notation and using basis labels.}
\label{fig:chainex}
\end{figure}
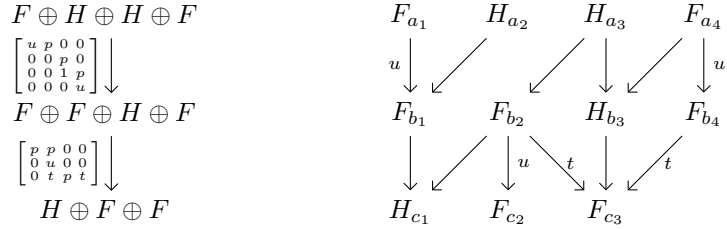

We now give the proof of our main splitting theorem for complexes.

\begin{proof}[Proof of Theorem~\ref{th:complex-splitting}]
Let $C$ be a bounded complex that has a finite direct
sum of copies of $H$ and $F$ in each degree.  We aim to show that $C$ is isomorphic to a direct sum
of shifts of copies of $A_k$, $B_r$, $H(n)$, for various values of $k,r \geq 0$ and $n \in \Z$, and the contractible
complexes $\bD(H)$ and $\bD(F)$.

If $C=0$ we are done, so assume $C\neq 0$. By shifting as needed,
we can assume $C$ is concentrated in degrees $0$ through $d$ with $C_0 \neq 0$.
Let $T_i(C)$ be the truncated complex given by $(T_i(C))_j=C_j$ if $j\leq i$
and $(T_i(C))_j=0$ if $j>i$.  Observe that any choice of basis for
$T_0(C)$ is trivially a decomposition into strands (just direct sums
of $H$ and $A_0=F$).  Assume for $m>1$ that there exists a basis for
$T_{m-1}(C)$ decomposing it as a direct sum of shifts of fundamental
complexes and contractible complexes.  For the inductive step, we will
find a basis for $T_m(C)$ decomposing it into such a direct sum.  In
the process, we will possibly change the basis at level $m$ as well as (potentially all) levels below.

Fix a basis for $T_{m-1}(C)$ so that the truncated complex is a direct
sum of fundamental complexes together with contractible ones.  Choose any basis for $C_m$.  One can visualize $T_{m-1}(C)$ written in strands
with the basis elements in $C_m$ mapping to some combination of those
strands (see the left side of Figure~\ref{fig:step1a} for an example).
We provide an algorithm to adjust the bases at each level so that
$T_{m}(C)$ is written as a direct sum of fundamental complexes and contractible ones. In each step, we explain what to do in general and give an example.  To distinguish these, we use Greek letters for the basis elements in the general case and Roman letters for the examples.\medskip

\noindent \emph{Step 1: Split off disks.} In this step, we will split off disks in $C$ with nonzero terms in degrees $m$ and $m-1$.  Suppose an isomorphism of summands appears in the
map from $C_m$ to $C_{m-1}$. See for example the left side of Figure \ref{fig:step1a} (which has three such
isomorphisms, all from $H$ to $H$). For concreteness, assume the
isomorphism is $\id\colon H\to H$. We will demonstrate how to change bases to split off a shifted copy of $\bD(H)$.  The other cases of isomorphisms $\id,t\colon F\to F$
can be handled similarly and are addressed at the end of this step.  If there are no isomorphisms of summands between degree $m$ and degree $m-1$, proceed to Step 2.

Fix one of the identity maps
$H_\alpha\to H_\beta$ between levels $m$ and $m-1$. We first
adjust the basis at level $m$ so that no other summand in degree $m$ maps to $H_\beta$.  This change of basis proceeds as follows. Any other basis element $\alpha'$ that
maps nontrivially to the submodule generated by $\beta$ is either on the $\cpdot$ side or the $\Theta$ side.  Replace $\alpha'$
with $\alpha'+\alpha$ if $\alpha'$ is on the $\cpdot$ side, and with
$\alpha'+ p^*\alpha$ if $\alpha'$ is on the $\Theta$ side.  The result is again a basis (see parts (v) and (iv) of Lemma \ref{le:basis}).
After changing all such $\alpha'$ in this way, no basis element other than $\alpha$ will map nontrivially to the
submodule generated by $\beta$.

\begin{figure}[ht]
\begin{tikzpicture}[>=angle 90, descr/.style={fill=white}]
\node (11) at (0,-1.5) {$F_{a_1}$};
\node (12) at (1,-1.5) {$H_{a_2}$};
\node (13) at (2.5,-1.5) {$H_{a_3}$};
\node (14) at (4,-1.5) {$F_{a_4}$};
\node (21) at (0,-3) {$F_{b_1}$};
\node (22) at (1,-3) {$F_{b_2}$};
\node (23) at (2,-3) {$H_{b_3}$};
\node (24) at (3,-3) {$H_{b_4}$};
\node (25) at (4,-3) {$F_{b_5}$};
\node (31) at (0,-4) {$\vdots$};
\node (32) at (1,-4) {$\vdots$};
\node (35) at (4,-4) {$\vdots$};
\path[->,font=\scriptsize, >=angle 90]
(11) edge node[descr]{$u$} (21)
(11) edge node[descr]{$u$} (22)
(13) edge (22)
(13) edge  (24)
(13) edge[thick] (23)
(12) edge (23)
(14) edge (23)
(14) edge (24)
(14) edge node[descr]{$u$} (25)
(21) edge (31)
(22) edge (32)
(25) edge (35)
;
\end{tikzpicture}\hspace{.5 in}
\begin{tikzpicture}[>=angle 90, descr/.style={fill=white}]
\node (11) at (-.5,-1.5) {$F_{a_1}$};
\node (12) at (.7,-1.5) {$H_{\color{blue}{a_2+a_3}}$};
\node (13) at (2.2,-1.5) {$H_{a_3}$};
\node (14) at (4,-1.5) {$F_{\color{blue}{a_4+p^*a_3}}$};
\node (21) at (0,-3) {$F_{b_1}$};
\node (22) at (1,-3) {$F_{b_2}$};
\node (23) at (2.1,-3) {$H_{b_3}$};
\node (24) at (3,-3) {$H_{b_4}$};
\node (25) at (4,-3) {$F_{b_5}$};
\node (31) at (0,-4) {$\vdots$};
\node (32) at (1,-4) {$\vdots$};
\node (35) at (4,-4) {$\vdots$};
\path[->,font=\scriptsize, >=angle 90]
(11) edge node[descr]{$u$} (21)
(11) edge node[descr]{$u$} (22)
(14) edge node[descr, xshift=5mm, yshift=2.5mm]{$u$} (22)
(14) edge node[descr]{$u$} (25)
(12) edge (22)
(12) edge (24)
(13) edge (22)
(13) edge[thick] (23)
(13) edge (24)
(21) edge (31)
(22) edge (32)
(25) edge (35)
;
\end{tikzpicture}
\caption{An example of Step 1.}
\label{fig:step1a}
\end{figure}
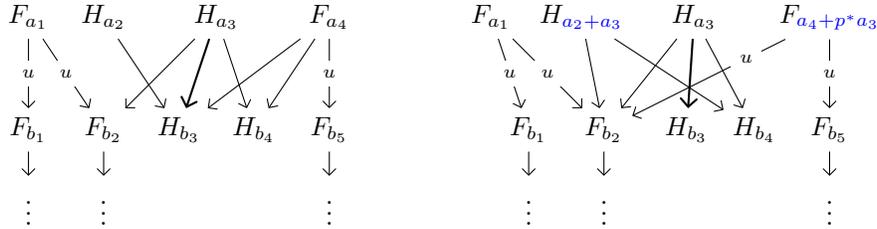

For example, in Figure \ref{fig:step1a} we
choose the identity map $H_{a_3}{\to} H_{b_3}$ sending the element $a_3$ to $b_3$.  As there is also a nonzero map $H_{a_2} \to H_{b_3}$, the basis element $a_2$ at the top is replaced with $a_2+a_3$.  Similarly, the element $a_4$ is replaced with $a_4+p^*a_3$. These new basis elements are depicted on the right of Figure \ref{fig:step1a} in blue.
Notice after the change of basis $H_{a_2+a_3}$ maps to the sum of two strands, as does $F_{a_4+p^*a_3}$, but neither maps to $H_{b_3}$ since we are working over $\Z/2$.  At this point, no other summands at level $m$ map nontrivially to $H_{b_3}$.  However, $a_3$ maps nontrivially to the summands generated by $b_2$ and $b_4$, so we are not yet able to split off a copy of $\bD(H)$.  Thus we consider the basis below degree $m$.

In general (as in the example), it is possible $\alpha$ maps nontrivially to some other summand.  In this situation, we change the basis at level $m-1$, adjusting the target basis element $\beta$ so that this is no longer the case.  Suppose $\beta'$ is another basis element at level $m-1$ such that $\alpha$ maps nontrivially to the summand generated by $\beta'$. Replace $\beta$ with $\beta+\beta'$ if $\beta'$ is on the $\cpdot$ side.
Replace $\beta$ with $\beta+p_*\beta'$ if $\beta'$ is on the $\Theta$ side. Repeat this process until $\alpha$ only maps to a single summand, and call this new basis element $\tilde{\beta}$.  We can now split off the complex $H_\alpha \to H_{\tilde{\beta}}$.

In our example, this change is shown in Figure \ref{fig:step1b}.  We first replace $b_3$ with $b_3+ p_* b_2$, and then replace that with $\tilde{\beta} = b_3+p_* b_2+b_4$.  Now $H_{a_3} \to H_{b_3+p_* b_2+b_4}$ splits off.  Observe that we have effectively chosen the diagonal basis element hit by $a_3$.  That is, we replace $b_3$ with the sum of the elements that are hit by $a_3$ (using $p_*$ as necessary, since all these elements must lie on the $\cpdot$ side).

Note the change of basis at level $m-1$ has not destroyed the
decomposition into strands in lower degrees.  The fact that $H_\alpha\ra H_\beta$ was
the identity and the criteria that $d^2=0$ in the chain complex, means the original strand
involving $H_{\beta}$ could not have had any lower degree terms.  The same is true of $H_{\tilde{\beta}}$.  Thus we have indeed split off a copy of $\bD(H)$.

\begin{figure}[ht]
\begin{tikzpicture}[>=angle 90, descr/.style={fill=white}]
\node (11) at (-.5,-1.5) {$F_{a_1}$};
\node (12) at (.7,-1.5) {$H_{a_2+a_3}$};
\node (13) at (2,-1.5) {$H_{a_3}$};
\node (14) at (4,-1.5) {$F_{a_4+p^*a_3}$};
\node (21) at (-.5,-3) {$F_{b_1}$};
\node (22) at (.7,-3) {$F_{b_2}$};
\node (23) at (2.4,-3) {$H_{\color{blue}{b_3+p_*b_2+b_4}}$};
\node (24) at (4,-3) {$H_{b_4}$};
\node (25) at (5,-3) {$F_{b_5}$};
\node (31) at (-.5,-4) {$\vdots$};
\node (32) at (.7,-4) {$\vdots$};
\node (35) at (5,-4) {$\vdots$};
\path[->,font=\scriptsize, >=angle 90]
(11) edge node[descr]{$u$} (21)
(11) edge node[descr]{$u$} (22)
(14) edge node[descr, xshift=5mm, yshift=2.5mm]{$u$} (22)
(13) edge[thick] (23)
(14) edge (24)
(14) edge node[descr]{$u$} (25)
(21) edge (31)
(22) edge (32)
(25) edge (35)
(12) edge (22)
(12) edge (24)
;
\end{tikzpicture}
\caption{The example for Step 1, continued.}
\label{fig:step1b}
\end{figure}
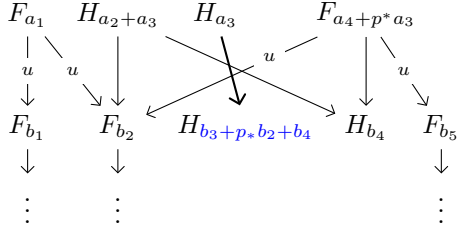

If initially we had instead chosen the identity map $F_{\alpha} \ra F_{\beta}$ the process to split off a copy of $\bD(F)$ would be analogous, except we would replace basis elements $\alpha'$ in level $m$ with either $\alpha'+\alpha$, $\alpha'+t\alpha$, $\alpha' + u\alpha$, or $\alpha' +p_*\alpha$. At level $m-1$, we similarly would add $\beta'$, $t\beta'$, $u\beta'$, or $p^*\beta'$ to $\beta$.
If we had instead chosen $t\colon F_{\alpha} \to F_{\beta}$, we could simply replace $\alpha$ with $t\alpha$ and then apply the steps for $id\colon F_{t\alpha}\to F_{\beta}$.

Continue this process until all isomorphisms from level $m$ to level $m-1$ have been split off as disks. \medskip

Now we turn to decomposing the remaining complex where these top-level disks have been split off. We abuse notation and again
refer to this complex as $C$.  Notice $C$ may still have contractible complexes as summands in lower degrees, but these can essentially be ignored.

Having completed Step 1, we may assume $C$ has no isomorphisms from degree $m$ to degree $m-1$.  We consider nonzero maps out of a copy of $H$ in degree $m$, and since there are no maps of the form $H \to H$, we suppose there is a map of the form $p\colon H \ra F$.  This type of map appears at the top of both $B_r$ strands and $H(-n)$ strands, for $n >0$.  In general, to split off a strand ending in $H$ we choose the shortest strand, and to split off a strand ending in $F$ we choose the longest.  In the next two steps, we will split off the $B_r$ summands, shortest strands first, and then the $H(-n)$ summands, longest first.  If $C$ has no maps of the form $p\colon H \ra F$ at this level, proceed to Step 4. \medskip

\noindent\emph{Step 2: Split off $B_r$ summands.} Assume there exists a map $p\colon H_{\alpha} \to F_{\beta}$ from
level $m$ to level $m-1$, and further assume that the strand in
$T_{m-1}(C)$ with $F_{\beta}$ at the top is isomorphic to a shift of
$H(i)$ for some $i>0$.  If there is no such strand, proceed to Step 3.  If there is such a strand, amongst all such $\alpha, \beta$ pairs, select one so the complex beginning with $F_{\beta}$ is isomorphic to a shifted copy of $H(i)$ for the smallest possible $i$. That is, choose the shortest strand that ends in an $H$.  An example is shown below on the
left in Figure \ref{fig:step2a} with $\alpha = a_1$ and $\beta=b_3$, where the chosen strand is a shifted copy of $H(2)$.

\begin{figure}[ht]
\begin{tikzpicture}[>=angle 90]
\node (04) at (2.3,0.2) {$F_{a_2}$};
\node (02) at (1,0.2) {$H_{a_1}$};
\node (11) at (0,-1) {$F_{b_1}$};
\node (12) at (1,-1) {$F_{b_2}$};
\node (13) at (2,-1) {$F_{b_3}$};
\node (14) at (3,-1) {$F_{b_4}$};
\node (21) at (0,-2) {$F_{c_1}$};
\node (22) at (1,-2) {$F_{c_2}$};
\node (23) at (2,-2) {$F_{c_3}$};
\node (24) at (3,-2) {$F_{c_4}$};
\node (31) at (0,-3) {$H_{d_1}$};
\node (33) at (2,-3) {$H_{d_3}$};
\node (34) at (3,-3) {$F_{d_4}$};
\node (44) at (3,-4) {$H_{e}$};
\path[->,>=angle 90]
(04) edge (13)
(04) edge (14)
(02) edge (11)
(02) edge (12)
(02) edge[thick] (13)
(02) edge (14)
(11) edge (21)
(21) edge (31)
(12) edge (22)
(13) edge[thick] (23)
(23) edge[thick] (33)
(14) edge (24)
(24) edge (34)
(34) edge (44)
;
\end{tikzpicture}\hspace{0.5 in}
\begin{tikzpicture}[>=angle 90]
\node (04) at (2.3,0.2) {$F_{\color{blue}{a_2+p^*a_1}}$};
\node (02) at (1,0.2) {$H_{a_1}$};
\node (11) at (0,-1) {$F_{b_1}$};
\node (12) at (1,-1) {$F_{b_2}$};
\node (13) at (2,-1) {$F_{b_3}$};
\node (14) at (3,-1) {$F_{b_4}$};
\node (21) at (0,-2) {$F_{c_1}$};
\node (22) at (1,-2) {$F_{c_2}$};
\node (23) at (2,-2) {$F_{c_3}$};
\node (24) at (3,-2) {$F_{c_4}$};
\node (31) at (0,-3) {$H_{d_1}$};
\node (33) at (2,-3) {$H_{d_3}$};
\node (34) at (3,-3) {$F_{d_4}$};
\node (44) at (3,-4) {$H_{e}$};
\path[->,>=angle 90]
(04) edge (11)
(04) edge (12)
(02) edge (11)
(02) edge (12)
(02) edge[thick] (13)
(02) edge (14)
(11) edge (21)
(21) edge (31)
(12) edge (22)
(13) edge[thick] (23)
(23) edge[thick] (33)
(14) edge (24)
(24) edge (34)
(34) edge (44)
;
\end{tikzpicture}
\caption{The example for Step 2.}
\label{fig:step2a}
\end{figure}
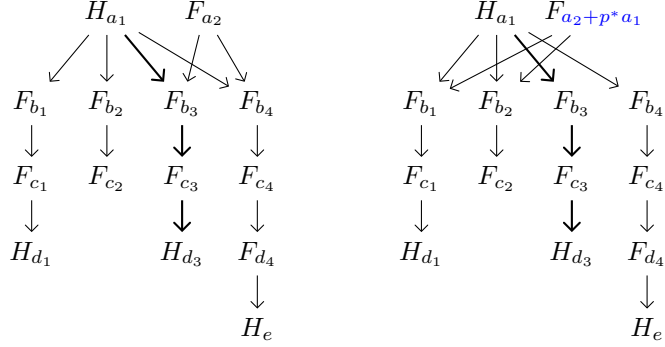

Our goal is to change the bases in order to split off a shifted copy
of $B_{r}$ with $r=i-1$. To do so, we first adjust the basis at level $m$ so that $H_{\alpha}$ is the only term hitting our chosen strand. If
$\alpha'$ is some other basis element that maps nontrivially to the
summand generated by $\beta$, replace $\alpha'$ with either
$\alpha'+\alpha$ or $\alpha'+p^*\alpha$. Continue in this way until $H_{\alpha}$ is the only term with a nontrivial map to the chosen strand.  This change is made in the
example on the right in Figure \ref{fig:step2a}.

Next we change the basis of our chosen strand at every level below, starting with level $m-1$.
Recall the element $\beta$
generates a copy of $F$. Replace $\beta$ with the diagonal element, i.e.\ the sum of all basis
elements whose summands are mapped to nontrivially by $H_{\alpha}$. Note that
because there are no longer any isomorphisms from degree $m$ to $m-1$, each of these basis elements
must generate a copy of $F$. Thus all these elements are on the $\Theta$ side and
the sum makes sense (there is no need to involve $p^*$ or $p_*$). Call this new diagonal basis element $\tilde{\beta}$. Observe that $H_{\alpha}$ now maps nontrivally only to $F_{\tilde{\beta}}$. We
 adjust the basis at level $m-2$ in a similar fashion, again
replacing the basis element in the chosen strand with a sum of basis elements. Eventually we will
encounter a summand of the form $H_{\delta}$ at the bottom of the chosen strand. At this level, we
replace the basis element $\delta$ with the diagonal element, but adjust as necessary.  We take the sum of all the relevant copies of basis
elements on the $\cpdot$ side (including $\delta$) and $p_*$ applied to relevant basis elements
on the $\Theta$ side. See Figure \ref{fig:step2b} for clarification.

\begin{figure}[ht]
\begin{tikzpicture}[>=angle 90]
\node (04) at (0,0.2) {$F_{a_2+p^*a_1}$};
\node (02) at (2,0.2) {$H_{a_1}$};
\node (11) at (0,-1) {$F_{b_1}$};
\node (12) at (1,-1) {$F_{b_2}$};
\node (13) at (2.7,-1) {$F_{\color{blue}{b_3+b_1+b_2+b_4}}$};
\node (14) at (4.3,-1) {$F_{b_4}$};
\node (21) at (0,-2) {$F_{c_1}$};
\node (22) at (1,-2) {$F_{c_2}$};
\node (23) at (2.7,-2) {$F_{\color{blue}{c_3+c_1+c_2+c_4}}$};
\node (24) at (4.3,-2) {$F_{c_4}$};
\node (31) at (0,-3) {$H_{d_1}$};
\node (33) at (2.7,-3) {$H_{\color{blue}{d_3+d_1+p_*d_4}}$};
\node (34) at (4.3,-3) {$F_{d_4}$};
\node (44) at (4.3,-4) {$H_{e}$};
\path[->,>=angle 90]
(04) edge (11)
(04) edge (12)
(02) edge[thick] (13)
(11) edge (21)
(21) edge (31)
(12) edge (22)
(13) edge[thick] (23)
(23) edge[thick] (33)
(14) edge (24)
(24) edge (34)
(34) edge (44)
;
\end{tikzpicture}
\caption{The example for Step 2, continued.}
\label{fig:step2b}
\end{figure}
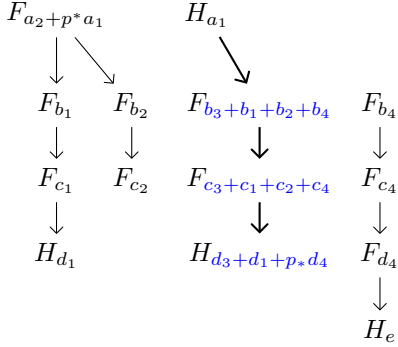

In the example, note that $d_3+d_1+p_*d_4$ maps to zero because
$p\colon F_{d_4}\to H_e$ is zero on the $\cpdot$ side, so we can now split off a
complex of the form $B_1$. The general case works as in the example:
after replacing the basis element $\delta$ with the sum, we can split off a shift of $B_r$ with $r=i-1$.

Continue this process until there are no more maps $p\colon H_{\alpha} \to F_{\beta}$ from
level $m$ to level $m-1$, where $F_{\beta}$ is the top of a shifted copy of $H(i)$ for any $i >0$. Afterward, if there are no more maps $p\colon H \to F$ at this level then skip to Step 4.  Otherwise proceed to Step 3 to split off shifted copies of $H(-n)$. \medskip

\noindent \emph{Step 3: Split off $H(-n)$ summands.} Having completed Steps 1 and 2, and again abusing notation, we may assume that any maps $p\colon H \to F$ from level $m$ to level $m-1$ in $C$ must hit strands made entirely of copies of $F$. So assume there is a map of the
form $p\colon H_\alpha \to F_\beta$ from level $m$ to level $m-1$ where in the truncated complex
$T_{m-1}(C)$ the summand with $F_\beta$ at the top is of the form
$A_i$ for some $i \geq 0$. In this step, choose the pair $\alpha$, $\beta$
such that the summand $A_i$ beginning with $F_\beta$ has the
\emph{largest} possible value of $i$ (in other words, $F_\beta$ begins
the longest possible strand). Our goal is to use this strand to split off a shifted copy of $H(-n)$ with $n=i+1$.

We can adjust the basis at level $m$ as in Step 2. We do not repeat these details and instead begin by assuming there are no other basis elements from level $m$ that map nontrivially to $F_\beta$. We next change the basis at level $m-1$ as in Step 2. That is, replace $\beta$ with the sum of all basis elements whose summands are mapped to nontrivially by $\alpha$. If $i> 0$, make the same adjustment at level $m-2$ and continue making this change until reaching level $m-(i+1)$. An example is provided in Figure \ref{fig:step3} for clarification.

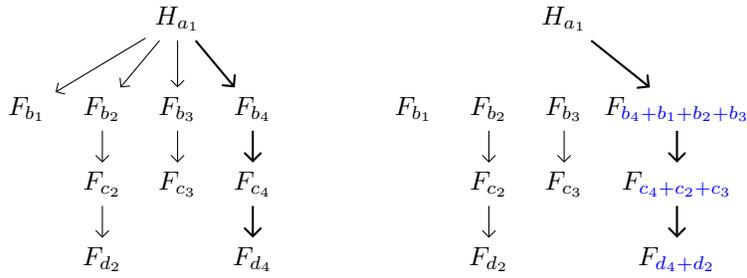
\begin{figure}[ht]
\begin{tikzpicture}[>=angle 90]
\node (02) at (2,0.2) {$H_{a_1}$};
\node (11) at (0,-1) {$F_{b_1}$};
\node (12) at (1,-1) {$F_{b_2}$};
\node (13) at (2,-1) {$F_{b_3}$};
\node (14) at (3,-1) {$F_{b_4}$};
\node (22) at (1,-2) {$F_{c_2}$};
\node (23) at (2,-2) {$F_{c_3}$};
\node (24) at (3,-2) {$F_{c_4}$};
\node (32) at (1,-3) {$F_{d_2}$};
\node (34) at (3,-3) {$F_{d_4}$};
\path[->,>=angle 90]
(02) edge (11)
(02) edge (12)
(02) edge (13)
(02) edge[thick] (14)
(12) edge (22)
(22) edge (32)
(13) edge (23)
(14) edge[thick] (24)
(24) edge[thick] (34)
;
\end{tikzpicture}\hspace{0.5in}
\begin{tikzpicture}[>=angle 90]
\node (02) at (2,0.2) {$H_{a_1}$};
\node (11) at (0,-1) {$F_{b_1}$};
\node (12) at (1,-1) {$F_{b_2}$};
\node (13) at (2,-1) {$F_{b_3}$};
\node (14) at (3.5,-1) {$F_{\color{blue}{b_4+b_1+b_2+b_3}}$};
\node (22) at (1,-2) {$F_{c_2}$};
\node (23) at (2,-2) {$F_{c_3}$};
\node (24) at (3.5,-2) {$F_{\color{blue}{c_4+c_2+c_3}}$};
\node (32) at (1,-3) {$F_{d_2}$};
\node (34) at (3.5,-3) {$F_{\color{blue}{d_4+d_2}}$};
\path[->,>=angle 90]
(02) edge[thick] (14)
(12) edge (22)
(22) edge (32)
(13) edge (23)
(14) edge[thick] (24)
(24) edge[thick] (34)
;
\end{tikzpicture}
\caption{The example for Step 3.}
\label{fig:step3}
\end{figure}

Now split off a complex that is a shift of $H(-(i+1))$. Continue this process until there are no longer any maps of the form $p\colon H \to F$ from degree $m$ to $m-1$. \medskip 

\noindent \emph{Step 4: Split off $H(n)$ summands.} Suppose we have completed Steps 1, 2, and 3. If there are any nonzero maps remaining from level $m$ to level $m-1$, then they must be of the form $u:F\to F$. We now consider basis pairs $\alpha, \beta$ such that $u:F_\alpha \to F_\beta$ appears in $C$ from level $m$ to $m-1$, and the summand beginning with $F_\beta$ in $T_{m-1}(C)$ ends in $H$.
That is, the summand beginning with $F_\beta$ is isomorphic to $H(i)$ for some $i>0$. As in Step 2, choose the pair $\alpha$, $\beta$ such that $F_\beta$ begins a copy of $H(i)$ for the smallest possible $i$. Change bases as in Step 2 to split off a copy of $H(n)$ with $n=i+1$ from $T_m(C)$.

Repeat the above until there are no such $\alpha$, $\beta$ pairs remaining.\medskip

\noindent \emph{Step 5: Split off $A_k$ summands.} Assume we have completed Steps 1 through 4. Again, if there are any nonzero maps remaining from level $m$ to level $m-1$, they must be of the form $u:F_\alpha \to F_\beta$. Having completed Step 4, it must be that the summand in $T_{m-1}(C)$ that begins with $F_\beta$ is of the form $A_i$ for some $i \geq 0$. As in Step 3, choose the $\alpha, \beta$ pair such that $F_\beta$ begins the longest possible $A_i$. Change bases as in Step 3 to split off a copy of $A_{k}$ with $k=i+1$ from $T_m(C)$.

Repeat the above until there are no such $\alpha$, $\beta$ pairs remaining.\medskip

After completing these steps, all remaining maps from level $m$ to level $m-1$
will be zero.  Thus any remaining $H$ or $F$ summands split off as
chain complexes with zero differentials. We have successfully
decomposed $T_{m}(C)$ as a sum of fundamental chain complexes. We can now
repeat the above steps to split $T_{m+1}(C)$, and so on. Since $C$ is
a bounded complex, there is a large enough $d$ so that $T_d(C)=C$.  Thus this process will eventually terminate to give a splitting of $C$.
\end{proof}


\section{An algebraic version of Kronholm's theorem}
\label{se:Kronholm}

Kronholm's freeness theorem, which first appeared in \cite{K}, states the $RO(C_2)$-graded Bredon cohomology with $\undt$-coefficients of any
finite $\Rep(C_2)$-complex is free as a module over the cohomology of a
point $\M_2$. A mistake in the
original proof was fixed in \cite{HM}, which also expanded the result
to Rep$(C_2)$-complexes of finite type. The proof of the freeness theorem involved delicate arguments about extensions of $\M_2$-modules. In this section, we give two alternate proofs.  The first uses $\tau$-localization to quickly deduce freeness as suggested by the referee. The second uses the splitting algorithm from Section~\ref{se:complex-classify-proof} to describe the cohomology explicitly.  The first has the advantage of being short and clear, while the second perspective helps clarify an important phenomenon observed by Kronholm now known as ``Kronholm shifts''.  This is a phenomenon in which the representation cell structure of a $\Rep(C_2)$-complex determines the bidegrees of the free generators in cohomology, up to some shifting of bidegrees.

\medskip

We begin by defining an analog of $\Rep(C_2)$-cells and
$\Rep(C_2)$-complexes in $\Ch(\undt)$. This is somewhat complicated by the fact that complexes in $\Ch(\undt)$ correspond to general $H\undt$-modules, not just those of the form $H\undt$ smashed with a pointed $C_2$-space.  Recall that a $\Rep(G)$-complex is a particular type of $G$-space built by attaching $\Rep(G)$-cells of the form $D(V)$ so that each filtration quotient looks like a wedge of representation spheres $S^V$.  In this section, we will focus on $\Rep(G)$-complexes and not consider $G$-CW complexes, which are built by attaching orbit cells $G/H \times D^m$.

We define a \dfn{representation cell} to be a 
fundamental complex of the form $\Sigma^{m}H(q)$ where $0\leq q \leq
m$. As discussed in Section~\ref{se:top-consequences}, $\Sigma^{m}H(q)$ corresponds to $H\undt \Smash S^{m,q}$.  We require $0 \leq q \leq m$ so that $S^{m,q}$ is an actual (rather than virtual) representation sphere.  Recall $\Sigma^{m}H(q)$ is the complex
\[
\Sigma^m H(q)\colon \qquad\qquad
\xymatrix{
0\ar[r] & F \ar[r]^{u} &F \ar[r]^{u} &\cdots \ar[r]^{u} &F \ar[r]^{p} &H \ar[r]
&0,
}
\]
where the leftmost $F$ is in degree $m$ and the $H$ is in degree $m-q$.  We refer to $m$ as the \dfn{topological dimension}, $q$ as
the \dfn{weight}, and $m-q$ as the \dfn{coweight} of the
representation cell.  For a representation cell $W$ we write
$\tp(W)$, $\wt(W)$, and $\cowt(W)$ for the topological dimension,
weight, and coweight, respectively.

The complex for a representation cell $W$ is zero in degrees above $\tp(W)$ and
in degrees below $\cowt(W)$.  Colloquially, $\tp(W)$ is the degree of
the ``top $F$'' and $\cowt(W)$ is the degree of the ``bottom $H$''.  The weight $\wt(W)$ is the ``length'' of
the fundamental complex, where length is one less than the number
of nonzero terms.  The weight is also the number of $F$'s in the strand.  It is useful to keep these interpretations of the
three invariants in mind while reading the arguments in this section.

We next want to define $\Rep(C_2)$-chain complexes so that under the
equivalence \ref{eq:QE} we have
\[
\Rep(C_2)\text{-chain complexes} \longleftrightarrow H\undt \Smash \Rep(C_2)\text{-complexes}.  
\]
If a space $K$ is formed by attaching an $(m,q)$-cell $D(\R^{m,q})$ to the
space $L$, then we have a cofiber sequence $L\to K \to
S^{m,q}$. Desuspending in the homotopy category then gives $K\simeq
\Cof(S^{m-1,q}\to L)$.  Thus we are led to define \mdfn{attaching an
  $(m,q)$-representation cell} to a chain complex of
$\undt$-modules $Y$ to be taking the cofiber of a map $\Sigma^{m-1}H(q)\to
Y$ in $\Ch(\undt)$.  We will show that (under well-controlled circumstances) the cofiber splits as a direct sum of representation cells, and thus the bigraded cohomology is a free $\M_2$-module by Proposition~\ref{pr:cohom-rings}.

The following example illustrates the freeness theorem and the shifting phenomenon in the context of chain complexes.  The reader is invited to compare it with the extension of $\M_2$-modules from Example 3.2 in \cite{HM}.  In the chain complex setting, a change of basis immediately solves the extension problem.

\begin{example}
The projective space $\R P^2_{tw} = \mathbb{P}(\R^{3,1})$ sits in a
cofiber sequence of the form $S^{1,0} \to \R P^2_{tw} \to S^{2,2}$.
Desuspending in the stable homotopy category, we may view $\R P^2_{tw}$ as $\Cof(S^{1,2} \to S^{1,0})$. Smashing with $H\undt$ and translating to chain complexes, we consider the corresponding complex $X = \Cof(\Sigma^{1}H(2) \to \Sigma^{1}H)$. 
The map $f \colon \Sigma^{1}H(2) \to \Sigma^{1}H$ is pictured on the
left in Figure~\ref{fig:rp_free}, where complexes are drawn
vertically.  If $f$ is not null-homotopic (as turns out to be the case here), it must be given by $p\colon F \to H$ in degree $1$. 
The cofiber $X$ is then pictured in the middle of
Figure~\ref{fig:rp_free}.  Applying the change of basis algorithm from
Section~\ref{se:complex-classify-proof},  Step 4 of the algorithm
splits the complex into the two summands as pictured on the right.
Thus we find $X \simeq \Sigma^{1}H(1) \oplus \Sigma^{2}H(1)$ (this
can also be deduced from Lemma~\ref{le:cof-theta} by identifying
$f\Smash \id_{H(-2)}$ with the map $\Sigma \theta$).  

\begin{figure}[ht]
\begin{tikzcd}[column sep=scriptsize, row sep=scriptsize]
 \phantom{F} &  \\
F \arrow{d}[swap]{u} \arrow{r}{f} & H \\
F \arrow{d}[swap]{p} &\\
H & 
\end{tikzcd} \hspace{2cm}
\begin{tikzcd}[column sep=tiny, row sep=scriptsize]
  F \arrow{d}[swap]{u} \arrow{dr}{p} & \\
  F \arrow{d}[swap]{p} & H\\
  H & \\
  \phantom{H}
  \end{tikzcd}
  \hspace{2cm}
  \begin{tikzcd}[column sep=tiny, row sep=scriptsize]
     & F \arrow{d}{p} \\
    F \arrow{d}[swap]{p} & H\\
    H & \\
    \phantom{H} &
    \end{tikzcd}
    \caption{Freeness theorem and shifts for $H\undt \Smash \R P^2_{tw}$.}
    \label{fig:rp_free}
\end{figure}

Notice the splitting algorithm decomposes $X$ as a direct sum of representation cells $\Sigma^{1}H(1)$ and $\Sigma^2H(1)$.  So the (reduced) bigraded cohomology of $X$ is a free $\M_2$-module generated in bidegrees $(1,1)$ and $(2,1)$.  Translating to topology we have
\[
H\undt \Smash \R P^2_{tw} \simeq H\undt \Smash (S^{1,1} \Wedge S^{2,1}).
\]

Moreover, back in chain complexes, the splitting algorithm has effectively transferred a copy of $F$ onto the second chain complex.  This is precisely how the shifting phenomenon observed by Kronholm manifests in chain complexes.  The original spheres $S^{1,0}$ and $S^{2,2}$ give rise to free $\M_2$ generators in the same topological dimensions but with shifted weights: the first generator shifts up one in weight and the second generator shifts down one.  These weight shifts occur because a single $F$ ``moved'' to the second complex during the change of basis.  Notice that no copies of $F$ appeared or disappeared, so the total weight is preserved.
\end{example}

As in the previous example, more general Kronholm shifts will be determined by copies of $F$ moving onto different strands and the total weight will be preserved. \medskip

Before proceeding to the proof in the general case, we need one more restriction on the complexes.  Without this restriction, it is easy to construct a chain complex out of representation cells that cannot correspond to a space and hence not a $\Rep(C_2)$-complex.

\begin{example}
 Consider the nontrivial map $f\colon \Sigma^1 H \to \Sigma^1 H(1)$.  We can compute $\Cof(\Sigma^1 H \to \Sigma^1 H(1)) \simeq \Sigma^2 B_0$ as depicted in Figure~\ref{fig:tau_not_free}.  However, by the structure theorem (Theorem~\ref{th:structure-thm}), the bigraded cohomology of $\Sigma^2B_0$ cannot be the bigraded cohomology of a $C_2$-space.

\begin{figure}[ht]
  \begin{tikzcd}[column sep=scriptsize, row sep=scriptsize]
   \phantom{F} &  \\
  H \arrow{r}{f} & F \arrow{d}[swap]{p} \\
   & H 
  \end{tikzcd} \hspace{2cm}
  \begin{tikzcd}[column sep=scriptsize, row sep=scriptsize]
    H \arrow{d}[swap]{p} & \\
    F \arrow{d}[swap]{p} & \\
    H & 
    \end{tikzcd}
      \caption{Freeness theorem fails for non-spaces.}
      \label{fig:tau_not_free}
  \end{figure}
\end{example}

To exclude these sorts of chain complexes, we make the following definition.

\begin{defn}
A map $X\ra Y$ in $\cD(\undt)$ is \dfn{spacelike} if the cofiber does
not have any $B$-type summands in its decomposition (see Theorem~\ref{th:complex-splitting}).
\end{defn}

We are now ready to define the appropriate analog of $\Rep(C_2)$-complexes.

\begin{defn}\label{complexdef}
A {\mdfn{$\Rep(C_2)$-chain complex}} is a chain complex $X$ with
a filtration $X_0\subseteq X_1\subseteq\dots \subseteq X$, where $X_0$ is either $0$ or a direct sum of copies of $H$, and $X_m$ is formed
by attaching $m$-cells to $X_{m-1}$. That is,
	\[
	X_m=\Cof\pars{\bigoplus_i\Sigma^{m-1}H(q_i)\to X_{m-1}}
	\]
where $0\leq q_i \leq m$ for each $i$. Furthermore, we require that all of the
attaching maps are spacelike.
\end{defn}

We are now ready to state a version of Kronholm's theorem in $\cD(\undt)$.

\begin{thm}\label{th:Kronholm}
Suppose $X$ is a $\Rep(C_2)$-chain complex built from finitely-many cells. Then $X$ is quasi-isomorphic to a direct sum of complexes of the form $\Sigma^{k_i}H(r_i)$ where $0\leq r_i \leq k_i$ for all $i$.
\end{thm}

Before giving the proof, we note the topological version of
Kronholm's theorem is an immediate corollary.

\begin{cor}[Kronholm; Hogle--May]
Let $L$ be a pointed $C_2$-space with the structure of a finite Rep$(C_2)$-complex.  Then
$\tilde{H}^{*,*}(L;\undt)$ is free as an $\M_2$-module.
\end{cor}

\begin{proof}
  The object of $\cD(\undt)$ corresponding to $H\undt
  \Smash \Sigma^\infty L$ can be represented by  a Rep$(C_2)$-chain complex $X$ satisfying Definition~\ref{complexdef}.  So Theorem~\ref{th:Kronholm}
  applies and $X$ is quasi-isomorphic to $\oplus_i \Sigma^{k_i}H(r_i)$ where $0\leq r_i \leq k_i$ for all $i$.  By Proposition~\ref{pr:cohom-rings}, the bigraded cohomology of $X$ is a free $\M_2$-module and thus so is the bigraded cohomology of the $C_2$-space $L$.
  \end{proof}

Now we turn to the proofs of Theorem~\ref{th:Kronholm}. We first outline a short proof of the freeness theorem via localization.  
We thank the referee for suggesting this argument.

\begin{proof}[Proof of freeness in Theorem~\ref{th:Kronholm} via localization] Let $X$ be a finite $\Rep(C_2)$-chain complex. Theorem~\ref{th:complex-splitting} implies
that up to quasi-isomorphism $X$ is the sum of fundamental complexes. The spacelike assumptions rules out the
appearance of any $B$-type summands.  To rule out $A$-type summands, it is enough to
rule out any $\rho$-torsion in the $\tau$-localization of the cohomology. This can
be done by inducting on the dimension of the complex and using
that $\M_2[\tau^{-1}]$ is a graded PID. We outline this argument in the paragraph below.

The base case is immediate because $X_0$ is just a direct sum of copies of $H$. In the inductive step we build $X_m$ from $X_{m-1}$ by attaching $m$-cells. The attaching map induces a map on the $\tau$-localized cohomology, which leads us to analyze the map of free $\M_2[\tau^{-1}]$-modules
\[
 \tau^{-1}H^{*,*}(X_{m-1})\to \bigoplus_{\text{\# of $m$-cells}} \Sigma^{m-1}\M_2[\tau^{-1}] .
\]
The module $\tau^{-1}H^{*,*}(X_{m-1})$ will be a direct sum of free $\M_2[\tau^{-1}]$-modules with generators in topological degrees strictly less than $m$.  For degree reasons, the only summands on which this map can be nonzero are the ones generated in topological degree $m-1$. On these summands, analyzing the map of $\M_2[\tau^{-1}]$-modules reduces to analyzing a map of $\F_2$-vector spaces. 

By making a change of basis if necessary, there are only two options for a degree $m-1$ summand in the domain: the summand is in the kernel, or it is mapped isomorphically onto a degree $m-1$ summand in the codomain. Thus, using the long exact sequence on cohomology, $\tau^{-1}H^{*,*}(X_m)$ will again be a direct sum of free $\M_2[\tau^{-1}]$-modules whose generators are in topological degrees $m$ or less. In particular, $\tau^{-1}H^{*,*}(X_m)$ has no $\rho$-torsion and so neither does $H^{*,*}(X_m)$. By induction, $X$ has no $A$-type summands and is thus quasi-isomorphic to a sum of terms of the form $\Sigma^{k_i}H(r_i)$. \end{proof}

\begin{remark}
The above argument has the advantage of being short and 
clear.  However, it obscures 
how the chain complexes are glued together and how the weights shift. That is, while the localization argument guarantees the 
 result will be a direct sum of strands of the form 
 $\Sigma^{k_i}H(r_i)$, it does not give us a way to predict the values of 
 $k_i$ or $r_i$ based on our starting cells.
 
 We thus give a second proof 
 that sheds some light on the nature of the Kronholm
 shifts, by which the bigradings in the original cell structure shift
 around to become the bigradings in the splitting decomposition.  We certainly do not claim the following proof is easier or clearer, but rather that it illustrates explicitly how the weights of the attached cells get shifted.  One can see these shifts are a consequence of the splitting algorithm in the proof below.
\end{remark}

We now embark on proving Theorem~\ref{th:Kronholm} using our splitting algorithm. We again aim to induct on the dimension of a $\Rep(C_2)$-chain complex $X$.  However, it will be simpler to consider the attaching map for a single representation cell $V$.  We thus consider a map $f\colon \Sigma^{-1}V \to Y$ where $V$ is a single representation cell and by induction we assume $Y$ is split as a direct sum of representation cells.  An example is depicted in Figure~\ref{fig:attaching_single_cell}, where there is a single representation cell $\Sigma^{-1}V$ shown on the left with a (potentially complicated) attaching map $f$ to the direct sum of representation cells on the right.  Notice on the right-hand side, the cells are ordered by topological dimension, the degree of the top $F$.

\begin{figure}[ht]
  \begin{tikzcd}[column sep=tiny, row sep=scriptsize]
   F \arrow{d} & & & & & & & & F \arrow{d} & F \arrow{d} \\
   F \arrow{d} & \arrow{rrr}{f} & \phantom{F} &  & \phantom{F} & F \arrow{d} & F \arrow{d} & F \arrow{d} & F \arrow{d} & F \arrow{d} \\
   H & & & & & F \arrow{d} & H & H & F \arrow{d} & H\\
     & & & & & H & & & H &
  \end{tikzcd}
      \caption{Attaching a single representation cell.}
      \label{fig:attaching_single_cell}
  \end{figure}

If we take the cofiber of $f$, then we get a complicated complex to which we can apply the splitting algorithm from Section~\ref{se:complex-classify-proof}.  According to the algorithm, the cofiber will split into strands, but not necessarily the same vertical strands we started with.  One can picture various copies of $F$ and $H$ breaking and reattaching to each other.  However, at the end of the process, there are the same number of summands of $F$ and $H$.  

Since we have assumed the attaching map is spacelike, there will be no $B$-type strands in the final decomposition.  Our aim is to show that no strands of the form $A_k$ or $H(-n)$ appear, so there are only strands of the form $H(n)$ for $n \geq 0$.  For a nontrivial attaching map, we will find either several disks split off or some number of copies of $F$ from $V$ will transfer to other strands.  The latter will decrease the weight of the newly attached cell and increase the weights of the others.

In the course of this, we will need to
consider maps between representation cells of particular dimensions.
Up to homotopy, we may choose nice representatives for the attaching
map on each summand.  In the following lemma, we determine the
possible maps for our setting. From now on we draw chain complexes
horizontally.

\begin{lemma}\label{le:maps}
Let $V = \Sigma^{a}H(b)$ and $W = \Sigma^s H(t)$ be representation cells with $a \geq s$, and if $a = s$ then $b \geq t$.  Then for any map $\Sigma^{-1}V \to W$, exactly one of the following holds:
\begin{enumerate}
\item the cells satisfy $\tp(\Sigma^{-1}V)=\tp(W)$ and
$\cowt(\Sigma^{-1}V)=\cowt(W)$, and the map is homotopic to the identity,
\begin{center}
  \begin{tikzcd}[column sep=scriptsize, row sep=scriptsize]
    \Sigma^{-1}V \arrow{d} & & F \arrow{r} \arrow{d}{1} & F \arrow{r} \arrow{d}{1} & \cdots \arrow{r} & F \arrow{r} \arrow{d}{1} & H \arrow{d}{1} & \phantom{H} & \phantom{H}\\
    W & & F \arrow{r} & F \arrow{r} & \cdots \arrow{r} & F \arrow{r} & H & & \semicolon
  \end{tikzcd}
\end{center}
\item the cells satisfy $\tp(\Sigma^{-1}V) = \tp(W)$ and
$\cowt(\Sigma^{-1}V) > \cowt(W)$, and the map is homotopic to one of the form
\begin{center}
  \begin{tikzcd}[column sep=scriptsize, row sep=scriptsize]
     \Sigma^{-1}V \arrow{d} & & F \arrow{r} \arrow{d}{1} & F \arrow{r} \arrow{d}{1} & F \arrow{r} \arrow{d}{1} & F \arrow{r} \arrow{d}{1} & H \arrow{d}{p} & & \\
    W & & F \arrow{r} & F \arrow{r} & F \arrow{r} & F \arrow{r} & F \arrow{r} & F \arrow{r} & H \semicolon
  \end{tikzcd}
\end{center}
\item the cells satisfy $\tp(\Sigma^{-1}V) \geq \tp(W)$ and
$\cowt(\Sigma^{-1}V) \leq \cowt(W)-2$, and the map is homotopic to one of the form
\begin{center}
  \begin{tikzcd}[column sep=scriptsize, row sep=scriptsize]
     \Sigma^{-1}V \arrow{d} & & F \arrow{r} & F \arrow{r} \arrow{d}{u} & F \arrow{r} \arrow{d}{0} & F \arrow{r} \arrow{d}{0} & F \arrow{r} \arrow{d}{0} & \arrow{r} F & H \\
    W & & & F \arrow{r} & F \arrow{r} & F \arrow{r} & H & & \semicolon
  \end{tikzcd}
\end{center}
\item the map is null-homotopic.
\end{enumerate}
\end{lemma}

\begin{proof}
To justify these are the only cases, we start by computing homotopy classes of maps $\cD(\undt)(\Sigma^{-1}V, W)$.  Using that $H(b)$ is invertible with inverse $H(-b)$, we find that
\begin{align*}
  \cD(\undt)(\Sigma^{-1}V, W) &\cong \cD(\undt)(\Sigma^{a-1}H(b),\Sigma^{s}H(t))\\
  &\cong \cD(\undt)(\Sigma^{a-1-s}H(b),H(t))\\
  &\cong \cD(\undt)(\Sigma^{a-1-s}H,H(t-b))\\
  &\cong \begin{cases}
    \Z/2 & \text{if } t-b \geq 0 \text{ and } -(t-b) \leq a-1-s \leq 0 \\
    \Z/2 & \text{if } t-b \leq -2 \text{ and } 0 \leq a-1-s \leq -(t-b)-2\\
    0 & \text{otherwise.}
  \end{cases}
\end{align*}
The last isomorphism follows from cases \textit{(b)--(d)} of
Proposition~\ref{pr:homotopy-compute}.  That the only nonzero value is $\Z/2$ means there is at most one nontrivial homotopy class of maps between representation cells.  One can readily check the three maps described in parts \textit{(1)--(3)} of the lemma are non-null.  It remains to show these are the only possibilities that satisfy the given constraints on $V$ and $W$. We have two cases for non-null maps: $t-b\geq 0$ or $t-b\leq -2$.

Observe in the context of representation cells $V$ and $W$, the value $t-b$ measures how many more copies of $F$ the cell $W$ has compared to $V$, i.e.\ how much longer the second strand is than the first. We now consider the two cases.\medskip

\noindent \emph{Case 1 $(t-b\geq 0)$}: When $t-b \geq 0$, the strand $V$ is the same length or shorter than $W$.  Under these circumstances, we show the constraints placed on the dimensions of the cells lead to either the identity map or the map in part \textit{(2)}.  

To get a non-null map, we must also have $a-1-s \leq 0$ or equivalently $a-1 \leq s$.  This means the top $F$ in $W$ is in the same degree or higher than (i.e.\ to the left of) the top $F$ in $\Sigma^{-1}V$.  By hypothesis, we also have $a \geq s$ or equivalently $a-1-s \geq -1$.  We see there are only two possibilities: $a-1-s = 0$ and the strands begin at the same place, or $a-1-s = -1$ and there is a single $F$ in $W$ above the start of $\Sigma^{-1}V$.

Suppose $a-1-s=0$, so the strands begin in the same degree.  If the strands are the same length, we have an easy choice of non-null map: the identity.  Otherwise, $W$ is longer than $\Sigma^{-1}V$.  In that case, we may assume the map is of the form in part \textit{(2)} of the lemma, having identity maps between copies of $F$ and $p \colon H \to F$ in the degree of $\cowt(V)$, as below:
\begin{center}
  \begin{tikzcd}[column sep=scriptsize, row sep=scriptsize]
    \Sigma^{-1}V \arrow{d} & & F \arrow{r} \arrow{d}{1} & F \arrow{r} \arrow{d}{1} & \cdots \arrow{r} & F \arrow{r} \arrow{d}{1} & H \arrow{d}{p} & & \\
    W & & F \arrow{r} & F \arrow{r} & \cdots \arrow{r} & F \arrow{r} & F \arrow{r} & F \arrow{r} & H.
  \end{tikzcd}
\end{center}
This map is not null and thus represents the nonzero homotopy class.  Notice here $\tp(\Sigma^{-1}V)=\tp(W)$ and the restrictions on the lengths of the strands mean $\cowt(\Sigma^{-1}V) > \cowt(W)$, as desired.

Lastly, consider when $a-1-s = -1$. We show there are no maps that
satisfy the hypotheses of the lemma. In this case we have $a=s$, for
which the lemma specifies that we must have $b \geq t$.  On the other
hand, $t \geq b$ in this case.  Thus $t = b$, but this contradicts
that $-(t-b) \leq a-1-s$ since $0$ is not less than $-1$.
\medskip

\noindent\emph{Case 2 $(t-b\leq -2)$}: This inequality implies
$\Sigma^{-1}V$ has at least two more nonzero terms than
$W$. Rearranging the two inequalities from the homotopy class
calculation, $0 \leq a-1-s$ and $a-1-s \leq -(t-b)-2$, we immediately
find $s \leq a-1$ and $(a-1)-b \leq s-t-2$.  That is,
$\tp(W)\leq \tp(\Sigma^{-1}V)$ 
and $\cowt(\Sigma^{-1}V) \leq \cowt(W)-2$.  A map of the form in part \textit{(3)} of the lemma satisfies these constraints and is not null, and thus is a representative of the unique nontrivial homotopy class.

\end{proof}

\begin{remark}
Observe that each of the non-null maps in Lemma~\ref{le:maps} has a
distinct lowest degree (right-most) nonzero component: either $\id$, $p$, or
$u$, corresponding to the cases (1)--(3).    This will be important in
our application below.
\end{remark}

We are now ready to prove Kronholm's theorem in the chain
complex setting using the splitting algorithm.

\begin{proof}[Proof of Theorem~\ref{th:Kronholm} via splitting]

We induct on the number of representation cells in the
$\Rep(C_2)$-chain complex $X$.  To build $X$, instead of attaching all
the $m$-cells at once, attach representation cells one at a time in
order of increasing topological dimension, and within each topological
dimension in order of increasing weight.  We attach the cells in order of
increasing weight to have better control of the attaching maps.  This will allow us to reduce to studying the attaching maps from Lemma~\ref{le:maps}.

In order to maintain control of the attaching maps, our induction will prove a slightly stronger result than in the
statement of the theorem.  We prove that if $X$ is obtained by
attaching an $(m,q)$-representation cell $V$ to a $\Rep(C_2)$-chain
complex $Y$ where $Y\he \oplus_i \Sigma^{k_i}H(r_i)$ with $0\leq
r_i\leq k_i$ {\bf and also satisfying} 
\begin{equation} 
\tag{**}
\text{$k_i\leq m$, and if $k_i=m$ then $r_i\leq q$}
\end{equation}
for all $i$, then $X$ is quasi-isomorphic to a direct sum of
representation cells satisfying the same inequalities as in (**).  In particular, any cells of dimension $m$ have weight no more than $q$. 
The base case where $Y$ has no cells is trivial.  

We need to analyze the attaching map $f\colon \Sigma^{m-1}H(q) \to Y$.  Since
$Y$ is quasi-isomorphic to $\oplus_i \Sigma^{k_i}H(r_i)$ and $Y$ is
cofibrant, while $\oplus_i \Sigma^{k_i}H(r_i)$ is fibrant, in the
projective model structure on $\Ch(\undt\MMod)$, there must be a
quasi-isomorphism $g\colon Y \to \oplus_i \Sigma^{k_i}H(r_i)$.  So we
can instead consider the attaching map $gf$
\[
\Sigma^{m-1}H(q) \to \oplus_i \Sigma^{k_i}H(r_i),  
\]
whose cofiber will still be quasi-isomorphic to $X$.  Thus, we can just
replace $Y$ with the direct sum in the codomain.
For convenience, let $W_i$ denote the summand
$\Sigma^{k_i}H(r_i)$, so $X$ is quasi-isomorphic to the cofiber of
$\Sigma^{-1}V \to \oplus_i W_i$.  We will apply the splitting
algorithm from Section~\ref{se:complex-classify-proof} to the cofiber.
Recall that the algorithm proceeds from lowest degree nonzero term to
highest, or from right to left when we draw complexes horizontally.

If the map $\Sigma^{-1}V \to \oplus_i W_i$ is null-homotopic, then the
cofiber immediately splits as the direct sum $V \oplus \left(\oplus_i
W_i\right)$.  The condition (**) on the cells is immediate.

The more interesting situation arises from a non-null map.  Projecting
onto each summand, we get maps of the form $\Sigma^{-1}V \to W_i$.
Since the cells $W_i$ satisfy condition (**) by assumption, we can apply Lemma~\ref{le:maps} letting $W=W_i$.  Thus there are only
three non-null possibilities (up to homotopy) for the map $\Sigma^{-1}V \to W_i$; as in
Lemma~\ref{le:maps} we label these possibilities (1)--(3).  Choose
these nice representatives for each non-null map and the zero map for
any null portion.  As the splitting algorithm proceeds from bottom to
top, we will want to consider the lowest degrees first. Recall the
lowest degree nonzero map for each representative is: $\id \colon H
\to H$ for case \textit{(1)}, $p \colon H
\to F$ in case \textit{(2)},  and  $u \colon F \to F$
in case \textit{(3)}.

Now consider $\Cof(\Sigma^{-1}V \to \oplus_i W_i)$ and apply the
splitting algorithm.  There are no maps between the $W_i$, so reading
from right to left the complex is split into strands until degree
$\cowt(V)$.  In degree $\cowt(V)$, the bottom $H$ from $V$ supports
either the identity map, $p$, or the zero map into the various strands.  If that copy of $H$
supports a nonzero map, we begin to apply the splitting algorithm
here.

The splitting algorithm considers isomorphisms first and uses them to split off disks.  So if the bottom $H$ in $V$ supports the identity map to a copy of $H$ in any of the summands, then for at least one $i$ there is a map of the form $\id: V \to W_i$.  An example of the cofiber (omitting other strands for brevity) is pictured below:
\begin{center}
  \begin{tikzcd}[column sep=scriptsize, row sep=scriptsize]
     V & F \arrow{r}{u} \arrow{dr}{1} & F \arrow{r}{u} \arrow{dr}{1} & \cdots \arrow{r}{u} & F \arrow{r}{p} \arrow{dr}{1} & H \arrow{dr}{1}  \\
    W_i & & F \arrow{r}[swap]{u} & F \arrow{r}[swap]{u} & \cdots \arrow{r}[swap]{u} & F \arrow{r}[swap]{p} & H.
  \end{tikzcd}
\end{center}
Regardless of whether the terms from $V$ support other nonzero maps,
in each degree the algorithm prioritizes the identity maps between
terms and we can choose to use the identity maps to $W_i$.  The
algorithm will split off many disks, as many as the length of $V$.
After the change of bases, no terms from $V$ will support nonzero
maps, so the remainder of the complex will be split as it was in $Y$.
Thus $X$ will be quasi-isomorphic to $Y$ but with one strand ($W_i$)
removed.  As before, none of the remaining strands have changed so
the inequalities from (**)  still hold.\medskip

Now suppose there are no identity maps supported by the bottom $H$ from $V$.  If that $H$ supports a nonzero map, it must be $p\colon H \to F$.  We will ultimately see this is not possible.  If it were, there would be at least one summand $W_i$ with a map $V \to W_i$ of the form in part \textit{(2)} of Lemma~\ref{le:maps}.  An example of the cofiber (again omitting other strands) is pictured below:
\begin{center}
  \begin{tikzcd}[column sep=scriptsize, row sep=scriptsize]
    V & F \arrow{r}{u} \arrow{dr}{1} & F \arrow{r}{u} \arrow{dr}{1} & F \arrow{r}{u} \arrow{dr}{1} & F \arrow{r}{u} \arrow{dr}{1} & H \arrow{dr}{p} & & \\
    W_i & & F \arrow{r}[swap]{u} & F \arrow{r}[swap]{u} & F \arrow{r}[swap]{u} & F \arrow{r}[swap]{u} & F \arrow{r}[swap]{u} & F \arrow{r}[swap]{p} & H.
  \end{tikzcd}
\end{center}
In the degree of $\cowt(V)$, the algorithm prioritizes the map $p \colon H \to F$.  Applying the algorithm would split off a $B$-type summand ($B_1$ in the example pictured).  By assumption, all the attaching maps in our complex are spacelike, contradicting that there are any $B$-type summands.  Thus the bottom $H$ from $V$ cannot support any nonzero maps to other summands other than the identity, and hence there are no maps of the form in part \textit{(2)} of Lemma~\ref{le:maps}. \medskip

Having dealt with these possibilities, we may now assume the bottom $H$ from $V$ does not support any nonzero maps, and any remaining non-null maps supported by this strand are of the form in part \textit{(3)} of Lemma~\ref{le:maps}.  Reading from bottom to top (i.e.\ right to left) in the cofiber, the strands will be split until we encounter a copy of $F$ from $V$ supporting the map $u \colon F \to F$.  Here Step 4 of the algorithm splits the strands by attaching the $F$ from $V$ to the shortest strand ending in $H$.

An example of the cofiber of such a map $\Sigma^{-1}V \to W_i$ (omitting other strands as usual) is depicted below:
\begin{center}
  \begin{tikzcd}[column sep=scriptsize, row sep=scriptsize]
    V & F \arrow{r} & F \arrow{r} \arrow{dr}[swap]{u} & F \arrow{r}  & F \arrow{r}  & F \arrow{r} & \arrow{r} F & H \\
    W_i & & & F \arrow{r} & F \arrow{r} & F \arrow{r} & H & \period
  \end{tikzcd}
\end{center}
After the splitting, the resulting strands for this example are
\begin{center}
  \begin{tikzcd}[column sep=scriptsize, row sep=scriptsize]
    F \arrow{r} & F \arrow{dr} & F \arrow{r}  & F \arrow{r}  & F \arrow{r} & \arrow{r} F & H \\
    & & F \arrow{r} & F \arrow{r} & F \arrow{r} & H & \period
  \end{tikzcd}
\end{center}
Observe that after the splitting, there are again two representation cells with topological dimensions $\tp(V)$ and $\tp(W_i)$.  Moreover, attaching the $F$ to the second strand changes the weights.  The weight of the first strand decreases by the same amount the weight of the second strand increases.

Of course, there may be many non-null maps of the form in part \textit{(3)}.  Continue reading from right to left until all strands have been split according to the algorithm.  The strand $V$ is finite so this process eventually terminates. \medskip

We give an example to illustrate these last steps in more detail.  In the example, we only show the shortest strands at each stage, since the algorithm will effectively ignore any longer ones.  Alternatively, using a diagonal change of basis for the cofiber, one may assume the copies of $F$ from $V$ support a nonzero map to at most one strand of each topological dimension. The example is shown below:
\begin{center}
  \begin{tikzcd}[column sep=scriptsize, row sep=scriptsize]
    V  & F \arrow{r} \arrow{dddr}[swap]{u} & F \arrow{r} \arrow{ddr}[swap]{u} & F \arrow{r} \arrow{dr}[swap]{u} & F \arrow{r}  & F \arrow{r}  & F \arrow{r} & \arrow{r} F & H \\
    W_{i_1} & & & & F \arrow{r} & F \arrow{r} & F \arrow{r} & H &\\
    W_{i_2} & & & F \arrow{r} & F \arrow{r} & H & & & \\
    W_{i_3} & & F \arrow{r} & F \arrow{r} & F \arrow{r} & F \arrow{r} & F \arrow{r} & H & \period
  \end{tikzcd}
\end{center}
At the first stage, the fifth copy of $F$ from $V$ (reading right to left) moves to the second strand, as in the previous example:
\begin{center}
    \begin{tikzcd}[column sep=scriptsize, row sep=scriptsize]
      \phantom{V} & F \arrow{r} \arrow{dddr}[swap]{u} & F \arrow{r} \arrow{ddr}[swap]{u} & F  \arrow{dr} & F \arrow{r}  & F \arrow{r}  & F \arrow{r} & \arrow{r} F & H \\
      \phantom{W_{i_1}} & & & & F \arrow{r} & F \arrow{r} & F \arrow{r} & H &\\
      \phantom{W_{i_2}} & & & F \arrow{r} & F \arrow{r} & H & & & \\
      \phantom{W_{i_3}} & & F \arrow{r} & F \arrow{r} & F \arrow{r} & F \arrow{r} & F \arrow{r} & H & \period
    \end{tikzcd}
  \end{center}
Then the algorithm attaches the next $F$ from $V$ onto the shortest strand, which is now the third strand pictured:
\begin{center}
    \begin{tikzcd}[column sep=scriptsize, row sep=scriptsize]
      \phantom{V} & F \arrow{r} \arrow{dddr}[swap]{u} & F  \arrow{ddr} & F  \arrow{dr} & F \arrow{r}  & F \arrow{r}  & F \arrow{r} & \arrow{r} F & H \\
      \phantom{W_{i_1}} & & & & F \arrow{r} & F \arrow{r} & F \arrow{r} & H &\\
      \phantom{W_{i_2}} & & & F \arrow{r} & F \arrow{r} & H & & & \\
      \phantom{W_{i_3}} & & F \arrow{r} & F \arrow{r} & F \arrow{r} & F \arrow{r} & F \arrow{r} & H & \period
    \end{tikzcd}
  \end{center}
Finally, the remaining $F$ from $V$ attaches to the third strand since it is shorter than the fourth:
\begin{center}
    \begin{tikzcd}[column sep=scriptsize, row sep=scriptsize]
      \phantom{V} & F \arrow{r} & F  \arrow{ddr} & F  \arrow{dr} & F \arrow{r}  & F \arrow{r}  & F \arrow{r} & \arrow{r} F & H \\
      \phantom{W_{i_1}} & & & & F \arrow{r} & F \arrow{r} & F \arrow{r} & H &\\
      \phantom{W_{i_2}} & & & F \arrow{r} & F \arrow{r} & H & & & \\
      \phantom{W_{i_3}} & & F \arrow{r} & F \arrow{r} & F \arrow{r} & F \arrow{r} & F \arrow{r} & H & \period
    \end{tikzcd}
  \end{center}

At the end of this process we have split the cofiber into a number of strands, though they may have different lengths than the original strands.  In any case, all of the strands are of the form $\Sigma^k H(r)$ for various $k$ and $r$ satisfying $0 \leq r \leq k$.  Thus $X$ is quasi-isomorphic to a direct sum of representation cells as desired.

It only remains to verify that the inequalities from (**) are satisfied by $X$.  In fact, since the copies of $F$ from $V$ have simply
detached and reattached, the collection of topological dimensions of
the strands is preserved (as well as the total weight).  So all the
topological dimensions are still less than or equal to $m$.  Note that
by Lemma~\ref{le:maps}, $\Sigma^{-1}V$ does not admit a non-null map to any of the
summands in $Y$ having topological dimension $m$; so these cells are
unaffected in the cofiber $X$ and still have weights less than or
equal to $q$ by (**).  The only other cell in $X$ of topological dimension $m$
is the one onto which the 
top $F$ from $V$ gets attached.  According to Step 4 of the algorithm, which
at each iteration chooses the shortest strand ending in $H$ to split off, this top $F$ from $V$
must have attached to a strand no longer than the original $V$.  So
that strand contributes a representation cell with weight at most $q$,
and thus (**) is still satisfied.  This completes the
induction.
\end{proof}

\newpage


\bibliographystyle{amsalpha}

\begin{thebibliography}{JTTW}

\bibitem[B]{B} P. Balmer, {\em The spectrum of prime ideals in
  tensor triangulated categories\/}, J. Reine Agnew. Math. {\bf 588}
(2005), 149--168.

\bibitem[BG]{BG} P. Balmer and M. Gallauer, {\em Three Real
  Artin-Tate motives\/}, Adv. Math. {\bf 406} (2022), 108535, 68pp.

\bibitem[BN]{BN} M. B\"okstedt and A. Neeman, {\em Homotopy limits in
  triangulated categories\/},
Compositio Math. {\bf 86} (1993), 209--234.


\bibitem[Bc]{Bc} S. Bouc, {\em Complexity and cohomology of
  cohomological Mackey functors\/}, Adv. Math. {\bf 221} (2009),
no. 3, 983--1045.


\bibitem[BSW]{BSW} S. Bouc, R. Stancu, and P. Webb, {\em On the
  projective dimension of Mackey functors\/},
Algebr. Represent. Theory {\bf 20} (2017), no. 6, 1467--1481.


\bibitem[Bo]{Bo} N. Bourbaki, {\em \'El\'ements de Math\'ematique:
  Th\'eorie des Ensembles}, Springer, 1970.

\bibitem[D1]{D1} A.~W.~M. Dress, {\em Notes on the theory of
representations of finite groups\/}, mimeographed notes, Bielefeld,
1971.


\bibitem[D2]{D2} A.~W.~M. Dress, {\em Contributions to the theory of
induced representations\/}, in Algebraic $K$-theory II,  Lecture Notes
in Math. {\bf 342}, pp. 183--240, Springer-Verlag, 1973.  

\bibitem[Gr]{Gr} J.~A. Green, {\em Axiomatic representation theory for
finite groups\/}, J. Pure Appl. Algebra {\bf 1} (1971), 41--77.

\bibitem[G]{G} J.~P.~C. Greenlees, {\em Some remarks on projective
  Mackey functors\/}, J. Pure and Appl. Alg. {\bf 81} (1992), 17--38.


\bibitem[HM]{HM} E. Hogle and C. May, {\em The freeness theorem for equivariant cohomology
of Rep($C_2$)-complexes}, Topology Appl. {\bf 285} (2020), 107413, 30pp.

\bibitem[K]{K} W. Kronholm, {\em A freeness theorem for $RO(\Z/2)$-graded
  cohomology\/}, Topology and its Applications {\bf 157} (2010),
902--915.

\bibitem[M]{M} C. May, {\em A structure theorem for $RO(C_2)$-graded
  Bredon cohomology\/}, Alg. Geom. Topol. {\bf 20} (2020), no. 4, 1691--1728.

\bibitem[R]{R} D.~N. Raies, {\em Mackey functors over the group $\Z/2$
and computations in homological algebra\/}, Ph.~D.~ thesis, University of
Oregon, 2019.

\bibitem[SS]{SS} S. Schwede and B. Shipley, {\em Equivalences of monoidal model categories\/}, Algebr. Geom. Topol. {\bf 3} (2003), no.~1, 287--334.

\bibitem[W]{W} P. Webb, {\em A guide to Mackey functors\/}, Handbook
of Algebra, Vol. 2, 805--836.  Elsevier/North Holland, Amsterdam,
2000.

\bibitem[Wb]{Wb} C.~A. Weibel, {\em An introduction to homological
algebra\/}, Cambridge Studies in Advanced Mathematics {\bf 38}, Cambridge University Press, Cambridge, 1994.

\bibitem[Z]{Z} M. Zeng, {\em Equivariant Eilenberg--MacLane spectra in
cyclic $p$-groups\/}, 2018 preprint, {\tt arXiv:1710.01769}.


\end{thebibliography}

\end{document}